\documentclass[12pt]{amsart}
\usepackage{amsfonts, amssymb, latexsym, hyperref, xcolor, tikz,tikz-cd}
\usepackage{pstricks, pst-plot, pst-node}
\usepackage{pst-func}
\usepackage{xcolor}
\usepackage{ytableau, stackrel, exercise}

\newtheorem{theorem}{Theorem}[section]
\newtheorem{proposition}[theorem]{Proposition}
\newtheorem{lemma}[theorem]{Lemma}

\newtheorem*{claim*}{Claim}
\newtheorem{corollary}[theorem]{Corollary}

\newtheorem{Main Conjecture}[theorem]{Main Conjecture}
\newtheorem{conjecture}[theorem]{Conjecture}

\newtheorem{problem}[theorem]{Problem}
\theoremstyle{definition}
\newtheorem{definition}[theorem]{Definition}
\theoremstyle{remark}
\newtheorem{exercise}[theorem]{Exercise}
\newtheorem{example}[theorem]{Example}

\theoremstyle{plain}

\newenvironment{hardexercise}{\refstepcounter{theorem}\paragraph{\emph{Exercise} \thetheorem*}}{}

\newcommand\Groth{{\mathfrak G}}

\newcommand{\cellsize}{12}
\newlength{\cellsz} 
\newsavebox{\cell}
\sbox{\cell}{\begin{picture}(\cellsize,\cellsize)
\put(0,0){\line(1,0){\cellsize}}
\put(0,0){\line(0,1){\cellsize}}
\put(\cellsize,0){\line(0,1){\cellsize}}
\put(0,\cellsize){\line(1,0){\cellsize}}
\end{picture}}
\newcommand\cellify[1]{\def\thearg{#1}\def\nothing{}%
\ifx\thearg\nothing
\vrule width0pt height\cellsz depth0pt\else
\hbox to 0pt{\usebox{\cell} \hss}\fi%
\vbox to \cellsz{
\vss
\hbox to \cellsz{\hss$#1$\hss}
\vss}}
\newcommand\tableau[1]{\vtop{\let\\\cr
\baselineskip -16000pt \lineskiplimit 16000pt \lineskip 0pt
\ialign{&\cellify{##}\cr#1\crcr}}}

\newcommand{\gap}{\hspace{1in} \\ \vspace{-.2in}}

\newcommand{\excise}[1]{}

\begin{document}
\pagestyle{plain}
\title{Schubert geometry and combinatorics}
\author{Alexander Woo}
\address{Dept.~of Mathematics, U.~Idaho, Moscow, ID 83844, USA} 
\email{awoo@uidaho.edu}
\author{Alexander Yong}
\address{Dept.~of Mathematics, U.~Illinois at Urbana-Champaign, Urbana, IL 61801, USA} 
\email{ayong@illinois.edu}
\date{March 2, 2023}
\maketitle

\begin{abstract}
This chapter combines an introduction and research survey about Schubert varieties. The theme
is to combinatorially classify their singularities using a family of polynomial ideals generated by determinants.
\end{abstract}

\setcounter{tocdepth}{1}
\tableofcontents
\vspace{-.4in}

\section{Introduction}

\subsection{Objectives} 
The purpose of this chapter is to 
provide both a basic introduction and a research survey on Schubert varieties. The organizing theme is to examine their singularities through the lens of certain polynomial ideals generated by determinants. Hence we focus on properties that can be studied using methods from commutative algebra.  With apologies upfront, we do not report on even a significant proportion of the many other methods that have been developed to study Schubert singularities\footnote{For example, see \emph{Frobenius splitting} \cite{Brion.Kumar}, \emph{Peterson translates} \cite{Carrell.Kuttler}, \emph{Standard monomial theory} \cite{SMT}, and \emph{Bott-Samelson resolutions} \cite{Bott.Samelson, Hansen, Demazure},
\emph{Billey-Postnikov decompositions} \cite{Richmond.Slofstra}, and more \cite{BL00}.}. This chapter is also not about Schubert \emph{calculus}, which is a largely separate subject.\footnote{See \cite{Fulton, Manivel} or recent surveys such as \cite{RYY, Knutson:survey}.} 

A cursory glance at the literature on Schubert varieties finds a broad (and perhaps daunting) range of ideas used: Lie theory, representation theory, algebraic geometry, commutative algebra, and combinatorics. Nevertheless, we
endeavor to keep the exposition as self-contained as possible. Therefore, where helpful, we highlight principles in lieu of formal statements.  
We include numerous exercises (with more difficult ones starred) aimed at preparing the 
interested reader from the ground up for the open problems.  We hope our focus on commutative algebra will allow readers to explore the subject, perhaps with the aid of computational commutative algebra systems, before having to learn more about the areas mentioned above.

\subsection{A brief history} 
Schubert varieties go back to H.~Schubert's work in the 19th century~\cite{Schubert} and further, but the ``modern era'' began in the 1950s with C.~Chevalley's manuscript \cite{Chevalley}.\footnote{In the forward to its posthumous publication,  A.~Borel places it ``most likely from 1958''.} 
It introduced the now omnipresent notion of \emph{Bruhat order}.\footnote{An instance of Stigler's law of eponymy.}
He also addresses the question of singularities of Schubert varieties\footnote{Specifically, he proves they appear in codimension at least two.}, although it is a historical curiosity that on the topic he makes the following remark:
\begin{quotation}
``\emph{...il para\^it probable que les $X(w)$ sont toujours des
vari\'et\'es non singuli\`eres, mais nous ne sommes pas parvenus \`a \'etablir ce point.}''\footnote{``... it seems probable that all Schubert varieties are smooth, although we cannot establish this claim.'' Exercise~\ref{exer:smallcounter} asks the reader to compute
small counterexamples for themselves.}
\end{quotation}
 
In this chapter, we are concerned with local properties.  A local property is some property that might or might not hold at a given point $p$ on an algebraic variety $X$ and is said to hold on $X$ if it holds at every point $p\in X$.  During Chevalley's lifetime, others proved that some local properties hold on every Schubert variety.   For example, C.~DeConcini--V.~Lakshmibai ~\cite{DeCon-Lak} and S.~Ramanan--A.~Ramanathan~\cite{RR, Ram85} showed that Schubert varieties have ``mild'' singularities: they are \emph{normal} and \emph{Cohen-Macaulay}. In contrast, this chapter is about local properties that hold only on (some points of) some Schubert varieties.
The prototype for questions we examine are:
\begin{itemize}
\item[(S1)] Which Schubert varieties are smooth?
\item[(S2)] Which points of a Schubert variety are smooth? 
\end{itemize}

Although (S1) is a special case of (S2), it is interesting in its own right.

We mostly concentrate on Schubert varieties in the complete flag variety. Essentially, this is the general case 
among all partial flag varieties associated to $GL_n({\mathbb C})$, including Grassmannians.  In this setting, 
(S1) and (S2) have been answered.  For (S1),
J.~Wolper \cite{Wolper} gave a combinatorial characterization and later
K.~Ryan \cite{Ryan} presented a geometric characterization. Following their work, 
V.~Lakshmibai--B.~Sandhya
\cite{Lak} offered a different combinatorial answer to (S1) in terms of
\emph{pattern
avoidance}. Their result states that
\begin{center}
$X_{w}$ is smooth if and only if $w$ avoids the patterns
$3412$ and $4231$.
\end{center}
In addition, they used pattern avoidance in their conjectural solution to (S2). This conjecture was 
independently proved by~\cite{billey.warrington, Cortez, KLR, manivel1} at the turn of the millenium.
It is the pattern avoidance approach that we follow in this chapter.  An earlier survey on Schubert varieties and
 pattern avoidance is \cite{Billey}.

Now that we know \emph{which} points of a Schubert variety are singular, one asks \emph{how} singular is a given point? As we explain, numerous measures/properties ${\mathcal P}$ with commutative algebra definitions can be studied simultaneously. Generalizing (S1) and (S2), one asks:
\begin{itemize}
\item[(P1)] Which Schubert varieties are globally ${\mathcal P}$? 
\item[(P2)] Which points of a Schubert variety satisfy ${\mathcal P}$? 
\end{itemize}

In \cite{WY:governing} we examined such properties ${\mathcal P}$ though determinantal
ideals that we called \emph{Kazhdan-Lusztig ideals}. By example we showed that classical pattern avoidance  is
an incomplete language to answer (P1) or (P2). To overcome this, we introduced \emph{interval pattern avoidance}
as a universal language to discuss these questions and to make comparisons. This chapter recounts those results and the subsequent developments to date.

\subsection{Organization}
In Section~\ref{sec:2} we start with the Grassmannian as a prequel. We illustrate how determinantal ideals
come about in the study of Schubert varieties.  Pl\"ucker coordinates, the Schubert cells and varieties, and the group
cosets description of the Grassmannian are introduced. Analogously, we explain the complete flag variety case.

\emph{Kazhdan-Lusztig ideals} are defined in Section~\ref{sec:3}. These determinantal ideals cut out affine open neighborhoods of Schubert varieties in the flag variety (up to an irrelevant affine factor). An important special case 
consists of W.~Fulton's \emph{Schubert determinantal ideals} \cite{Fulton:duke}, which have been an object of significant interest. The main point is that studying singularity properties of Schubert
varieties reduces to studying the same for Kazhdan-Lusztig varieties.

Section~\ref{sec:5} defines \emph{interval pattern avoidance}. 
Theorem~\ref{thm:poset} shows that under mild assumptions about the property ${\mathcal P}$,
answers to (P1) and  (P2) \emph{can} be given
in terms of \emph{interval pattern avoidance} (although exact answers may not be known at present). 
This universal language provides a concrete way to compare and contrast properties ${\mathcal P}$. Universality is proved using Kazhdan-Lusztig ideals.

Section~\ref{sec:4} delves into the combinatorial commutative algebra aspects of Kazhdan-Lusztig ideals, following A.~Knutson-E.~Miller's ideas \cite{Knutson.Miller:annals} about Schubert determinantal ideals.  Through a running example, we
review the concepts of Gr\"obner bases, 
multigradings, Hilbert series, prime decompositions, and the Stanley-Reisner correspondence. Exercises give
the reader a hands-on view of the application of these ideas to determinantal ideals. We then point to the
generalizations to Kazhdan-Lusztig
ideals from \cite{WY:Grobner}, although we leave most of the elaboration out of this text.

Section~\ref{sec:4.3} concerns \emph{free resolutions} of modules over a polynomial ring. Our main purpose is to
give enough detail, using relevant examples, to concretely define the commutative algebra concepts in the case of
Kazhdan-Lusztig ideals. This is needed for the next section.

Section~\ref{sec:6} is this chapter's summit. It analyzes  seven properties/measures ${\mathcal P}$ and studies them in the case of Schubert varieties: smoothness, being a local compete intersection, being Gorenstein, factoriality, Hilbert-Samuel multiplicity, Castelnuovo-Mumford regularity of the tangent cone, and Kazhdan--Lusztig polynomials. We give known results, conjectures,
and open problems in each case.

The remaining three sections offer additional notes.
Section~\ref{sec:7} indicates analogues of the questions we consider for other Lie types; it  
compiles information that seems to not appear in any one place in the literature.
 Section~\ref{sec:8} discusses analogues for other varieties. Finally,
Section~\ref{sec:hints} provides hints and references for selected exercises. 

\section{Schubert basics}\label{sec:2}

\subsection{Grassmannians}\label{sec:Grasscase} We begin with a set-theoretic definition:
\begin{definition} 
The \emph{Grassmannian} $Gr_k({\mathbb C}^n)$ is 
the parameter space\footnote{A synonym of ``parameter space'' 
is ``moduli space''. However the latter is typically used in cases where the points of the space are abstract
objects (such as isomorphism classes of curves) rather than objects embedded in a particular ambient space, as we have here with
subspaces of ${\mathbb C}^n$.} of $k$-dimensional
planes in ${\mathbb C}^n$.
\end{definition}

\subsubsection{A covering by affine spaces}\label{subsec:2.1.1} Suppose $\vec v_1,\vec v_2,\ldots, \vec v_k$ is a basis of a $k$-dimensional plane $V$.
View these vectors as columns of a $(n\times k)$-dimensional matrix $M$.  \emph{Column operations} (the analogues of row operations, except performed on columns) on $M$ are changes of basis and therefore column equivalent matrices represent the same
$k$-plane $V$.  In other words, one identifies the Grassmannian with a group quotient:
\begin{equation}
\label{eqn:Grdesc}
Gr_k({\mathbb C}^n)=\left({\sf Mat}_{n\times k}-\{M\in {\sf Mat}_{n\times k} : {\rm rank}(M)<k\}\right)/GL_k,
\end{equation}
where ${\sf Mat}_{n\times k}$ is the set of $n\times k$ dimensional matrices and the quotient is by $GL_k$ acting on the right
by matrix multiplication.

When first studying linear algebra, column reduction is done to make rows $1,2,\ldots,k$ the pivot rows if possible.
However, there is no particular reason besides convenience for this ordering.\footnote{While column operations and column echelon form are probably less familiar to the reader than row operations and row echelon form, it turns out that the group cosets description of the Grassmannian is compatible with these conventions.}  Rather, let ${\mathcal S}_{k,n}$ be the set of $k$-subsets of $[n]:=\{1,2,\ldots,n\}$, and
given any 
\[I:=\{i_1<i_2<\ldots<i_k\}\in {\mathcal S}_{k,n},\] 
we can think of the rows in $I$ as the ``first'' rows and attempt to column reduce $M$ so that the rows in $I$ are pivot rows and the remaining rows are free rows.  

Now, let $U_I$ be the set of all $n\times k$ matrices $[x^I_{ab}]$ such that $x^I_{i_c,c}=1$ for all $c$ with $1\leq c\leq k$ and $x^I_{i_c,d}=0$ for all $c\neq d$ with $1\leq c,d\leq k$.  Each $U_I$ can be canonically identified with $\mathbb{C}^{k(n-k)}$ with coordinates $x^I_{ab}$ for $a\not\in I$, $1\leq b\leq k$.  The uniqueness of reduced column echelon form tells us the following:

\begin{theorem}
Given any $I\in \mathcal{S}_{k,n}$, each matrix in $U_I$ represents a distinct $k$-plane $V$.  Hence, we can identify $U_I$ with a subset of $Gr_k({\mathbb C}^n)$.
\end{theorem}

Since for any $k$-plane $V$ and any choice of basis for $V$, the matrix $M$ has full rank, there is some set of rows $I$ such that $M$ can be column reduced with pivots in $I$.  Hence the sets $U_I$ cover $Gr_k({\mathbb C}^n)$. Each $U_I$ is a \emph{chart}, and together they form an \emph{atlas} for $Gr_k(\mathbb{C}^n)$ considered as a manifold.

\begin{example}\label{exa:U24}
For $Gr_2({\mathbb C}^4)$, the atlas consists of:
$$
U_{34}=\begin{bmatrix} \ast & \ast \\ \ast & \ast \\ 1 & 0 \\ 0 & 1\end{bmatrix},
U_{24}=\begin{bmatrix} \ast & \ast \\ 1 & 0 \\ \ast & \ast \\ 0 & 1\end{bmatrix},
U_{14}=\begin{bmatrix} 1 & 0 \\ \ast & \ast \\ \ast & \ast \\ 0 & 1\end{bmatrix},
U_{23}=\begin{bmatrix} \ast & \ast \\ 1 & 0 \\ 0 & 1 \\ \ast & \ast\end{bmatrix},
U_{13}=\begin{bmatrix} 1 & 0 \\ \ast & \ast \\ 0 & 1 \\ \ast & \ast \end{bmatrix},
U_{12}=\begin{bmatrix} 1 & 0 \\ 0 & 1 \\ \ast & \ast \\ \ast & \ast \end{bmatrix}.
$$
\end{example}

\subsubsection{Transition maps; Grassmannians are smooth complex manifolds and algebraic varieties}
\label{subsubsec:transgr}
Suppose $V\in U_I\cap U_J$.  Then $V$ has coordinates $x^I_{ab}(V)$ when considered as a point in $U_I$ and coordinates $x^J_{ab}(V)$ when considered as a point in $U_J$.  By taking a generic matrix in $U_I$ and column reducing it so that the rows in $J$ are pivots, we see that the coordinates $x^J_{ab}$ are rational functions of the coordinates $x^I_{ab}$.  Since rational functions are smooth, this gives $Gr_k({\mathbb C}^n)$ the structure of a smooth manifold, and furthermore a complex manifold and a smooth complex algebraic variety.  We give $Gr_k({\mathbb C}^n)$ the complex topology by declaring $W\subseteq Gr_k({\mathbb C}^n)$ to be open if $W\cap U_I$ is open, for all $I$, as a subset of $U_I\cong\mathbb{C}^{k(n-k)}$, or we give $Gr_k({\mathbb C}^n)$ the Zariski topology by declaring $W$ to be open if $W\cap U_I$ is open, for all $I$, in the Zariski topology.

\begin{example}[$Gr_2({\mathbb C}^4)$] Suppose $I=\{1,2\}$ and $J=\{2,4\}$.
Then the generic matrix in $U_I$ is given by
$$U_I=\begin{bmatrix}1 & 0 \\ 0 & 1 \\ x^I_{31} & x^I_{32} \\ x^I_{41} & x^I_{42}\end{bmatrix}.$$
Doing column operations to put pivots in rows $2$ and $4$ gives
$$U_I=\begin{bmatrix}-x^I_{42}/x^I_{41} & 1/x^I_{41}\\ 1& 0 \\ (x^I_{32}x^I_{41}-x^I_{42}x^I_{31})/x^I_{41} & x^I_{31}/x^I_{41} \\ 0 & 1 \end{bmatrix}.$$

This tells us that $$x^J_{11}=-x^I_{42}/x^I_{41},$$
$$x^J_{12}=1/x^I_{41},$$
$$x^J_{31}=-\det\begin{bmatrix}x^I_{31} & x^I_{41} \\ x^I_{32} & x^I_{42}\end{bmatrix}/x^I_{41},$$
and
$$x^J_{32}=x^I_{31}/x^I_{41}.$$
\end{example}

There is a systematic way to determine these changes of coordinates.  Given a $k$-plane $V$, ``define'' $p_I(V)$ to be the determinant of the $k\times k$ submatrix of $M$ using the rows in $I$, where $M$ is any matrix whose columns are a basis for $V$.  Now, $p_I$ is not well-defined, since it depends on the choice of basis, but the ratio $p_I/p_J$ is well-defined for any $I$ and $J$ (why?).  Now combine the previous sentence with the observation that 
\[x^J_{ab}=p_{J'}/p_J, \text{ \ where $J'=\{j_1,\ldots,j_{b-1},a,j_{b+1},\ldots,j_k\}$. }\] 
(The set $J'$ might not be correctly ordered; it should be considered as an oriented set, so that we reorder it by some permutation and multiply by the sign of the permutation.)

\begin{exercise} 
Find the change of coordinates for $I=\{1,2\}$, $J=\{3,4\}$ by first using column operations and second by
the ``systematic'' method just described.
\end{exercise}

\subsubsection{Projective algebraic geometry} 
\begin{definition}
The \emph{complex projective plane} ${\mathbb P}^{n}$ is the set of equivalence classes of 
${\mathbb C}^{n+1}-\{(0,0,\ldots,0)\}$ under the equivalence relation $\sim$ defined by 
\[(x_0,x_1,\ldots,x_n)\sim (x_0',x_1',\ldots,x_n')\]
if
there exists $k\in {\mathbb C}^{*}$ such that $(x_0',x_1',\ldots,x_n')=k(x_0,x_1,\ldots,x_n)$. That is,
\begin{equation}
\label{eqn:Pndesc}
{\mathbb P}^{n}=\left({\mathbb C}^{n+1}-\{(0,0,\ldots,0)\}\right)/{\mathbb C}^*.
\end{equation}
 Equivalence classes are denoted $[x_0:x_1:x_2:\cdots:x_n]$.
\end{definition}

\begin{definition}
A \emph{projective variety} is a subset $X$ of ${\mathbb P}^n$ where $X$ is the common solutions of a system
of homogeneous polynomials in $x_0,x_1,\ldots,x_n$.
\end{definition}

\begin{exercise}[The example of ${\mathbb P}^1$]\label{exer:P1}

\gap

(a) Explain why ${\mathbb P}^1$ may be identified with $Gr_1({\mathbb C}^2)$. Hence $Gr_1({\mathbb C}^2)$
trivially has the structure of a projective variety.\footnote{Similarly, ${\mathbb P}^n$ can be identified with $Gr_{1}({\mathbb C}^{n+1})$ and (\ref{eqn:Pndesc})
is a special case of (\ref{eqn:Grdesc}).}

(b) The standard open charts of ${\mathbb P}^1$ are 
\[U:=\{[1:x_1]: x_1\in {\mathbb C}\} \text{\ and 
$V=\{[x_0:1]: x_0\in {\mathbb C}\}$.}\] 
Which is $U_{\{1\}}$ and which is $U_{\{2\}}$ in the notation of Subsection~\ref{subsec:2.1.1}? What are $p_{\{1\}}, p_{\{2\}}$?

(c) Determine the transition functions between $U$ and $V$.
\end{exercise}

In general, the $p_I$ are known as \emph{Pl\"ucker coordinates}, and they define an embedding 
\begin{align*}
Gr_k({\mathbb C}^n) & \rightarrow \mathbb{P}^{\binom{n}{k}-1} \\
V & \mapsto [p_{\{1,2,3,\ldots,k\}}(V):\cdots: p_I(V):\cdots: P_{\{n-k+1,n-k+2,\ldots,n\}}(V)]
\end{align*}
known as the \emph{Pl\"ucker embedding}.  

\begin{exercise}
Prove that this map is well-defined, i.e., it does not depend on the choice of basis for $V$,
and $[p_{\{1,2,3,\ldots,k\}}(V):\cdots: p_I(V):\cdots: P_{\{n-k+1,n-k+2,\ldots,n\}}(V)]\in \mathbb{P}^{\binom{n}{k}-1}$.
\end{exercise}

We already established that $Gr_k({\mathbb C}^n)$ is an algebraic variety, hence the Pl\"ucker embedding
is a map in the algebraic category. The map establishes that $Gr_{k}({\mathbb C}^n)$ has the structure of
a \emph{projective} algebraic variety. There are relations between the coordinates $\{p_I\}$ known as the \emph{Pl\"ucker relations} generating the ideal defining the image. No such relations occur in the case of
$Gr_1({\mathbb C}^2)={\mathbb P}^1$. In the case of $Gr_2({\mathbb C}^4)$ there is one relation:
\[p_{\{1,2\}}p_{\{3,4\}}-p_{\{1,3\}}p_{\{2,4\}}+p_{\{1,4\}}p_{\{2,3\}}=0.\] While part of the canon of Schubert varieties, we will actually not make use of these relations, so we refer the reader to  \cite{Kleiman.Laksov}
and \cite[Lecture 6]{Harris} for further reading. As in Exercise~\ref{exer:P1}, each chart $U_I$ is (the pullback of) the intersection of the image with the standard chart of $\mathbb{P}^{\binom{n}{k}-1}$ where $p_I\neq 0$.

\subsubsection{Schubert cells and varieties}
Once again fix an ordering of rows for column reduction.  We reduce from bottom to top and right to left, 
trying to make the bottom right (or southeastmost) entry the first pivot, with our pivots going from southeast to northwest.  Given a $k$-plane $V$, by uniqueness of reduced column echelon form, the set of pivot rows does not depend on the basis originally chosen for $V$.  Let 
$X_I^{\circ}\subseteq Gr_k({\mathbb C}^n)$ 
be the subset of points corresponding to $k$-planes whose set of pivot rows is $I$.  The sets $X_I^{\circ}$ are called \emph{Schubert cells}.  The Grassmannian is the disjoint union of its Schubert cells.

\begin{exercise}
Continuing Exercise~\ref{exer:P1}, describe the Schubert cells of ${\mathbb P}^1=Gr_1({\mathbb C}^2)$.
\end{exercise}

\begin{example}\label{exa:24}
In the case of $Gr_2({\mathbb C}^4)$, the Schubert cells are as follows:
$$
X_{34}^{\circ}=\begin{bmatrix} \ast & \ast \\ \ast & \ast \\ 1 & 0 \\ 0 & 1\end{bmatrix},
X_{24}^{\circ}=\begin{bmatrix} \ast & \ast \\ 1 & 0 \\ 0 & \ast \\ 0 & 1\end{bmatrix},
X_{14}^{\circ}=\begin{bmatrix} 1 & 0 \\ 0 & \ast \\ 0 & \ast \\ 0 & 1\end{bmatrix},
X_{23}^{\circ}=\begin{bmatrix} \ast & \ast \\ 1 & 0 \\ 0 & 1 \\ 0 & 0\end{bmatrix},
X_{13}^{\circ}=\begin{bmatrix} 1 & 0 \\ 0 & \ast \\ 0 & 1 \\ 0 & 0\end{bmatrix},
X_{12}^{\circ}=\begin{bmatrix} 1 & 0 \\ 0 & 1 \\ 0 & 0 \\ 0 & 0\end{bmatrix}.
$$
\end{example}

\begin{exercise}\label{exer:ShowthatFeb25}
Show that 
$\dim(X_I^{\circ})=\sum_{a=1}^k (i_a-a)$.
\end{exercise}
\begin{exercise}\label{exer:finitefield}
Let ${\mathbb F}_q$ is the finite field of order $q=p^k$ where $p$ a prime.
Show that the number of points of 
$Gr_k({\mathbb F}_q^n)$ is the \emph{$q$-binomial coefficient}
$\genfrac{[}{]}{0pt}{}{n}{k}_q:=\frac{[n]_q!}{[k]_q![n-k]_q!}$, where $[i]_q:=1+q+q^2+\cdots +q^{i-1}$ and
$[i]_q!:=[1]_q [2]_q\cdots [i]_q$.\footnote{Exercises~\ref{exer:ShowthatFeb25} and~\ref{exer:finitefield} together
imply that the (cohomological) Poincar\'e polynomial for the Grassmannian $P_{k,n}(q)$ whose 
coefficient of $q^k$ is the number of Schubert cells of codimension $k$ agrees with the point count of $Gr_k({\mathbb F}_q^n)$. This is 
a (rather trivial) instance of the (now proved) \emph{Weil conjectures} which relates cohomology Poincar\'e polynomials to point counting.}
\end{exercise}

\begin{definition}
\emph{Bruhat order} on ${\mathcal S}_{k,n}$ is the partial order obtained by declaring that, if $I=\{i_1<i_2<\cdots<i_k\}$ and $J=\{j_1<j_2<\cdots<j_k\}$, then $I\leq J$ if $i_a\leq j_a$ for all $a$, $1\leq a\leq k$.
\end{definition}

\begin{exercise}
  Show that $$X_I^{\circ}=U_I\setminus \bigcup_{J\not\leq I} U_J.$$
\end{exercise}

An equivalent, more succinct, way to define Schubert cells is as follows. 
\begin{definition}
A \emph{complete flag} in $\mathbb{C}^n$ is a nested sequence of subspaces
$$F_\bullet=F_1\subsetneq F_2\subsetneq\cdots\subsetneq F_{n-1}$$
of $\mathbb{C}^n$, with $\dim(V_i)=i$ for all $i$.  
\end{definition}
It is convenient to set $E_i$ to be the subspace spanned by the $\vec e_1, \vec e_2,\ldots, \vec e_i$
where $\vec e_t$ is the $t$-th standard basis element, and $E_\bullet=E_1\subsetneq\cdots\subsetneq E_{n-1}$. The flag $E_\bullet$ is called the \emph{standard flag}.

\begin{definition}
To each $I\in {\mathcal S}_{k,n}$, the
\emph{Schubert cell} is
\[X_I^{\circ}:=\{V\in Gr_k({\mathbb C}^n):\dim(V\cap E_t)=\#([t]\cap I), \ 1\leq t\leq n\}.\]
\end{definition}

Each subspace $V$ intersects the flag $E_\bullet$ in some specific sequence of dimensions, so
\begin{equation}
\label{eqn:decompFeb3hhh}
Gr_k({\mathbb C}^n)=\coprod_{I\in {\mathcal S}_{k,n}} X_I^\circ.
\end{equation}
\begin{exercise}\label{exer:24decomp}
Show that (\ref{eqn:decompFeb3hhh}) agrees with the decomposition described in Example~\ref{exa:24}. 
\end{exercise}

Now we define the main objects of this chapter, the Schubert varieties.

\begin{definition}
The \emph{Schubert variety} $X_I$ is the closure of $X_I^{\circ}$. (The closure in the Zariski topology or the Euclidean topology from the manifold structure is the same, so we will ignore this distinction throughout this chapter.)
\end{definition}

\begin{proposition}[Incidence condition description of Schubert varieties]
\gap
\begin{enumerate}
\item The Schubert variety $X_I$ is a union of Schubert cells as follows:
\[X_I^:=\coprod_{J\leq I} X_J^{\circ},\]
where $\leq$ refers to \emph{Bruhat order} on ${\mathcal S}_{k,n}$.
\item Therefore, the Schubert variety $X_I$ is defined by changing the intersection conditions in the definition of the Schubert cell from equalities to inequalities.  Precisely,
\begin{equation}
\label{eqn:intersectioncond}
X_I=\{V\in Gr_k({\mathbb C}^n):\dim(V\cap F_t)\geq\#([t]\cap I)\}.
\end{equation}
\end{enumerate}
\end{proposition}

\begin{example}\label{detJan27ex}
We consider the Schubert variety $X_{24}\subseteq Gr_{2}({\mathbb C}^4)$.  It contains all the points $Gr_2({\mathbb C}^4)$ other than those in the Schubert cell $C_{34}$.  By the intersection conditions (\ref{eqn:intersectioncond}), 
\[V\in X_{24} \text{\ if $\dim(V\cap F_1)\geq 0$, $\dim(V\cap F_2)\geq 1$, $\dim(V\cap F_3)\geq 1$, and $\dim(V \cap F_4)\geq 2$.}\]  
The first and last conditions are vacuous, and the third condition is implied by the second, so only the second condition is essential to the definition of $X_{24}$.

Suppose $\vec v_1,\vec v_2$ forms a basis for $V$.  Since $F_2$ is spanned by $\vec e_1$ and $\vec e_2$, by the previous paragraph, $V\in X_{24}$ if and only if the determinant of the matrix whose columns are $\vec v_1,\vec v_2, \vec e_1, \vec e_2$ is 0, or, equivalently, if $p_{34}(V)=0$.\footnote{Although $p_{I}(V)$ itself depends on the choice of a basis for $V$, whether $p_I(V)$ \emph{vanishes} does not, since a change of basis multiplies it by a nonzero scalar.}  We can express this condition in terms of the local coordinates $x^I_{ab}$ on each open set $U_I$.  For example, a subspace $V\in U_{12}$ has as one basis the columns of
$$\begin{bmatrix} 1 & 0 \\ 0 & 1 \\ x^{12}_{33}(V) & x^{12}_{34}(V) \\ x^{12}_{43}(V) & x^{12}_{44}(V)\end{bmatrix},$$
and 
\[V\in U_{12} \cap X_{24} \iff x^{12}_{33}x^{12}_{44}-x^{12}_{34}x^{12}_{43}=0.\]
\end{example}

Similar arguments work in general and this chapter's main ``principle'' is:

\begin{center}
\fbox{\begin{minipage}{30em}
 One can profitably understand Schubert varieties by study of their local equations; these local equations
define a determinantal variety.
   \end{minipage}}
\end{center}

A \emph{determinantal variety} means one defined by minors (possibly of different sizes) of some matrix of
indeterminants. For specifics on how the Grassmannian examples generalize, see, for example V.~Kreiman-V.~Lakshmibai \cite[Section~3.3]{Kreiman.Lakshmibai}.

\subsubsection{Group cosets}\label{sec:groupcosetsvvv}
It is useful to realize both the Grassmannian and its Schubert varieties in terms of a group quotient of $GL_n$.
To do this, fix a particular $k$-plane, such as $E=\langle \vec e_1,\ldots, \vec e_k\rangle$ and think about the stabilizer $P\subset GL_n$ of $E$ (where we think of $GL_n$ as acting on the underlying vector space $\mathbb{C}^n$).
The subgroup $P$ is one of a class of subgroups known as \emph{parabolic subgroups}.

\begin{exercise}
Show that the group of matrices sending $E$ to itself (meaning it sends every vector in $E$ to some other vector in $E$) is
\[P=\begin{pmatrix} 
* & \vline & * \\ 
\hline
0 & \vline & *
\end{pmatrix},\]
where the ``$0$'' block has $(n-k)$ rows and $k$ columns.
\end{exercise}

By the orbit-stabilizer theorem, $Gr_k({\mathbb C}^n)$ is identified with $GL_n/P$ as sets. The latter is a \emph{homogeneous space}
in the sense of, e.g., \cite[Section 23.3]{Fulton.Harris}. In fact, the topological and geometric structure of $Gr_k(\mathbb{C}^n)$ as described in the previous subsections agrees with the topological and geometric structure of this homogeneous space constructed by starting with $G$ and identifying all the points in the same right $P$-orbit.  
\begin{exercise}
Give a bijection from $Gr_k({\mathbb C}^n)$ to $GL_n/P$.
\end{exercise}
\begin{center}
\fbox{\begin{minipage}{30em}
 The advantage of this group formulation of the Grassmannian is that it extends naturally to replacing
$GL_n$ by another algebraic group, and $P$ with a different parabolic subgroup (the main case being the complete flag variety below, but see Section~\ref{sec:7}).
   \end{minipage}}
\end{center}
 From this perspective, if $B\subset P$ is the \emph{Borel subgroup} of invertible upper triangular
matrices, then $B$ acts by left-multiplication on $GL_n/P$ with finitely many orbits. Those orbits are precisely
the Schubert cells $X_I^{\circ}$. Each orbit contains a unique $T$-fixed point, where $T$ is the maximal torus of
invertible $n\times n$ matrices. In particular, if $I=\{i_1,\ldots,i_k\}$, then $X_I$ is the $B$-orbit of the $T$-fixed $k$-plane $W_I=\langle \vec e_{i_1}, \vec e_{i_2},\ldots,
\vec e_{i_k} \rangle$.  Since each Schubert cell is a $B$-orbit, and $B$ is an algebraic group, it provides
local isomorphisms between a neighborhood of $W_I$ and a neighborhood of any other point in the Schubert cell $C_I$. Thus:
\begin{center}
\fbox{\begin{minipage}{30em}
Local properties of a Schubert variety reduce, without loss of generality, to the study of
neighborhoods of its $T$-fixed points.
  \end{minipage}}
\end{center}

\subsection{Flag varieties}
The focus of this chapter will be the \emph{flag variety}.  In fact:
\begin{center}
\fbox{\begin{minipage}{30em}
Most questions about Schubert
varieties in the Grassmannian (or any partial flag variety) can be reduced to questions about Schubert varieties on the complete flag variety.   \end{minipage}}
\end{center}
See \cite[Section~1.2]{Brion.Lectures} for justification.

\begin{definition}
The flag variety ${\rm Flags}({\mathbb C}^n)$ is the parameter space of complete flags in ${\mathbb C}^n$.
\end{definition}

\subsubsection{Definitions via groups}
The group $GL_n$ acts transitively on ${\rm Flags}({\mathbb C}^n)$ since its action on $\mathbb{C}^n$ gives an action on subspaces and hence
on flags.
\begin{exercise}\label{exer:Borels}
(a) Let $E_\bullet$ be the standard flag where $E_i=\langle \vec e_1,\ldots, \vec e_i\rangle$ for all $i$.  Show that the stabilizer of $E_i$
is the group $B$ of invertible upper triangular matrices.

(b) Suppose $F_{\bullet}$ is similarly defined with respect to another ordered basis $\vec v_1, \vec v_2,\ldots, \vec v_n$ and the change of basis matrix from the standard basis to this one is $g\in GL_n$. What stabilizes $F_{\bullet}$?
\end{exercise}

By the orbit-stabilizer theorem, identify ${\rm Flags}({\mathbb C}^n)$ with $GL_n/B$. Parallel to the group quotient viewpoint on the Grassmannian and its Schubert cells, we have:
\begin{definition} The \emph{Schubert cells} of $GL_n/B$ are the
$B$-orbits under left multiplication.
\end{definition}
Let $S_n$ be the symmetric group on $[n]$. If $w\in S_n$, 
define the complete flag
\begin{equation}
\label{eqn:permutedflag}
E_{\bullet}^{(w)}=\langle \vec 0\rangle \subsetneq
\langle \vec e_{w(1)}\rangle \subsetneq \langle \vec e_{w(1)}, \vec e_{w(2)} \rangle \subsetneq \cdots
\subsetneq
\langle \vec e_{w(1)}, \vec e_{w(2)},\cdots, \vec e_{w(i)}\rangle \subsetneq \cdots \subsetneq {\mathbb C}^n.
\end{equation}

\begin{exercise} Show that:

(a) Equation \eqref{eqn:permutedflag} gives a bijection between $S_n$ and the $T$-fixed points of ${\rm Flags}({\mathbb C}^n)$.

(b) Each $B$-orbit of $GL_n/B$ contains a unique $T$-fixed point.
\end{exercise}

Hence, we have a bijection between $S_n$ and Schubert cells, and we can give the following definitions:

\begin{definition}\label{def:Schcellfirstdef}
Given $w\in S_n$, the \emph{Schubert cell} $X_w^{\circ}$ is the $B$-orbit $B\cdot E_\bullet^{(w)}$.
\end{definition}

\begin{definition}
The \emph{Schubert variety} $X_w$ is the closure of $X_w^{\circ}$. 
\end{definition}

\subsubsection{Definitions via incidence conditions; Bruhat order} As in the Grassmannian case, we have descriptions of Schubert cells and Schubert varieties in terms of dimensions of intersections and one in terms of column echelon form for an appropriate matrix.  To explain this,
we first define a matrix that encodes these dimensions.

\begin{definition}\label{Woosrank}
Given $w\in S_n$, the \emph{rank matrix} of $w$, denoted $R_w$, is the matrix
\[R_w:=[r^{(w)}_{p,q}]_{p,q=1}^n\] 
where
the entry $r^{(w)}_{p,q}=\#\{k \mid k\leq q, w(k)\leq p\}$.
\end{definition}

One can picture the rank matrix as follows.  First we have the following definition:
\begin{definition}
The \emph{permutation matrix} of $w$ is
\[P^{(w)}_{i,j}=\begin{cases}
1 & \text{if $i=w(j)$,}\\
0 & \text{otherwise.}
\end{cases}\]
\end{definition}
Then $(p,q)$-th entry of the rank matrix of $w$ is the number of $1$'s in the permutation matrix $P^{(w)}$
that are weakly northwest (meaning weakly above and weakly to the left) of $(p,q)$-th entry.

\begin{proposition}
The Schubert cell is given by
\[X_w^{\circ}=\{F_{\bullet}\mid \dim(E_p\cap F_q)=r^{(w)}_{p,q} \ \  \forall p,q\}.\]
\end{proposition}

\begin{proof}
The $T$-fixed flag $E^{(w)}$ satisfies the given condition, and the action of $B$ preserves the condition.
\end{proof}

As in the Grassmannian, the Schubert variety can be defined by relaxing the equalities to inequalities, so
\begin{equation}
\label{eqn:Xwineq}
X_w=\{F_{\bullet}\mid \dim(E_p\cap F_q)\geq r^{(w)}_{p,q} \ \ \forall p,q\}.
\end{equation}

\begin{definition}
\label{def:Bruhatorder}
\emph{Bruhat order} on $S_n$ is the partial order defined by $u\leq v$ if $r^{(u)}_{p,q}\leq r^{(v)}_{p,q}$ for all $p,q$ with $1\leq p,q\leq n$.  
\end{definition}

Combining \eqref{eqn:Xwineq} and Definition~\ref{def:Bruhatorder} implies the \emph{Bruhat decomposition},
a decomposition of $X_w$ into a disjoint union of Schubert cells:
\[X_w=\coprod_{v\leq w} X_v^\circ, \text{\ where $\leq$ denotes Bruhat order.}\]

Here are some additional definition and exercises about $S_n$ that we will refer to later.

\begin{definition}
The set of \emph{inversions} of $w\in S_n$ is 
\[{\rm Inv}(w)= \{(i,j)\mid 1\leq i<j\leq n, w(i)>w(j)\}.\]
\end{definition}

\begin{definition}\label{def:Coxlength}
The \emph{Coxeter length} of $w$ is 
$\ell(w):=\#{\rm Inv}(w)$.
\end{definition}

\begin{definition}\label{w0def}
The unique permutation in $S_n$ of longest length ${n\choose 2}$ is 
\[w_0:= n \ n-1 \ n-2 \cdots 3 \ 2 \ 1.\]
\end{definition}

\begin{exercise}\label{exer:Bruhatexer}
(a) \emph{Bruhat order} is also defined as the transitive, reflexive, antisymmetric closure of the covering relation 
$u\leq ut_{ij}$ if $\ell(ut_{ij})=\ell(u)+1$,
 where $t_{ij}=(i \ j )$ is a transposition. Prove the two definitions
of Bruhat order agree. 

(b) Prove that the covering relation $u\leq ut_{ij}$ holds for $i<j$ if and only if $w(i)<w(j)$ and there does not exist
$i<k<j$ such that $w(i)<w(k)<w(j)$.
\end{exercise}

\begin{exercise}\label{exer:additionalexerSymmetric}

(a) Prove that the symmetric group has a presentation given by the quotient of the free group
$\langle \sigma_1,\ldots,\sigma_{n-1}\rangle$ by the relations 
\[\sigma_i^2=id, \ \sigma_{i}\sigma_{i+1}\sigma_i =\sigma_{i+1}\sigma_i \sigma_{i+1}, \text{\ and 
$\sigma_i \sigma_j = \sigma_j\sigma_i$ if $|i-j|>1$.}\] 
(This is the Coxeter group presentation of $S_n$.)

(b) A factorization 
$w=s_{i_1}s_{i_2}\cdots s_{i_L}$
into simple transpositions $s_i=(i \ i+1)$ is a \emph{reduced word} for $w\in S_n$ if it is of
shortest length. Prove that any two such expressions are connected using the relations from (a).

(c) Prove that the length of any reduced word for $w\in S_n$ is $\ell(w)$.
\end{exercise}

For more about symmetric groups as Coxeter groups we refer the reader to \cite{bjorner.brenti}.

\subsubsection{Covering and transition equations}\label{sec:covtrans123}
As with Grassmannians, we can give the flag variety the structure of a complex manifold or complex algebraic variety by providing a covering by affine charts together with smooth transition functions.

Given a flag $F_\bullet$, write a matrix $M$ such that, for all $i$, the first $i$ columns of $M$ are a basis of $F_i$.
Different choices of basis are related by rightward column operations on $M$, which correspond to multiplying $M$ on the right by an upper triangular matrix.
That is, we only allow the column operations of
adding a multiple of a column to a column to its right and multiplying a column by a nonzero constant; in particular we do not allow
switching columns, which means that our pivots will not necessarily show up in staircase order. 

Given $w\in S_n$, we define an open set 
\[U_w\subseteq {\rm Flags}({\mathbb C}^n)\]
as the set of all flags whose matrices can be column reduced (using only the operations of adding a multiple of a column to a column to its
right and multiplying a column by a nonzero constant) with the pivots being the $1$ entries in $P^{(w)}$.  Since every invertible matrix can be column
reduced with some set of pivots (and every flag is represented by an invertible matrix), the open sets $U_w$ for all $w\in S_n$ cover ${\rm Flags}({\mathbb C}^n)$.
Moreover, by uniqueness of reduced column echelon form, we have local coordinate functions $x^{(w)}_{ab}$ on $U_w$.  Furthermore, $C_w$ consists of all points in $U_w$
that are not in $U_v$ for some $v\geq w$ (in Bruhat order).

\begin{example}\label{exa:Uflags3}
For ${\rm Flags}({\mathbb C}^3)$, the atlas consists of:
$$
U_{123}=\begin{bmatrix} 1 & \ast & \ast \\ 0 & 1 & \ast \\ 0 & 0 & 1\end{bmatrix},
U_{213}=\begin{bmatrix}  0 & 1 & \ast \\ 1 & \ast & \ast \\ 0 & 0 & 1\end{bmatrix},
U_{132}=\begin{bmatrix} 1 & \ast & \ast \\  0 & 0 & 1 \\ 0 & 1 & \ast  \end{bmatrix},$$
$$
U_{231}=\begin{bmatrix} 0 & 0 & 1  \\ 1 & \ast & \ast  \\ 0 & 1 & \ast  \end{bmatrix},
U_{312}=\begin{bmatrix} 0 & 1 & \ast \\ 0 & 0 & 1 \\ 1 & \ast & \ast   \end{bmatrix},
U_{321}=\begin{bmatrix} 0 & 0 & 1\\ 0 & 1 & \ast \\ 1 & \ast & \ast  \end{bmatrix}.
$$
\end{example}

 We can also pull back Pl\"ucker coordinates from Grassmannians as follows.
Given a flag $F_\bullet=F_1\subsetneq\cdots\subsetneq F_{n-1}$, and a subset $I\subseteq [n]$ with $k$ elements, ``define'' $p_I(F_\bullet)$ as $p_I(F_k)$.
\begin{exercise}
(a) Show that, on $U_w$, $x^{(w)}_{ab}$ is a particular ratio of Pl\"ucker coordinates, similar to the case of Grassmannians (see Section~\ref{subsubsec:transgr}).

(b) Write coordinates on $U_v$ as rational functions of coordinates on $U_w$, thereby establishing
${\rm Flags}({\mathbb C}^n)$ as a complex manifold and as an algebraic variety.
\end{exercise}

\subsubsection{Description of Schubert cells} We can also describe each $X_w^{\circ}$ using column echelon form.  Given a flag $F_\bullet$, write a matrix $M$ such that, for all $i$, the first $i$ columns of $M$ are a basis of $F_i$.
Here we do column reduction from left to right, making the bottommost choice of pivot at every step.  Then,
given a permutation $w\in S_n$, the Schubert cell $X_w^{\circ}$ consists of the flags corresponding to the matrices $M$ whose column echelon form has pivots at the $1$'s
in the permutation matrix $P^{(w)}$.

\begin{example}\label{exa:Schcells3}
For ${\rm Flags}({\mathbb C}^3)$, the Schubert cells are:
$$
X_{123}^{\circ}=\begin{bmatrix} 1 & 0 & 0 \\ 0 & 1 & 0 \\ 0 & 0 & 1\end{bmatrix},
X_{213}^{\circ}=\begin{bmatrix}  \ast & 1 & 0 \\ 1 & 0 & 0 \\ 0 & 0 & 1\end{bmatrix},
X_{132}^{\circ}=\begin{bmatrix} 1 & 0 & 0 \\  0 & \ast & 1 \\ 0 & 1 & 0  \end{bmatrix},$$
$$
X_{231}^{\circ}=\begin{bmatrix} \ast & \ast & 1  \\ 1 & 0 & 0  \\ 0 & 1 & 0  \end{bmatrix},
X_{312}^{\circ}=\begin{bmatrix} \ast & 1 & 0 \\ \ast & 0 & 1 \\ 1 & 0 & 0   \end{bmatrix},
X_{321}^{\circ}=\begin{bmatrix} \ast & \ast & 1\\ \ast & 1 & 0 \\ 1 & 0 & 0  \end{bmatrix}.
$$
\end{example}

\begin{exercise}\label{exer:upandright}
For each $w\in S_n$ prove that each matrix in the above description of $X_{w}^\circ$ is represents a unique right $B$ coset in $BP_wB/B$. That is, prove this description of $X_w^{\circ}$ agrees with Definition~\ref{def:Schcellfirstdef}.
\end{exercise}

\begin{exercise}\label{exer:potentially}
The reduced column echelon form has a potentially nonzero entry at $(i,j)$ if and only if $w(j)>i$ and $w^{-1}(i)>j$.  This gives a bijection between
potentially nonzero entries and ${\rm Inv}(w)$.
\end{exercise}

Exercise~\ref{exer:potentially} combined with Definition~\ref{def:Coxlength} implies
\[X_w^{\circ}\cong \mathbb{C}^{\ell(w)}.\]
The Schubert cell $X_{w_0}^{\circ}=U_{w_0}$ is dense in ${\rm Flags}({\mathbb C}^n)$; it is
therefore called the \emph{big cell}.

\subsubsection{The flag variety as a projective algebraic variety}
Recall that in Section~\ref{sec:covtrans123} we defined 
for each $F_{\bullet}\in {\rm Flags}({\mathbb C}^n)$ and each $I\in {\mathcal S}_{k,n}$ the Pl\"{u}cker
coordinate $p_I(F_{\bullet})$ to mean $p_I(F_k)$. Viewing ${\rm Flags}({\mathbb C}^n)$ as a subset of the
product of Grassmannians 
\[Gr:=Gr_1({\mathbb C}^n)\times Gr_2({\mathbb C}^n) \times \cdots \times
Gr_{n-1}({\mathbb C}^n),\] 
we can use the Pl\"ucker coordinates to first embed ${\rm Flags}({\mathbb C}^n)$ into
 a product of projective spaces 
 \[{\mathbb P}:={\mathbb P}^{{n\choose 1}-1}\times {\mathbb P}^{{n\choose 2}-1}\times \cdots\times
 {\mathbb P}^{{n\choose n-1}-1}.\] 
 This can be followed up with  an embedding of ${\mathbb P}$ into a single (very large dimensional!) projective space ${\mathbb P}'$ using the \emph{Segre embedding}.  In fact, each $U_w$ is the intersection of this embedding with a standard chart
 on ${\mathbb P}'$. While one can find local equations for Schubert varieties using this embedding, we take a different approach in the next section.

\subsubsection{Opposite Schubert cells}
\label{sec:opschubJan28}
Schubert cells and varieties are arbitrarily defined in terms of a choice of reference flag (equivalently, a choice of Borel
subgroup $B$). Having made the standard choice of $B=$ invertible upper triangular matrices, it will be important
in Section~\ref{sec:3} to make use of the ``opposite" choice of Borel subgroup, namely, $B_-=$ the group of invertible lower triangular matrices.  

\begin{definition}\label{def:oppcell} For $w\in S_n$, the
\emph{opposite Schubert cell} is the $B_-$-orbit 
\[\Omega_w^\circ:=B_-\cdot E^{(w)}_\bullet\subseteq {\rm Flags}({\mathbb C}^n)\cong GL_n/B.\]  
\end{definition}

 The following exercises are about the analogues of  statements for Schubert cells.
\begin{exercise}\label{exer:Jan27abc}
Show that $\Omega_w^\circ$ is the set of flags whose matrix
representatives have reduced column echelon forms with pivots at the $1$'s in $P_w$ when we make the \emph{topmost} instead of bottommost choice of pivot at every
column.  Then deduce that the column echelon form (with this choice of pivots) of a matrix representing a flag in $\Omega_w^\circ$ has a potentially nonzero entry at $(i,j)$ if and only if $w(j)<i$ and $w^{-1}(i)>j$, and hence that 
$\Omega_w^\circ\cong\mathbb{C}^{\binom{n}{2}-\ell(w)}$.
\end{exercise}

By Exercise~\ref{exer:Jan27abc}, if $w=id$, the opposite Schubert cell $\Omega_{id}^{\circ}\cong \mathbb{C}^{\binom{n}{2}-\ell(w)}$ is dense in ${\rm Flags}({\mathbb C}^n)$. It is called the \emph{opposite big cell} of ${\rm Flags}({\mathbb C}^n)$.

\begin{example}\label{exa:Schcells3}
For ${\rm Flags}({\mathbb C}^3)$, the opposite Schubert cells are:
$$
\Omega_{123}^{\circ}=\begin{bmatrix} 1 & 0 & 0 \\ \ast & 1 & 0 \\ \ast & \ast & 1\end{bmatrix},
\Omega_{213}^{\circ}=\begin{bmatrix}  0 & 1 & 0 \\ 1 & 0 & 0 \\ \ast & \ast & 1\end{bmatrix},
\Omega_{132}^{\circ}=\begin{bmatrix} 1 & 0 & 0 \\  \ast & 0 & 1 \\ \ast & 1 & 0  \end{bmatrix},$$
$$
\Omega_{231}^{\circ}=\begin{bmatrix} 0 & 0 & 1  \\ 1 & 0 & 0  \\ \ast & 1 & 0  \end{bmatrix},
\Omega_{312}^{\circ}=\begin{bmatrix} 0 & 1 & 0 \\ 0 & \ast & 1 \\ 1 & 0 & 0   \end{bmatrix},
\Omega_{321}^{\circ}=\begin{bmatrix} 0 & 0 & 1\\ 0 & 1 & 0 \\ 1 & 0 & 0  \end{bmatrix}.
$$
\end{example}

Let $\widetilde{R}_w$ be the matrix $\widetilde{R}_w:=[\widetilde{r}^{(w)}_{p,q}]_{p,q=1}^n$ where
the entry 
\begin{equation}
\label{eqn:SWrankmatrix}
\widetilde{r}^{(w)}_{p,q}=\#\{k \mid k\leq q, w(k)\geq p\}.
\end{equation}

\begin{exercise}
Show $\Omega_w^\circ=\{F_\bullet \mid E^{(w_0)}_p\cap F_q = \widetilde{r}_{p,q}(w)\}$, where
$E^{(w_0)}_p=\langle e_{n+1-p},\ldots, e_n\rangle$.
\end{exercise}

One defines \emph{opposite Schubert varieties} by $\Omega_w=\overline{\Omega_w^\circ}$.  The opposite Schubert cells and varieties are translates of the usual Schubert cells and varieties: $\Omega^\circ_{w}=P_{w_0}\cdot X^\circ_{w_0w}$.

\section{Kazhdan-Lusztig ideals and varieties}\label{sec:3}

\begin{center}
\fbox{\begin{minipage}{30em}
The analogue of the determinantal ideal of Example~\ref{detJan27ex} is
the \emph{Kazhdan-Lusztig ideal} \cite{WY:governing}. These were introduced to study affine open neighborhoods of Schubert varieties at
a $T$-fixed point \cite{Kazhdan.Lusztig}.    \end{minipage}}
\end{center}

The opposite big cell $\Omega_{id}^{\circ}$ (Definition~\ref{def:oppcell}) is an affine open neighborhood
of ${\rm Flags}({\mathbb C}^n)$. Hence $v\Omega_{id}^{\circ}\cap X_w$ is an affine open neighborhood of $X_w$ centered at $e_v$. Suppose $X\subset {\rm Flags}({\mathbb C}^n)$ is any subvariety.
\begin{definition}
\label{defn:patch}
The \emph{patch} of $X$ at any point $gB\in GL_n/B$ is the affine open neighborhood $g\Omega_{id}^{\circ}\cap X$.
\end{definition}

This lemma of D.~Kazhdan-G.~Lusztig underpins the approach of the chapter:

\begin{lemma}[Lemma~A.4 of \cite{Kazhdan.Lusztig}] 
\label{lemma:KL}
$X_w\cap x\Omega^\circ_{id}\cong (X_w \cap \Omega^\circ_x)
\times \mathbb{A}^{\ell(x)}.$
\end{lemma}

\begin{exercise}\label{exer:KLexer}
Identify $\Omega_x^{\circ}$ with the space of matrices described in Section~\ref{sec:opschubJan28}. Futhermore,
identify ${\mathbb A}^{\ell(x)}$ with the space of unit lower triangular matrices 
$m=\left[m_{ij}\right]_{i,j=1}^n$ for which
$m_{ij}=0$ (for $i>j$) unless $x(j)>x(i)$.

(a) Define a map $\eta:\Omega_{x}^{\circ}\times {\mathbb A}^{\ell(x)}\to x\Omega_{id}^{\circ}$ by 
$\eta(m,a)=ma$ (matrix multiplication). Show $\eta$ is an isomorphism.  

(b) Suppose $m\in (X_w\cap \Omega_x^{\circ}), a\in {\mathbb A}^{\ell(x)}$. Show $\eta(m,a)\in X_w\cap x\Omega_{id}^{\circ}$, and thereby complete the proof of Lemma~\ref{lemma:KL}.
\end{exercise}

By virtue of Lemma~\ref{lemma:KL}, one can replace the patch $X_w\cap x\Omega_{id}^{\circ}$ with a variety of smaller
dimension, by ignoring the irrelevant affine factor ${\mathbb A}^{\ell(x)}$. That is, it suffices to study:

\begin{definition} The \emph{Kazhdan-Lusztig variety} is ${\mathcal N}_{v,w}:=X_w\cap \Omega_v^{\circ}$.
\end{definition}

Explicit coordinates and equations for ${\mathcal N}_{v,w}$ were investigated in \cite{WY:governing}. Let ${\sf Mat}_{n\times n}$ be the space of 
all $n\times n$ complex matrices. The coordinate ring is ${\mathbb C}[{\bf z}]$ where ${\bf z}=\{z_{ij}\}_{i,j=1}^n$ are the
functions on the entries of a generic matrix $Z$. Here 
\[z_{ij} = \text{the entry in the $i$-th row
from the \emph{bottom}, and the $j$-th column to the left.}\] 

Following Section~\ref{sec:opschubJan28} and in this notation, we identify $\Omega_{v}^{\circ}$ as the 
affine subspace of ${\sf Mat}_{n\times n}$ consisting of matrices $Z^{(v)}$ where $z_{n-v(i)+1,i}=1$, and $z_{n-v(i)+1,s}=0, z_{t,i}=0$ for $s>i$ and $t>n-v(i)+1$. Let ${\bf z}^{(v)}\subseteq {\bf z}$ 
be the unspecialized variables. Furthermore, let $Z_{st}^{(v)}$ be the southwest $s\times t$ submatrix of $Z^{(v)}$. 

\begin{definition}\label{def:KLideal}
The \emph{Kazhdan-Lusztig ideal} is 
$I_{v,w}\subset {\mathbb C}[{\bf z}^{(v)}]$ 
generated by all ${\widetilde r}_{st}^{(w)}+1$ minors
of $Z_{st}^{(v)}$ where $1\leq s,t\leq n$ and
${\widetilde r}_{st}^{(w)}=\#\{k \mid k\leq q, w(k)\geq p\}$ as in (\ref{eqn:SWrankmatrix}).
\end{definition}

\begin{exercise}\label{exer:theeqnsJan30}
Show that ${\mathcal N}_{v,w}$ is set-theoretically cut out by $I_{v,w}$. That is, $P\in {\mathcal N}_{v,w}$
if and only if $P$ is a zero of the generators of $I_{v,w}$. 
\end{exercise}

\begin{example}\label{exa:1} 
Let $w=7234615$, $v=2136457$ (in one line notation). The rank matrix 
 ${\widetilde R}_w$ and the matrix of variables $Z^{(v)}$ are:
 \[{\widetilde R}_w=\left(\begin{matrix}
 1&2&3&4&5&6&7\\
 1&2&3&4&5&5&6\\
 1&1&2&3&4&4&4\\
 1&1&1&2&3&3&4\\
 1&1&1&1&2&2&3\\
 1&1&1&1&2&2&2\\
 1&1&1&1&1&1&1
 \end{matrix}\right), \ \ 
 Z^{(v)}=\left(\begin{matrix}
0 & 1 & 0 & 0 &0 &0 &0\\
1 & 0 & 0 & 0 & 0 &0 &0\\
z_{51} & z_{52} & 1 & 0 & 0 &0 &0\\
z_{41} & z_{42} & z_{43} & 0 & 1 & 0 &0\\
z_{31} & z_{32} & z_{33} & 0 & z_{35} &1 &0\\
z_{21} & z_{22} & z_{23} & 1 & 0 &0 &0\\
z_{11} & z_{12} & z_{13} & z_{14} & z_{15} & z_{16} &1\\
\end{matrix}\right)\]
The Kazhdan-Lusztig ideal $I_{v,w}$ contains among its generators, all $2\times 2$ minors
of $Z_{52}^{(v)}$. It also has inhomogeneous generators such as 
\begin{equation}
\label{eqn:inhomog}
\left|\begin{matrix} z_{33} & 0 & z_{35}\\ z_{23}&1&0\\ z_{13}& z_{14} & z_{15}\end{matrix}\right|=z_{33}z_{15}+z_{35}z_{23}z_{14}-z_{35}z_{13}.
\end{equation}
\end{example}

\begin{definition}
We say $I_{v,w}$ is \emph{standard homogeneous} if there exist homogeneous polynomials that generate $I_{v,w}$. In this case, we also say ${\mathcal N}_{v,w}$ is standard homogeneous. 
\end{definition}

\begin{exercise}\label{exer:inhomgex}
Find an example of a Kazhdan-Lusztig ideal $I_{v,w}$, whose defining generators (those from Definition~\ref{def:KLideal})
are not all homogeneous (as in~\eqref{eqn:inhomog}), but which is standard homogeneous.
\end{exercise}

\begin{problem}\label{prob:standardhomogeneous}
Classify all $(v,w)\in S_n \times S_n$ such that $I_{v,w}$ is standard homogeneous.
\end{problem}

For some analysis concerning Problem~\ref{prob:standardhomogeneous}, see \cite[Sections~5.1, 5.2]{WY:Grobner}.
For more recent work on this problem, see \cite[Proposition~6.3]{Neye}. 

\begin{hardexercise}\label{exer:inverseiso}
Prove that ${\mathcal N}_{v,w}\cong {\mathcal N}_{v^{-1},w^{-1}}$. That is, all Kazhdan-Lusztig varieties
of $X_w$ are isomorphic to those of $X_{w^{-1}}$. Does $X_w\cong X_{w^{-1}}$ always hold?
\end{hardexercise}

\begin{problem}[\cite{WY:governing}]
\label{problem:KLiso}
Prove or disprove: if $[u,v]$ and $[u',v']$ are Bruhat-poset intervals in $S_{n}$ and $S_{n'}$ respectively such that
$[u,v]\cong [u',v']$ (as posets) then ${\mathcal N}_{u,v}\cong {\mathcal N}_{u',v'}$.
\end{problem}

An affirmative answer to Problem~\ref{problem:KLiso} would have consequences for 
each of the numerical measures of singularity in this paper.\footnote{As far as we know, little seems known about
classifying isomorphism classes of Bruhat intervals, or efficiency of algorithms to decide if
two such intervals are isomorphic. However, a theorem of M.~Dyer \cite{Dyer} shows that, for any $k\in {\mathbb N}$ there are only finitely many non-isomorphic intervals $[v,w]$ of height $\ell(w)-\ell(v)=k$.}

\begin{exercise}
Define coordinates and equations for the \emph{patch ideal} associated to the patch $v\Omega_{id}^{\circ}\cap X_w$.
\end{exercise}

The following concept was introduced by W.~Fulton \cite{Fulton:duke}:

\begin{definition}\label{def:Schubdetidealabc}
The \emph{Schubert determinantal ideal} $I_w$ is defined similarly as $I_{v,w}$ except that we 
replace $Z^{(v)}$ with the matrix $Z=(z_{ij})$.
\end{definition}

\begin{example}
Let $w=3412$, then 
\[{\widetilde R}_w=\left(\begin{matrix}
 1&2&3&4\\
 1&2&2&3\\
 1&2&2&2\\
 0&1&1&1
 \end{matrix}\right), \ \ 
 Z=\left(\begin{matrix}
z_{41} & z_{42} & z_{43} & z_{44}\\
z_{31} & z_{32} & z_{33} & z_{34}\\
z_{21} & z_{22} & z_{23} & z_{24}\\
z_{11} & z_{12} & z_{13} & z_{14}
\end{matrix}\right), I_w=\langle z_{11}, \left| 
\begin{matrix}
z_{31} & z_{32} & z_{33}\\
z_{21} & z_{22} & z_{23}\\
z_{11} & z_{12} & z_{13}
\end{matrix}\right|\rangle.
\]
\end{example}

\begin{definition}
The zero-set of $I_w$ in ${\sf Mat}_{n\times n}$ is the \emph{matrix Schubert variety}, denoted ${\mathfrak X}_{w}$.
\end{definition}

\begin{exercise}\label{exer:specialcase333}
The Schubert determinantal ideal and matrix Schubert
variety are special cases of their Kazhdan-Lusztig counterparts. (In fact, show this in two different ways.)
\end{exercise}

\begin{hardexercise} \label{exer:mono}
The \emph{monomialization} of an ideal $I$, denoted ${\sf mono}(I)$ is the largest monomial ideal contained
in $I$. Determine ${\sf mono}(I_w)$.
\end{hardexercise}

\section{Interval pattern avoidance}\label{sec:5}

\begin{center}
\fbox{\begin{minipage}{30em}
Interval pattern embedding/avoidance
gives a 
universal combinatorial language to study singularities of Schubert varieties.
   \end{minipage}}
\end{center}

We start with the classical notion of permutation pattern avoidance. Let $v\in S_m$ and $w\in S_n$ be two permutations, where $m\leq n$.

\begin{definition}\label{def:basicpattern}
The permutation $v\in S_m$ \emph{embeds in} $w\in S_n$ if there exist indices
$1\leq\phi_1<\phi_2<\ldots<\phi_m\leq n$ such that $w(\phi_1),
w(\phi_2), \ldots, w(\phi_m)$ are in the same relative order as
$v(1),\ldots, v(m)$. 
\end{definition}
 In other words, we require that $w(\phi_j)<w(\phi_k) \iff v(j)<v(k)$.  

\begin{definition}
The permutation $w$ (classically) \emph{avoids} $v$ if no such embedding exists.
\end{definition}

As a useful warmup exercise, we consider the classes of covexillary and cograssmannian permutations.\footnote{In the literature, vexillary and grassmannian permutations are more commonly used. $w\in S_n$ is
vexillary (resp.~grassmannian) if $w_0w$ is covexillary (resp.~cograssmannian).}  Many of the problems we discuss in this chapter are easier in these special cases, so these classes of permutations will be mentioned several times.

\begin{definition}
$w\in S_n$ is \emph{covexillary} if it is $3412$-avoiding. 
\end{definition}
\begin{definition}
$w\in S_n$ is \emph{cograssmannian} if
it contains at most one ascent, that is here is at most one index $k$ such that $w(k)<w(k+1)$.
\end{definition}

The following exercise characterizes covexillary and cograssmannian permutations in terms of an important combinatorial object associated to a permutation.

\begin{exercise}\label{exer:rotheess}
(a) \emph{Fulton's essential set} is
\[E(w)=\{(i,j)\in D(w): (i,j+1),(i-1,k)\not\in D(w)\}.\footnote{This is upside down from Fulton's original
definition \cite{Fulton:duke}. In the literature this is sometimes called the \emph{coessential set}.}\]
Characterize $w$ covexillary in terms of $E(w)$. Do the same for $w$ cograssmannian.

(b) Prove that if $w$ is cograssmannian then it is covexillary.
\end{exercise}

The combinatorial notion of pattern avoidance entered the study of Schubert varieties with the theorem of V.~Lakshmibai--B.~Sandhya~\cite{Lak} mentioned in the introduction.  Although pattern avoidance continued to be used in the study of Schubert varieties, its appearance was somewhat mysterious until S.~Billey--T.~Braden~\cite{Billey.Braden} defined the \emph{pattern map} and used it to give a geometric explanation for why $X_w$ must be singular if $X_v$ is singular and $v$ pattern embeds in $w$.\footnote{The pattern map was also independently defined by N.~Bergeron--F.~Sottile~\cite{Bergeron.Sottile} in the context of Schubert calculus.}  The authors gave a characterization of Gorenstein Schubert varieties using notions based on pattern avoidance in~\cite{WY:Goren}, but showed that pattern avoidance cannot be sufficient to characterize Gorenstein Schubert varieties.  In part to explain this appearance of pattern avoidance ideas, the authors showed in ~\cite{WY:governing} that interval pattern avoidance, a generalization of pattern avoidance, suffices to characterize any local property satisfying certain very mild hypotheses.  This notion is described in the remainder of this section.

Let $[u,v]$ and $[x,w]$ be poset intervals
in the Bruhat orders on $S_m$ and $S_n$ respectively.  

\begin{definition}
The interval $[u,v]\subseteq S_m$ \emph{interval pattern embeds in} $[x,w]\subseteq S_n$ if there is a
common embedding $\Phi=(\phi_1,\ldots,\phi_m)$ of $u$ into $x$ and $v$
into $w$, where the entries of $x$ and $w$ outside of $\Phi$ agree,
and, furthermore, $\ell(v)-\ell(u)=\ell(w)-\ell(x)$.
\end{definition}

We have the following easy exercise.

\begin{exercise}\label{exer:determined}
Given an interval $[u,v]\in S_m$, a permutation $w\in S_n$, and a sequence of embedding indices $\Phi=(\phi_1,\ldots,\phi_m)$ by which $v$ embeds in $w$, there is a unique
permutation $x$ such that $\Phi$ gives an embedding of $u$ into $x$ and such that the entries of $x$ and $w$ outside of $\Phi$ agree.
\end{exercise}

In light of Exercise~\ref{exer:determined}, we write $\Phi(u)$ for this unique permutation $x$.  (Note that the notation is slightly misleading as $w$ is also required to determine $x$.)

\begin{exercise}
\label{exa:45132}
Let $u = 21453 = s_3 s_4 s_1$ and $v=45132= s_2 s_3 s_2 s_4 s_3 s_1 s_2$.
Note $u\leq v$.  Now let $\Phi$ be the
embedding of $v$ into $w=\underline{7}\underline{8}1\underline{2}9\underline{5}6\underline{3}4$ where the underlined positions
indicate the embedding. That is, $\phi_1 =1,
\phi_2 = 2, \phi_3 =4, \phi_4 =6, \phi_5 =8$. Then $\Phi(u) =
\underline{3}\underline{2}1\underline{7}9\underline{8}6\underline{5}4$. 
Check that $[u,v]$ embeds into
$[\Phi(u),w]$.
\end{exercise}

Exercise~\ref{exer:determined} permits us to make the following definition:

\begin{definition}
The interval $[u,v]$ \emph{embeds} in $w$ if $[u,v]$ embeds in $[\Phi(u),w]$.  
\end{definition}

\begin{exercise}
\label{lemma:embed_iso}
An embedding $\Phi$ of $[u,v]$ into $[\Phi(u),w]$ is an interval
pattern embedding if and only if
$[u,v]$ and $[\Phi(u),w]$ are isomorphic as posets.
\end{exercise}

Now the following terminology is natural.

\begin{definition} The permutation $w$ \emph{interval pattern avoids}
$[u,v]$ if there are no interval pattern embeddings of $[u,v]$ into
$[x,w]$ for any $x\leq w$.
\end{definition}

Note that classical pattern avoidance is indeed a special case of interval avoidance, since $w$ avoids $v$ if and only if $w$ avoids the interval $[v,v]$.

\begin{exercise}\label{exer:Feb11uuu}
Let $u=21534$ and $v=31524$. Construct $w$ that contains $v$ in the classical sense but
interval pattern avoids $[u,v]$.
\end{exercise}

To state the universality theorem for interval pattern avoidance, let 
$S_\infty  = \bigcup_{r\geq 1} S_r$
be the infinite symmetric group of permutations on ${\mathbb N}=\{1,2,3,\ldots\}$ with only finitely many non-fixed points. 
Set
\[\mathfrak{S}=\{[u,v]: u\leq v \mbox{ in some $S_r$}\}\subseteq S_\infty \times S_{\infty}.\]

\begin{definition}
Define $\prec_I$ to be the partial order on $\mathfrak{S}$ generated by:
\begin{enumerate}
\item $[u,v]\prec_I [x,w]$ if there is an interval pattern embedding of
$[u,v]$ into $[x,w]$, and
\item $[u,v]\prec_I [u^\prime, v]$ if $u^\prime\leq u$.
\end{enumerate} 
\end{definition}

\begin{definition}
An \emph{upper order ideal} ${\mathcal I}$ (under the partial order
$\prec_I$) is a subset of $\mathfrak{S}$ such that, if $[u,v]\in
{\mathcal I}$ and $[u,v]\prec_I [x,w]$, then $[x,w]\in {\mathcal I}$.
\end{definition}

\begin{definition}
\label{def:SSPname}
A local algebraic property ${\mathcal P}$ of varieties is a 
\emph{semicontinuously stable property} (SSP) if the set of points at which holds on any variety is a closed subset of that
variety, and the property is preserved under products with affine space. 
\end{definition}

\begin{exercise}[``Things only get worse as you move down Bruhat order.'']\label{exer:getsworse}
Prove that if ${\mathcal P}$ is SSP and holds at $E^{(v)}_{\bullet}\in X_w$ then it holds for any $E^{(v')}_{\bullet}\in X_w$ where
$v'\leq v$ in Bruhat order. 
\end{exercise}

\begin{theorem}
\label{thm:poset}
Let ${\mathcal P}$ be SSP. The set of
intervals $\{[u,v]\}\subseteq \mathfrak{S}$ such that ${\mathcal P}$
holds at the $T$-fixed point $E^{u}_{\bullet}$ on the Schubert variety $X_v$ is an
upper order ideal ${\mathcal I}_{\mathcal P}$ under $\prec$.
\end{theorem}

The following exercise outlines the proof of Theorem~\ref{thm:poset} in \cite{WY:Grobner}.

\begin{exercise}
\label{exa:proof_example}
Let $u,v,w,\Phi$ be as in Exercise~\ref{exa:45132}. Identify opposite Schubert cells
with spaces of matrices as in Section~\ref{sec:opschubJan28}.

Let 
$[9]-\{\phi_1<\phi_2<\phi_3<\phi_4<\phi_5\}=\{\overline\phi_1<\overline\phi_2<\overline\phi_3<\overline\phi_4\}$
be the \emph{non}-embedding indices. Define the algebraic map 
\[\Psi:{\mathcal N}_{\Phi(u), w}\to \Omega^\circ_u\] 
as the projection which deletes the columns
$\overline{\phi}_1,\ldots, \overline{\phi}_{4}$ and rows $w(\overline{\phi}_1),\ldots, w(\overline{\phi}_{4})$ 
from an element $g\in {\mathcal N}_{\Phi(u), w}$. 

(a) For this example, prove ${\mathcal N}_{u,v}\cong {\mathcal N}_{\Phi(u),w}$ (isomorphism of algebraic varieties).

(b) Generalizing the reasoning from (a), prove ${\mathcal N}_{u,v}\cong {\mathcal N}_{\Phi(u),w}$ whenever $[u,v]$ interval embeds into
$[\Phi(u),w]$.

(c) Prove Theorem~\ref{thm:poset} from (b) and Exercise~\ref{exer:getsworse}.
\end{exercise}

We also wish to characterize Schubert varieties that globally avoid
${\mathcal P}$, or, in other words, those Schubert varieties for which
${\mathcal P}$ does not hold at any point.
The following corollary says that this can be done in terms of
interval pattern avoidance.

\begin{corollary}[{\cite[Corollary~2.7]{WY:governing}}]
\label{cor:brpa}
Let $\mathcal{P}$ be a SSP. Then the set
of permutations $w$ such that $\mathcal{P}$ does not hold at any point
of $X_w$ is the set of permutations $w$ that avoid all the intervals
$[u_i,v_i]$ in some (possibly infinite) set
$A_\mathcal{P}\subseteq \mathfrak{S}$.
\end{corollary}

The remainder of this section gives an interesting application of interval pattern avoidance beyond the study of Schubert varieties.
We first recall the celebrated
\emph{Schensted correspondence}. This is a bijection 
\[{\sf Schensted}: S_n\stackrel{\sim}{\longrightarrow} \bigcup_{\lambda} {\sf SYT}(\lambda)\times {\sf SYT}(\lambda),\]
where the union is over all integer partitions $\lambda$ of size $n$ and ${\sf SYT}(\lambda)$ is the set
of standard Young tableau of shape $\lambda$. This is computed by \emph{column} inserting the $w=w_1w_2\ldots w_n$ (one-line notation of $w$) to produce a pair of tableau $(P(w),Q(w))$. We refer the reader to, e.g., \cite[Chapter~7]{ECII} for details. However, as an example, $w=31524$ inserts as
\[\left(\tableau{3},\tableau{1}\right)\!\to\! 
\left(\tableau{1 & 3},\tableau{1 & 2}\right)\!\to\! 
\left(\tableau{1 & 3\\ 5 },\tableau{1 & 2\\ 3}\right)\!\to\!
\left(\tableau{1 & 3\\ 2& 5 },\tableau{1 & 2\\ 3 & 4}\right)\!\to\!
\left(\tableau{1 & 3\\ 2& 5\\ 4 },\tableau{1 & 2\\ 3 & 4\\ 5}\right)=(P(w),Q(w)).
\]
(The reader unfamiliar with the correspondence might find it a worthwhile exercise to decode what the bijection
is from the example, and furthermore to prove its correctness.) 
The following definition and exercise follow \cite{Lanini}.

\begin{definition}
$x, y \in S_n$ are in the same Kazhdan-Lusztig \emph{right cell} if $P(x)=P(y)$.
\end{definition}

\begin{hardexercise}\label{exer:Lanini}
(a) Given $x,y\in S_n$, there exist $v,w\in S_N$ for an $N\geq n$ such that $v,w$ are in the same
right cell, $v(i)=w(i)$ for $i\leq N-n$ and such that $x,y$ (classically) pattern 
embed respectively into $v,w$ in the last $n$ positions.

(b) If furthermore $x\leq y$ (Bruhat order), ${\mathcal N}_{v,w}\cong {\mathcal N}_{x,y}$.
\end{hardexercise}

Exercise~\ref{exer:Lanini} is used to produce infinitely many negative answers to questions  in combinatorial
representation theory \cite{Lanini}. Among these is the 0-1 conjecture; see Section~\ref{sec:KL} and
specifically Theorem~\ref{the01thing}.

\section{Combinatorial Commutative Algebra}\label{sec:4}

In this section we introduce concepts from combinatorial commutative algebra to study properties of Kazhdan-Lusztig ideals. We discuss results from \cite{WY:Grobner} that generalize those of \cite{Knutson.Miller:annals}
in the case of Schubert determinantal ideals. Our presentation follows a toy running example 
(Example~\ref{therunningexample} below).

\subsection{Gr\"obner bases}\label{sec:4.1}
We start with a quick summary of Gr\"obner bases, as found in, e.g., \cite{CLO}.
Let $R={\mathbb C}[x_1,\ldots,x_n]$ and $I\subseteq R$ be an ideal. 
\begin{definition}
A \emph{term order} $\prec$
on $R$  is a total order on the monomials $x^{\gamma}$ such that:
\begin{itemize}
\item  $1\prec x_i$ for $1\leq i\leq n$; and
\item if $x^{\alpha}\prec x^{\beta}$ then $x^{\alpha}\cdot x^{\gamma}\prec x^{\beta}\cdot x^{\gamma}$.
\end{itemize}
\end{definition}

\begin{example}
\emph{Pure lexicographic order} is the term order where $x^{\alpha}\succ x^{\beta}$
if $\alpha_i>\beta_i$ for the smallest $i$ such that $\alpha_i\neq \beta_i$ (if it exists).
\end{example}

\begin{definition}
 For $f\in R$, the \emph{initial term} with respect to the term order $\prec$, denoted
${\sf init}_{\prec}(f)$, is the $\prec$-largest term of $f$.
\end{definition}

\begin{definition}
The \emph{initial ideal} of $I$ with respect to the term order $\prec$ is
\[{\sf init}_{\prec}(I):=\{{\sf init}_{\prec}(f): f\in I\}.\]
\end{definition}

\begin{definition}
A generating set $g_1,g_2,\ldots,g_m$ of $I$ is a \emph{Gr\"obner basis} of $I$ with respect to the term order $\prec$ if
$\langle {\sf init}_{\prec}(g_i):i=1,\ldots,m\rangle={\sf init}_{\prec}(I)$.
\end{definition}
\emph{Buchberger's criterion} is a test for deciding if a generating set $\{g_1,\ldots,g_m\}$ is a Gr\"obner basis of $I$ with respect to the term order $\prec$. When iterated it gives Buchberger's algorithm for computing a Gr\"obner basis
for $I$.  

\begin{center}
\fbox{\begin{minipage}{30em}
 A Gr\"obner basis describes a ``flat'' degeneration
of $R/I$ to $R/{\sf init}_{\prec}(I)$. For SSPs (Definition~\ref{def:SSPname}),
$R/{\sf init}_{\prec}(I)$ can only be ``worse'' than $R/I$.
   \end{minipage}}
\end{center}

For a proper account of the above assertion we point to \cite{eisenbud:view}. 

\begin{example}[The running example]
\label{therunningexample}
Let $R={\mathbb C}[z_{11},z_{12},z_{13},z_{21},z_{22},z_{23},z_{31},z_{32},z_{33}]$ be the coordinate ring
of ${\sf Mat}_{3\times 3}$. Let $I_{\rm Ex.~\ref{therunningexample}}$ be generated by the $2\times 2$ minors of a generic $3\times 3$ matrix
$\left[\begin{matrix} 
z_{31} & z_{32} & z_{33}\\
z_{21} & z_{22} & z_{23}\\
z_{11} & z_{12} & z_{13}
\end{matrix}\right]$. Let $\prec$ be the pure lexicographic order with $z_{11}\prec z_{12}\prec z_{13}\prec z_{21}\prec z_{22} \prec z_{23}\prec z_{31}\prec z_{32} \prec z_{33}$.

Notice that the initial term of any minor is the ``antidiagonal term''. For example,
\[{\sf init}_{\prec}\left(\left|\begin{matrix} z_{21} & z_{22} \\ z_{11} & z_{12}\end{matrix}\right|\right)=z_{11}z_{22}.\] 
\end{example}

\begin{exercise}\label{runningisspecialcase}
Show that Example~\ref{therunningexample} is a special case of Schubert determinantal ideals 
and also a special case of Kazhdan-Lusztig ideals. 
\end{exercise}

\begin{exercise}\label{exer:Buchberger}
Prove that the defining minors of $I_{\rm Ex.~\ref{therunningexample}}$
are a Gr\"obner basis with respect to the term order $\prec$.\end{exercise}

Therefore,
\[{\sf init}_{\prec}(I_{\rm Ex.~\ref{therunningexample}})=\langle z_{11}z_{22}, z_{11}z_{23}, z_{11}z_{32}, z_{11}z_{33}, z_{12}z_{23}, z_{12}z_{33},
z_{21}z_{32}, z_{21}z_{33}, z_{22}z_{33}\rangle.\]

\begin{definition}
The \emph{radical} of an ideal $I$ is
$\sqrt{I}:=\{g\in R: g^k\in I \text{ \ for some $k\geq 1$}\}$.
An ideal is \emph{radical} if $I=\sqrt{I}$.
\end{definition}

\begin{center}
\fbox{\begin{minipage}{30em}
An ideal $I$ is radical if it has ``no hidden equations'': there does not exist 
$f\in R- I$ that vanishes on its zero locus $V(I)$. It is often difficult to prove that an ideal is
radical. Since being radical is a semicontinuous property, one method is to show ${\rm init}_{\prec}(I)$ is radical.
   \end{minipage}}
\end{center}

The first sentence in the principle above is \emph{Hilbert's Nullstellensatz}, which formally states that, if ${\mathcal I}(V(I))$ is the ideal of all polynomials that vanish on $V(I)$, then ${\mathcal I}(V(J))=\sqrt{J}$. If $V(I)=V(\sqrt{I}):=X$
we say that $I$ defines $X$ \emph{set-theoretically}.

\begin{example}
A non-radical ideal is $I=\langle x^2, xy, y^2\rangle\subset {\mathbb C}[x,y]$. For example, $x\not\in I$ but $x\in \sqrt{I}$.
Now, $V(I)=\{(0,0)\}$. Both $x$ and $y$ vanish on $V(I)$ but $x,y\not\in I$.\end{example}

\begin{definition}
A \emph{square-free monomial ideal} $I$ of  a polynomial ring ${\mathbb C}[x_1,\ldots,x_n]$ is one that is generated by square-free monomials.
\end{definition}
 
\begin{example} 
It is not obvious from the definition that $I_{\rm Ex.~\ref{therunningexample}}$ is radical.
However, square-free monomial ideals are clearly radical, hence 
$I_{\rm Ex.~\ref{therunningexample}}':={\sf init}_{\prec}(I_{\rm Ex.~\ref{therunningexample}})$ is radical,
and thus $I_{\rm Ex.~\ref{therunningexample}}$ is radical using the principle above.
\end{example}

\begin{example}
The method of proving an ideal $I$ is radical using Gr\"obner bases is sensitive to the choice of term order $\prec$.
For example let $I=\langle x_1 x_2-x_3^2\rangle\subset {\mathbb C}[x_1,x_2,x_3]$; this ideal is radical. 
If $\prec$ is the pure lexicographic order with $x_1\succ x_2\succ x_3$ then ${\sf init}_{\prec}(I)=\langle x_1 x_2\rangle$ is squarefree. On the other hand if $\prec^\prime$ is the pure lexicographic order with
$x_1\prec^\prime x_2 \prec^\prime x_3$ then ${\sf init}_{\prec^\prime}(I)=\langle x_3^2\rangle$ is not radical and the
above principle cannot be applied.
\end{example}

We now explain the generalization of Exercise~\ref{exer:Buchberger} to Kazhdan-Lusztig ideals. 
Let $\prec$ be the pure lexicographic term order on monomials in ${\bf z}^{(v)}$ induced by
\begin{equation}\label{eqn:diagterm}
z_{ij}\succ z_{kl} \text{\  if $j>l$, or if $j=l$
and $i<k$.}
\end{equation}  
\begin{theorem}[\cite{WY:Grobner}]
\label{thm:Grobnerthm}
The defining
minors from Definition~\ref{def:KLideal} form a Gr\"{o}bner basis with squarefree lead terms for $I_{v,w}\subseteq {\mathbb C}[{\bf z}^{(v)}]$ with respect to $\prec$. In particular, $I_{v,w}$ is radical.
\end{theorem}

E.~Neye \cite{Neye} has given another proof of Theorem~\ref{thm:Grobnerthm} together 
with a similar Gr\"obner basis result for the patch ideal of a Schubert variety.

\begin{exercise}\label{exer:KMgrob}
The Gr\"obner basis theorem of \cite{Knutson.Miller:annals} shows that the defining generators of the
Schubert determinantal ideal $I_w$ (Definition~\ref{def:Schubdetidealabc}) form
a Gr\"obner basis with respect to a pure lexicographic order  satisfying (\ref{eqn:diagterm}). Prove it,
assuming Theorem~\ref{thm:Grobnerthm}. 
\end{exercise}

An \emph{antidiagonal term order} is one that picks the antidiagonal term of any minor of a generic matrix.

\begin{problem}\label{prob:diagGrob}
Find a Gr\"obner basis for $I_w$ under an antidiagonal term order.\footnote{In the conventions of \cite{Knutson.Miller:annals} this corresponds to their diagonal term order, whereas our
diagonal term order is their antidiagonal term order (precisely because our convention places $z_{11}$ in the
southwest corner of the matrix).}
\end{problem}

A solution in the case that $w$ is covexillary is found in \cite{KMY}; see recent developments in 
 \cite{Hamaker.Pechenik.Weigandt, Klein, Klein.Weigandt}. 

\begin{problem}
Find a Gr\"obner basis for $I_{v,w}$ under some antidiagonal term order.
\end{problem}

\subsection{Prime decompositions}

\begin{definition}
An ideal $I$ is \emph{prime} if $ab\in I\implies a\in I$ or $b\in I$.
\end{definition}

\begin{definition}
A \emph{prime decomposition} of an ideal $J$ is $J=\bigcap_{t=1}^\ell J_t$, 
where each $J_t$ is a prime ideal and moreover
$J_s\not\subseteq J_t$ for $s\neq t$.
\end{definition}

An ideal $I\subset R$ will have a prime decomposition if and only if it is radical (which is the only case that concerns us in this chapter), although all ideals have something more general called a \emph{primary decomposition} by the Lasker-Noether theorem \cite[Section~4.8]{CLO}.

\begin{hardexercise}
\label{exer:Jan28.23qs}
Prove that $I_{v,w}$ is prime. 
\end{hardexercise}

\begin{exercise}\label{exer:primedecompabc}
Verify that the prime decomposition of $I_{\rm Ex.~\ref{therunningexample}}'$ is:
\begin{align}\label{eqn:Feb10abc}
I_{\rm Ex.~\ref{therunningexample}}'= & \langle z_{11}, z_{12}, z_{21}, z_{22}\rangle \cap 
\langle z_{11}, z_{12}, z_{21}, z_{33}\rangle\cap \langle z_{11}, z_{12}, z_{32}, z_{33}\rangle  \\ 
 & \langle z_{11}, z_{21}, z_{23}, z_{33}\rangle \cap \langle z_{11}, z_{23}, z_{32}, z_{33}\rangle 
 \cap \langle z_{22}, z_{23}, z_{32}, z_{33}\rangle.\nonumber
\end{align}
\end{exercise}

Geometrically, the prime decomposition (\ref{eqn:Feb10abc}) shows that 
$V(I_{\rm Ex.~\ref{therunningexample}}')$ is equidimensional (all irreducible components
are the same dimension), since all the prime ideals in the decomposition defines varieties of the same dimension. Equidimensionality is also a
property holding on closed subsets, so $I_{\rm Ex.~\ref{therunningexample}}$ itself is equidimensional (actually, $I_{\rm Ex.~\ref{therunningexample}}$ is prime and hence $V(I_{\rm Ex.~\ref{therunningexample}})$ is
irreducible, but that is non-obvious to show).

The reader can check (or get) their answer using {\tt Macaulay2}.\footnote{{\tt Macaulay2} code for exploring Kazhdan-Lusztig varieties is available at the authors' websites.}
The commands are:
\begin{verbatim}
R=QQ[z11,z12,z13,z21,z22,z23,z31,z32,z33, MonomialOrder=>Lex]
M=matrix({{z31,z32,z33},{z21,z22,z23},{z11,z12,z13}})
I=minors(2,M)
J=gb I
K=ideal leadTerm(J)
P=primaryDecomposition(K)
\end{verbatim}

A \emph{semistandard Young tableau} of shape $\tableau{ \ & \ \\ \ & \ }$ is a filling using $1,2,3$ such that the entries are
weakly increasing along rows, and strictly so along columns. The reader can check that there are six such
tableaux, namely,
\[\tableau{ 1 & 1\\ 2& 2},\  \tableau{1 & 1\\ 2& 3},\  \tableau{1 & 1\\ 3&3}, \ \tableau{1 &2 \\ 2& 3},\ 
\tableau{1 & 2\\ 3 & 3}, \ \tableau{ 2 & 2\\ 3 & 3},
\]
which is the number of prime components in the decomposition (\ref{eqn:Feb10abc}) of 
$I_{\rm Ex.~\ref{therunningexample}}'$!

\begin{exercise}\label{exer:coincidence}
Explain this coincidence of six from the previous paragraph. For each component in 
(\ref{exer:primedecompabc}), place a $+$ in matrix position $(i,j)$ in
in the $3\times 3$ grid if $z_{ij}$ appears as a generator. For example, associated to the component
$\langle z_{11}, z_{23}, z_{32}, z_{33}\rangle$ is the ``plus-diagram'' is 
$\left[\begin{matrix} \cdot & + & + \\ 
\cdot & \cdot & +\\
+ & \cdot & \cdot 
\end{matrix}\right]$. Give a natural bijection between the $6$ plus-diagrams and the $6$ tableaux. 
\end{exercise}

Exercise~\ref{exer:coincidence} is a special case of \cite[Theorem~B]{Knutson.Miller:annals} which concerns initial
ideals of $I_w$. There, the objects are not tableaux, but rather \emph{pipe dreams} naturally
label the prime components of the initial ideal of $I_w$. One also sees Exercise~\ref{exer:coincidence}
as a special case \cite{KMY}. In \cite[Theorem~3.2]{WY:Grobner} one obtains the analogous result for
$I_{v,w}$ using pipe dreams defined on the Rothe diagram $D(v)$.

\subsection{Simplicial complexes}
 \begin{definition}
 An (abstract) \emph{simplicial complex} $\Delta$
on the set $[n]$ is a collection of subsets $\{F\}$ that is closed under taking subsets, that is, if $F\in\Delta$ and
$F'\subseteq F$ then $F'\in \Delta$. Each subset $F\in \Delta$ is a \emph{face}. A maximal face under inclusion is a \emph{facet}.
\end{definition}

\begin{center}
\fbox{\begin{minipage}{30em}
The \emph{Stanley-Reisner correspondence} is the bijection between simplical complexes $\Delta$ and
square-free monomial ideals $I$ obtained by assigning to each minimal non-face $\{i_1,\ldots,i_d\}$ 
a generator $x_{i_1}x_{i_2}\cdots x_{i_d}$. This correspondence provides a dictionary between algebra
and simplicial topology.
   \end{minipage}}
\end{center}

\begin{exercise}\label{exer:SRFeb10abc}

(a) Prove that the Stanley-Reisner correspondence is indeed a bijection.

(b) Apply the Stanley-Reisner correspondence to the square-free monomial ideal
$I_{\rm Ex.~\ref{therunningexample}}'$. Show that the resulting simplicial complex
$\Delta_{\rm Ex.~\ref{therunningexample}}$ has
six facets. How does this correspond to the six tableaux
from Exercise~\ref{exer:coincidence}?
\end{exercise}

The simplicial complex of Exercise~\ref{exer:SRFeb10abc} is an example of a \emph{subword complex}  \cite{subword}. Instances of these complexes are the Stanley-Reisner complexes found in \cite{Knutson.Miller:annals, KMY, WY:Grobner}. 
A related but different notion is that of \emph{tableau complexes} \cite{KMY:tableaux}. 

\subsection{Multigradings and Hilbert series}\label{sec:4.2}
Our reference for this subsection is \cite{Miller.Sturmfels}.
 
\begin{definition}
A ${\mathbb Z}^r$-\emph{multigrading} on $R$ is defined by a multidegree map ${\sf deg}:{\mathbb N}^n\to {\mathbb Z}^r$ where ${\mathbb N}=\{0,1,2,\ldots\}$. This map is assumed to be additive, so, for all $\alpha,\beta\in\mathbb{N}^n$,
\[{\sf deg}(\alpha+\beta)=
{\sf deg}(\alpha)+{\sf deg}(\beta).\] 
The multigrading assigns the monomial $\prod_{i=1}^n x_i^{u_i}\!\in\!R$ the multidegree ${\sf deg}(u_1,\ldots,u_n)\in {\mathbb Z}^r$.
\end{definition}

 The additivity condition means that we have a decomposition
\[R=\bigoplus_{{\bf a}\in {\mathbb Z}^r} R_{\bf a}\]
where $R_{\bf a}$ is the vector space (over ${\mathbb C})$ spanned by monomials of multidegree ${\bf a}$, 
and this decomposition is \emph{graded}, meaning hat, if $f\in R_{\bf a}, g\in R_{\bf b}$ then $fg\in R_{{\bf a}+{\bf b}}$,
for all ${\bf a}, {\bf b}\in {\mathbb Z}^r$.

\begin{example}
The \emph{standard grading} is ${\sf deg}:{\mathbb N}^n\to {\mathbb Z}$ defined by 
\[{\sf deg}(u_1,\ldots,u_n)=u_1+u_2+\cdots+u_n.\] 
\end{example}

\begin{definition}\label{def:posmult555}
A multigrading ${\sf deg}:{\mathbb N}^n\to {\mathbb Z}^r$ on $R$ is \emph{positive} if $\dim_{\mathbb C}R_{\bf a}<\infty$ for all
${\bf a}\in {\mathbb Z}^r$.
\end{definition}

We only use positive multigradings in this chapter.
Definition~\ref{def:posmult555} is equivalent to a number of other conditions; 
see \cite[Theorem~8.6]{Miller.Sturmfels}.

\begin{definition}
A polynomial $f=\sum_{\alpha \in{\mathbb N}^r} c_{\alpha} x^{\alpha} \in R$ is \emph{homogeneous} if 
 $f\in R_{\bf a}$ for some ${\bf a}\in {\mathbb Z}^r$. An ideal $I$ of $R$ is \emph{homogeneous} if it is generated by homogeneous elements.
\end{definition}

Suppose $S=R/I$ where $I$ is homogeneous and $R$ is positively multigraded. 
For ${\bf a}\in {\mathbb Z}^r$, let $S_{\bf a}\subset S$ be the vector subspace spanned
by (equivalence classes) of monomials of degree ${\bf a}$. Thus
\begin{equation}
\label{eqn:Sismulti}
S=\bigoplus_{{\bf a}\in {\mathbb Z}^r}S_{\bf a}.
\end{equation}
It is true \cite[Section~8.1-8.2]{Miller.Sturmfels} that in this situation, $\dim_{\mathbb C}(S_{\bf a})<\infty$. Therefore the following definition makes sense:

\begin{definition}
The \emph{Hilbert series} of $S$ (with respect to a positive multigrading ${\sf deg}$) is
\[{\sf Hilb}(S;t)=\sum_{{\bf a}\in {\mathbb Z}^r} \dim(S_{\bf a})t^{\bf a},\]
where $t=\{t_1,\ldots,t_r\}$ and $t^{\bf a}:=t_1^{{\bf a}_1}\cdots t_r^{{\bf a}_r}$.
\end{definition}

\begin{theorem}[{\cite[Theorem~8.20]{Miller.Sturmfels}}]\label{thm:fractional}
Let $R={\mathbb C}[x_1,\ldots,x_n]$ be a positive multigraded ring with grading ${\sf deg}:{\mathbb N}^n\to {\mathbb Z}^r$, and let $I$ be a homogeneous ideal of $R$. Then
\[{\sf Hilb}(R/I;t)=\frac{{\mathcal K}(R/I;t_1,\ldots,t_r)}{\prod_{i=1}^n (1-t^{{\sf deg}(x_i)})},\]
where ${\mathcal K}(R/I;t_1,\ldots,t_r)\in {\mathbb Z}[t_1,\ldots,t_r]$.
\end{theorem}

For an explanation for why one sees the ``fractional form'' expression for Hilbert series, we need the notion of free resolutions
as discussed in the next section. See Exercise~\ref{exer:HilbexerJun26}.

\begin{definition}
Under the hypotheses of $R/I$, the polynomial ${\mathcal K}(R/I;t)$ is 
the \emph{$K$-polynomial}.  The \emph{multidegree} ${\mathcal C}(R/I;t)\in {\mathbb Z}[t_1,\ldots,t_r]$
is the polynomial obtained by taking the lowest (total) degree terms of ${\mathcal K}(R/I;1-t_1,\ldots,1-t_r)$.
\end{definition}

\begin{center}
\fbox{\begin{minipage}{30em}
While the $K$-polynomial is equivalent to the information encoded in the Hilbert series, the multidegree only tracks information
about the highest dimensional components of $V(I)\subset {\mathbb C}^n$.
   \end{minipage}}
\end{center}

\begin{exercise}\label{exer:hilbformulaszzz}

(a) If $R$ is standard graded, how many monomials are there of degree $k$?

(b)  Show that
\[{\sf Hilb}(R;t_1,\ldots,t_r)=\frac{1}{(1-t^{{\sf deg}(x_1)})\cdots (1-t^{{\sf deg}(x_n)})}.\]
\end{exercise}

Gr\"obner bases preserve multigraded Hilbert series:

\begin{theorem}[{\cite[Theorem~8.36]{Miller.Sturmfels}}]
\label{thm:preserve}
Let $R$ be a positive multigraded ring and $I$ a homogeneous ideal. Then ${\sf init}_{\prec}$ is (trivially) homogeneous with respect to the same grading, and 
\[{\sf Hilb}(R/I; t)={\sf Hilb}(R/{\sf init }_{\prec}(I);t).\]
\end{theorem}

\begin{exercise}\label{exer:2x2detex}
Let $R={\mathbb C}[z_{11},z_{12},z_{21},z_{22}]$ with the the multigrading 
${\sf deg}:{\mathbb N}^4\to {\mathbb Z}^4$ that assigns
$z_{ij}$ the multidegree $\vec e_i-\vec e_{2+j}$. Let 
$I=\langle z_{11}z_{22}-z_{12}z_{21}\rangle$ and $S=R/I$.

(a) Show that the multigrading is positive.

(b) Check that $I$ is homogeneous.

(c) What is a vector space basis for $S_{(1,0,-1,0)}$? How about $S_{(1,1,-1,-1)}$?

(d) Pick a term order $\prec$ such that ${\sf init}_{\prec}(z_{11}z_{22}-z_{12}z_{21})=z_{11}z_{22}$. Clearly,
$\{z_{11}z_{22}-z_{12}z_{21}\}$ is a Gr\"obner basis with respect to $\prec$. Let $I'={\sf init}_{\prec}(I)=\langle z_{11} z_{22}\rangle$. Let $S'=R/I'$. Confirm Theorem~\ref{thm:preserve} 
makes sense by computing the dimensions of $S'_{(1,0,-1,0)}$ and $S'_{(1,1,-1,-1)}$.

(e) Repeat (d) with a term $\prec$ such that ${\sf init}_{\prec}(z_{11}z_{22}-z_{12}z_{21})=z_{12}z_{21}$. Let
$I''$ be the initial ideal and $S''=R/I''$.

(f) Let $t=(x_1,x_2,y_1,y_2)$. Show that
\[{\sf Hilb}(S';t)=\frac{(1-\frac{x_1}{y_1})+(1-\frac{x_2}{y_2})-
(1-\frac{x_1}{y_1})(1-\frac{x_2}{y_2})}{
(1-\frac{x_1}{y_1})(1-\frac{x_1}{y_2})(1-\frac{x_2}{y_1})(1-\frac{x_2}{y_2})}\]
and 
\[{\sf Hilb}(S'',t)=\frac{(1-\frac{x_1}{y_2})+(1-\frac{x_2}{y_1})-
(1-\frac{x_1}{y_2})(1-\frac{x_2}{y_1})}{
(1-\frac{x_1}{y_1})(1-\frac{x_1}{y_2})(1-\frac{x_2}{y_1})(1-\frac{x_2}{y_2})}.\]
Notice ${\sf Hilb}(S',t)={\sf Hilb}(S'',t)$ (and hence, by Theorem~\ref{thm:preserve} both equal
${\sf Hilb}(S;t)$).
\end{exercise}

We now complete our discussion of Example~\ref{therunningexample}. This exercise is similar to, but more
complicated than, Exercise~\ref{exer:2x2detex}.

\begin{exercise}\label{runningexfinale}
Define a (positive) multigrading by ${\sf deg}(z_{ij})=\vec e_i- \vec e_{3+j}\in {\mathbb Z}^6$.  Here we
let $t=(x_1,x_2,x_3,y_1,y_2,y_3)$.

(a) Show that $I_{\rm Ex.~\ref{therunningexample}}$ is homogeneous with respect to this grading.

(b) Compute ${\sf Hilb}(R/I_{\rm Ex.~\ref{therunningexample}};t)$. (Hard)

(c) Show that ${\mathcal C}(R/I_{\rm Ex.~\ref{therunningexample}})=\sum_P {\sf wt}(P)$, where the sum is over
the six plus diagrams from Exercise~\ref{exer:coincidence} and 
\[{\sf wt}(P)=\prod_{\text{$+$ in position $(i,j)$}} x_i-y_j.\]
For instance, for the plus diagram $P$ depicted in Example~\ref{exer:coincidence}, 
${\sf wt}(P)=(x_1-y_1)(x_2-y_3)(x_3-y_2)(x_3-y_3)$.

That is, the multidegree is the generating series over the plus diagrams.
\end{exercise}

\begin{proposition}\label{lemma:funnygrading}
Let $\vec e_1,\ldots,\vec e_n$ be the standard basis vectors in $\mathbb{Z}^n$.  
Fix $v\in S_n$, and let $w\in S_n$ satisfy $v\leq w$. The ideal $I_{v,w}$ is homogeneous under
the multigrading 
${\sf deg}(z_{ij})=e_{v(j)}-e_{n-i+1}$ 
for each $z_{ij}$ in ${\bf z}^{(v)}$.
\end{proposition}

\begin{exercise}
Prove Proposition~\ref{lemma:funnygrading}.
\end{exercise}

The reader may ask where do the multigradings from Exercise~\ref{exer:2x2detex}/Exercise~\ref{runningexfinale} and Proposition~\ref{lemma:funnygrading} come from? For the one from Exercise~\ref{exer:2x2detex}, it comes from the action of the $({\mathbb C}^*)^2\times ({\mathbb C}^*)^2$ algebraic torus\footnote{which means an \emph{algebraic group} isomorphic to $(*{\mathbb C}^*)^r$
for some positive integer $r$} where the left factor acts by scaling rows of a $2\times 2$ matrix 
and the right factor acts by scaling columns (by inverse). One similarly has a $({\mathbb C}^*)^3\times ({\mathbb C}^*)^3$ action for Exercise~\ref{runningexfinale}. In the case of Proposition~\ref{lemma:funnygrading},
it comes from  
the left-multiplication action of $T$ on ${\mathcal N}_{v,w}$.
More generally, the action of an algebraic torus $T=({\mathbb C}^*)^r$ on a affine variety $V$ makes its coordinate ring ${\mathbb C}[V]=R/I$ a $T$-module. The degree of a variable $z\in\mathbb{C}[V]$ is precisely the torus character $\chi(t_1,\ldots,t_r)$ acting on the one-dimensional representation spanned by $z$.

In order to state a formula for the multigraded Hilbert series of $S=R/I_{v,w}$, we need to introduce a family
of polynomials of significant interest in algebraic combinatorics. Let ${\bf x}=\{x_1,x_2,\ldots\}$ and ${\bf y}=\{y_1,y_2,\ldots\}$ be two countable collections of indeterminates.
Let  ${\sf Pol}$ be the set of Laurent polynomials in ${\bf x},{\bf y}$ with integer coefficients.
\begin{definition}
The \emph{isobaric divided difference operator} is
\begin{align*}
\pi_i: & \ \ {\sf Pol}\to {\sf Pol}\\
f \ &  \mapsto  \frac{x_{i+1}f(\cdots, x_i, x_{i+1},\cdots)-x_i f(\cdots,x_{i+1},x_i,\cdots)}{x_{i+1}-x_i}.
\end{align*}
\end{definition}

\begin{exercise}
Verify that $\pi(f)$ is indeed in ${\sf Pol}$.
\end{exercise}

\begin{definition}\label{def:Grothen}
The \emph{Grothendieck polynomials} ${\mathfrak G}_w({\bf x}, {\bf y})$ are defined for each $w\in S_n$ by the following recurrence. If $w=w_0$ then
\[{\mathfrak G}_{w_0}:=\prod_{i+j\leq n}1-\frac{x_i}{y_j}.\]
Otherwise there exists $1\leq k<n$ such that $w(k)<w(k+1)$ and
\[{\mathfrak G}_w=\pi_i({\mathfrak G}_{wt_{i \ i+1}}).\]
\end{definition}

\begin{example}\label{exa:231g}
For $n=3$, 
\[{\mathfrak G}_{w_0}=\left(1-\frac{x_1}{y_1}\right)\left(1-\frac{x_1}{y_2}\right)\left(1-\frac{x_2}{y_1}\right).\]
Hence
\begin{align*}
{\mathfrak G}_{231} & =\pi_1 {\mathfrak G}_{w_0}\\
& = \frac{x_2\left(1-\frac{x_1}{y_1}\right)\left(1-\frac{x_1}{y_2}\right)\left(1-\frac{x_2}{y_1}\right)-x_1
\left(1-\frac{x_2}{y_1}\right)\left(1-\frac{x_2}{y_2}\right)\left(1-\frac{x_1}{y_1}\right)}{x_2-x_1}\\
& = \frac{(x_2-y_1)(x_1-y_1)}{y_2^2}.
\end{align*}
\end{example}

\begin{exercise}\label{exer:computetwogroth}
Compute ${\mathfrak G}_{312}$ and ${\mathfrak G}_{132}$. 
\end{exercise}

\begin{exercise}
(a) Prove that $\pi_i \pi_{i+1} \pi_i = \pi_{i+1} \pi_i \pi_{i+1}$ and $\pi_i \pi_j = \pi_j \pi_i$ for $|i-j|>1$. Now use 
Exercise~\ref{exer:additionalexerSymmetric}(a) to conclude that the definition of ${\mathfrak G}_w$ does not depend
on the choice(s) of $k$ in Definition~\ref{def:Grothen}.

(b) Show $\pi_i^2=\pi_i$.
\end{exercise}

Definition~\ref{def:Grothen} was introduced in 1982 by A.~Lascoux-M.-P.~Sch\"utzenberger to study Schubert calculus
of $GL_n/B$. Under appropriate specializations one obtains the \emph{Schubert polynomials} (either in both the ${\bf x}$ and ${\bf y}$ variables or just the ${\bf x}$ variables). These polynomials have
numerous non-cancellative formulas and have been the subject of significant interest in algebraic combinatorics up to
present day. We refer to \cite{Manivel, Kirillov} for some references and further background while noting that substantial amounts of even more recent work has been done (see, e.g., \cite{Klein.Weigandt} and the references therein).

The following shows that, after a substitution, the Grothendieck polynomials are $K$-polynomials for  
Kazhdan-Lusztig ideals.
\begin{theorem}
\label{thm:specializationa}
Let $R={\mathbb C}[z_{ij}, 1\leq i,j\leq n]$ and $I_{v,w}$ be the Kazhdan-Lusztig ideal for $v,w \in S_n$.
The multigraded Hilbert series polynomial of $S=R/I_{v,w}$ with respect to the positive multigrading from Proposition~\ref{lemma:funnygrading} is given by
\[{\sf Hilb}(R/I_{v,w},t_1,\ldots,t_n)=\frac{\Groth_{w_0w}(t_{v(1)},\ldots,t_{v(n)};t_{n},t_{n-1},\ldots,t_1)}{\prod_{1\leq i,j\leq n} (1-t_{v(j)}/t_{n-i+1})},\]
where in the denominator, the product is over all $(i,j)$ such that $z_{ij}\in {\bf z}^{(v)}$.
\end{theorem}

Theorem~\ref{thm:specializationa} reformulates part of \cite[Theorem~4.5]{WY:Grobner}. In \cite{WY:Grobner}, the statement is in terms of specializing ``unspecialized Grothendieck polynomials''. An 
advantage is that the multigradings for Proposition~\ref{lemma:funnygrading}, and
for Exercises~\ref{exer:2x2detex} and~\ref{runningexfinale} (which are Schubert determinantal ideals in light disguise), are derived from a specialization of another 
multigrading. In \cite{Knutson.Miller:annals}, the ``pipe dream'' combinatorial formulas for 
Grothendieck polynomials (viewed as a $K$-polynomial) and the Schubert polynomials (viewed as a multidegree) arise naturally from the Gr\"obner degeneration and the prime decomposition of the initial scheme. This
result (or a similar one from \cite{KMY} which is closer on point) generalizes Exercise~\ref{runningexfinale} which relates the ``plus diagrams'' to the multidegree for $R/I_{\rm Ex.~\ref{therunningexample}}$. In \cite{WY:Grobner} we do the same for the specializations of these
polynomials.

\begin{example}
Let $v=id, w=213$. Since $w_0w=231$, using the computation of ${\mathfrak G}_{231}$ from Example~\ref{exa:231g}, one obtains that
\[{\sf Hilb}(R/I_{id,231};t_1,t_2,t_3)=\frac{(t_3-t_1)(t_3-t_2)/t_3^2}{(1-\frac{t_1}{t_3})(1-\frac{t_1}{t_2})(1-\frac{t_2}{t_3})}.\] 
\end{example}

\begin{exercise}\label{sec5final}
Compute ${\mathcal K}(R/I_{132,132};t_1,t_2,t_3)$.
\end{exercise}

\section{Syzygies and (minimal) free resolutions}\label{sec:4.3}

In order to concretely define the properties of Schubert varieties found in the next section, we need the 
notion of free resolutions.
The general theory is covered in \cite{eisenbud:view}, more specifically in \cite{eisenbud:gos}, and in the
multigraded case in \cite{Miller.Sturmfels}; a recent survey paper is \cite{Threethemes}. We give an exposition of the theory at the level of generality needed in this chapter.

As in Section~\ref{sec:4},  we assume $R={\mathbb C}[x_1,\ldots,x_n]$, $I\subseteq R$ is an ideal; treat $S=R/I$ and $I$ as $R$-modules.
\begin{definition}
A \emph{free resolution} of $S$ is an exact sequence of homomorphisms of finitely generated free $R$-modules
\[\cdots\longrightarrow F_{i+1}\stackrel{\partial_{i+1}}{\longrightarrow}F_i\stackrel{\partial_i}{\longrightarrow}F_{i-1}
\stackrel{\partial_{i-1}}{\longrightarrow}\cdots\stackrel{\partial_{2}}{\longrightarrow} F_1
\stackrel{\partial_{1}}{\longrightarrow} F_0 \stackrel{\partial_{0}}{\longrightarrow} S \longrightarrow 0,\]
that is ${\rm im}(\partial_{i+1})=\ker(\partial_i)$ for $i\geq 0$. Each of the maps $\partial_i$ are the
\emph{differentials}.
\end{definition}

If $F_i$ has rank $\beta_i$ then we will think of $\partial_i$ as a $\beta_{i-1}\times \beta_{i}$ matrix
with entries from $R$.

\begin{definition}
A free resolution is \emph{finite} if it is of the form
\begin{equation}
\label{eqn:finitefree}
0 \longrightarrow F_n\stackrel{\partial_n}\longrightarrow \cdots\longrightarrow F_{i+1}\stackrel{\partial_{i+1}}{\longrightarrow}F_i\stackrel{\partial_i}{\longrightarrow}F_{i-1}
\stackrel{\partial_{i-1}}{\longrightarrow}\cdots\stackrel{\partial_{2}}{\longrightarrow} F_1
\stackrel{\partial_{1}}{\longrightarrow} F_0 \stackrel{\partial_{0}}{\longrightarrow} S \longrightarrow 0.
\end{equation}
\end{definition}

D.~Hilbert proved the celebrated result 
that, as restated in modern language, every $R$-module $S$ has a free resolution of length at most $n$~\cite{Hilbert.syzygy}. Thus, one can study $S$ by comparison
with free modules.

More can be said if one assumes, as we do now, that $R$ has a multigrading ${\sf deg}:{\mathbb N}^n \to
{\mathbb Z}^r$ and $I$ is homogeneous with
respect to that grading, so $S$ is multigraded, which means (\ref{eqn:Sismulti}) holds. In what follows
we use an innocuous ``accounting trick'' that forces the differentials to be degree-preserving maps:
\begin{definition}[Degree shift]\label{def:degshift}
For $\alpha\in {\mathbb Z}^r$ let $R(-\alpha)$ be the free $R$-module of rank $1$ such that 
$R(-\alpha)_\beta=R_{\beta-\alpha}$.
\end{definition}

We will write:
\begin{equation}
\label{eqn:Jun252022}
F_i=\bigoplus_{j=1}^k R(-\alpha^{i,j})
\end{equation}
where $k$ is the rank of $F_i$ and each $\alpha^{i,j}\in {\mathbb Z}^r$.

\begin{definition}
A finite free resolution (\ref{eqn:finitefree}) is \emph{multigraded} if each of the differentials $\partial_i$ are
multigraded $R$-homomorphisms, i.e., we additionally require that the maps are multidegree preserving.
\end{definition}

In order to discuss invariants of $R/I$ we need this definition:

\begin{definition}\label{def:minimal}
A finite free resolution \eqref{eqn:finitefree} is \emph{minimal} if there are no nonzero constant entries in
$\partial_i$ for each $i$.
\end{definition}

If $R$ is \emph{positively} multigraded, minimal free resolutions exist. 
Moreover, in this case, a minimal free resolution is indeed ``minimal'' in the following sense: the ranks of
the free modules in the complex are as small as possible. The next example explains the ideas
behind the definition, for a determinantal ideal.

\begin{example}\label{exa:22in23}
Let $F_0:=R={\mathbb C}[z_{11}, z_{12}, z_{13}, z_{21},z_{22},z_{23}]$ with the standard
grading. Let $I$ be the ideal generated by $2\times 2$ minors of a generic $2\times 3$ matrix. That is,
\[I=\langle z_{11}z_{22}-z_{12}z_{21}, z_{11}z_{23}-z_{13}z_{21}, z_{12}z_{23}-z_{13}z_{22}\rangle.\] 
For later reference, notice the generators are \emph{minimal}  (removing any generator changes the ideal). 
For now, there is the projection from $F_0=R\to S$. 

The kernel of this map is $I$. Let us encode this
fact by defining $F_1=R(-2)^{\oplus 3}$ and setting 
\[\partial_0(f_1,f_2,f_3)=f_1(z_{11}z_{22}-z_{12}z_{21})+f_2(z_{11}z_{31}-z_{21}z_{13})+f_3(z_{12}z_{23}-z_{22}z_{31})\in R.\]
The point of the degree shift ``accounting trick'' is that if $(f_1,f_2,f_3)$ are homogeneous polynomials of degree $d$,
i.e., $(f_1,f_2,f_3)\in (R(-2)^{\oplus 3})_{d+2}$ then their image is in $R_{d+2}$. 
Finally, for the record we encode $\partial_0$ as the $1\times 3$ matrix
\[\partial_0=[z_{11}z_{22}-z_{12}z_{21} \  \ z_{11}z_{31}-z_{21}z_{13} \ \ z_{12}z_{23}-z_{22}z_{31}]\]
Indeed, ${\rm im}(\partial_0)=I$ as desired.

Now, ${\rm ker}(\partial_1)$ is non-trivial. That is, there are algebraic relations among the columns of $\partial_0$.
Two relations are
\[z_{23}\cdot (z_{11}z_{22}-z_{12}z_{21})-z_{22}\cdot (z_{11}z_{31}-z_{21}z_{13})+z_{21}\cdot (z_{12}z_{23}-z_{22}z_{31})=0,\] 
\[ -z_{13}\cdot (z_{11}z_{22}-z_{12}z_{21})+z_{12} \cdot (z_{11}z_{31}-z_{21}z_{13})-z_{21}\cdot (z_{12}z_{23}-z_{22}z_{31})=0.\]
These are called the \emph{first order syzygies}.  

Similarly, we encode these relations as columns of a matrix $\partial_1:=\left[\begin{matrix} z_{23} & -z_{13}\\ -z_{22} & z_{12}\\ z_{21} & -z_{21}\end{matrix}\right]$ thought of as a map from $F_2=R(-3)^{\oplus 2}$ to $F_1=R(-2)^{\oplus 3}$. The reader can convince themselves that ${\rm im}(\partial_1)={\rm ker}(\partial_0)$ and that $\partial_1$ too is degree-preserving.

Finally, there are no second order syzygies in this case, that is, no algebraic relations between the columns of
$\partial_1$. Hence $\ker(\partial_1)$ is just $(0,0)$. Thus if we define $\partial_2$ to be the zero map, we
have a graded finite free resolution
\[0\stackrel{\partial_2}{\longrightarrow}F_2=R(-3)^{\oplus{2}}\stackrel{\partial_1}{\longrightarrow}
F_1=R(-2)^{\oplus{3}}\stackrel{\partial_0}{\longrightarrow} F_0=R \longrightarrow S\longrightarrow 0.\]
\end{example}

Hilbert~\cite{Hilbert.syzygy} proves that this process of determining syzygies, and second order syzygies, followed
by third order syzygies, and so on, always terminates after $n$ steps, giving a graded free-resolution. 

\begin{center}
\fbox{\begin{minipage}{30em}
The length of a minimal free resolution, and even the ranks (and degree shifts) of free modules that appear are
all invariant, assuming minimal choices are made throughout.   \end{minipage}}
\end{center}

\begin{exercise}
Let $I$ now be the ideal generated by $2\times 2$ minors of $\left[\begin{matrix} z_{21}  & z_{22} & z_{23} & z_{24}\\ z_{11} & z_{12} & z_{13} & z_{14}\end{matrix}\right]$. Here one does see second order syzygies. 

(a) Find a minimal free resolution of $S=R/I$.

(b) Confirm it using {\tt Macaulay 2} as follows:

\begin{verbatim}
R=QQ[z11,z12,z13,z14,z21,z22,z23,z24]
M=matrix{{z21,z22,z23,z24}, {z11,z12,z13,z14}}
rs=res minors(2,M)
rs.dd
\end{verbatim}
\end{exercise}

Let us summarize with the following theorem (see \cite[Section~8.3]{Miller.Sturmfels}):
\begin{theorem}\label{Hilbtheorem123}
If $R={\mathbb C}[x_1,\ldots,x_n]$ is positively graded and $I$ a homogeneous ideal. There is a minimal
finite free resolution of the form
\[0\to \bigoplus_j R(-\alpha^{n,j})^{\oplus \beta_{n,j}}\to \bigoplus_j R(-\alpha^{n-1,j})^{\oplus\beta_{n-1,j}}\to \cdots \to \bigoplus_j R(-\alpha^{0,j})^{\oplus \beta_{0,j}}\to R\to R/I\to 0\]
where $n, \alpha^{i,j}\in {\mathbb Z}^r, \beta_{i,j}\in {\mathbb N}$ only depend on $R/I$. 
\end{theorem}

\begin{definition}
For fixed $i$, $\beta_i:=\sum_j \beta_{i,j}$ are the \emph{Betti numbers}. Each $\beta_{i,j}$ is a \emph{graded
Betti number}.
\end{definition}

The next exercise explains the appearance of the ``fractional form'' for Hilbert series (Theorem~\ref{thm:fractional}):
 
\begin{exercise}\label{exer:HilbexerJun26}

(a) Determine ${\sf Hilb}(R(0);t_1,\ldots,t_r)$, ${\sf Hilb}(R(-\alpha);t_1,\ldots,t_r)$, and 
${\sf Hilb}(F_i;t_1,\ldots,t_r)$ where $F_i$ is as in \eqref{eqn:Jun252022}.

(b) Suppose $S$ is standard graded and (\ref{eqn:finitefree}) is graded. Give a formula for ${\sf Hilb}(S,t)$
as an alternating sum of ${\sf Hilb}(F_i,t)$. 

(c) Now use (a) to show that 
${\sf Hilb}(S,t)=\frac{K(t)}{(1-t)^n}$ 
for some
$K(t)\in {\mathbb Z}[t]$.

(d) Repeat (b) and (c) for a general multigraded ring. 
\end{exercise}

\begin{exercise}\label{macexernnn}
Use {\tt Macaulay2} to compute the graded betti numbers when $I=I_{v,w}$ is a Kazhdan-Lusztig
ideal, with respect to the positive multigrading from  Proposition~\ref{lemma:funnygrading}.
\end{exercise}
By Theorem~\ref{Hilbtheorem123}, the (graded) Betti numbers are invariants of $S=R/I$ (as an $R$-module).
Their importance can be expressed as follows:

\begin{center}
\fbox{\begin{minipage}{30em}
Many important properties of $R/I$ or $V(I)$ are encoded by the (graded) Betti numbers.
  \end{minipage}}
\end{center}

This principle will become evident in Section~\ref{sec:6}. This motivates the following problem:

\begin{problem}
\label{prob:betti}
Determine the Betti numbers (or better yet) a description of the  minimal free resolution for Schubert determinantal ideals (or more generally, Kazhdan-Lusztig varieties).
\end{problem}

Problem~\ref{prob:betti} seems very difficult in general. The \emph{Betti table} is a standard way of organizing the 
graded betti numbers. For Schubert determinantal ideals we only know the number of rows of the 
Betti table (for the standard grading) \cite{PSW} and (implicitly) the values of the first column of this table \cite{GaoYong}. The latter understanding comes from the fact that
the minimal generators of $I_w$ are determined in \emph{ibid}; see Exercise~\ref{exer:minimal}.
Only the special case of $k\times k$ minors in a $m\times n$ matrix (over characteristic $0$) is entirely solved; see \cite{Lascoux:det, Weyman}.  A first approximation and possible building block is provided by the Schubert complexes of S.~Sam~\cite{Sam.Schubert}.
In our results about  measures of singularities we do not actually appeal to the minimal free resolution or the Betti numbers, although we will use the existence to define the problems.

\section{Singularity measures}\label{sec:6}

\begin{definition} A \emph{local ring} is a ring $R$ with a unique maximal ideal ${\mathfrak m}$; its
\emph{residue field} is $\Bbbk=R/{\mathfrak m}$.  We sometimes denote the local ring by $(R,\mathfrak{m},\Bbbk)$ to keep track of all the information in the notation.
\end{definition} 

Let $X$ be a complex variety and $p\in X$. The local ring of $X$ at $p$, denoted ${\mathcal O}_{X,p}$,
consists of the ring of germs
of regular functions defined in some neighborhood of $p$ and regular at $p$, the maximal ideal ${\mathfrak m}={\mathfrak m}_p$ of regular functions vanishing at $p$, and ${\Bbbk}={\mathbb C}$. It captures the local
behavior of $X$ at $p$ and is an isomorphism invariant of $X$ and the point $p$; see \cite[I.3]{Hartshorne} for
precise definitions. For those readers familiar with these concepts, we do define
properties of algebraic varieties in terms of the local ring, appealing to general references such as \cite{eisenbud:view, Bruns-Herzog}. However, we mostly give equivalent
definitions, in the case of Schubert varieties, in terms of Kazhdan-Lusztig ideals. Then we proceed to state 
theorems and conjectures that solve (P1) or (P2) using interval pattern avoidance.

\subsection{Smoothness}
  
\begin{definition}
The \emph{Zariski cotangent
space} at $p$ is ${\mathfrak m}_p/{\mathfrak m}_p^2$.
\end{definition}
The cotangent space is a vector space over ${\Bbbk}$ (which is
${\mathbb C}$ in our case). 
\begin{definition}
The \emph{Zariski tangent space} at $p$ in $X$ is $\left({\mathfrak m}_p/{\mathfrak m}_p^2\right)^*$ (vector space
dual).
\end{definition}
\begin{definition}
$p\in X$ is \emph{smooth} if $\dim_{{\mathbb C}}(\left({\mathfrak m}_p/{\mathfrak m}_p^2\right)^*)=\dim_{{\mathbb C}}X$.
\end{definition}
Since $\dim_{\mathbb C}(X_w)=\ell(w)$, 
$X_w$ is smooth at $p=E^{(v)}_\bullet$ if and only if $\dim_{\mathbb C} \left({\mathfrak m}_{p}/{\mathfrak m}_{p}^2\right)^*=\ell(w)$.

\begin{exercise}
\label{exer:Nsmooth}
Prove that $X_w$ is smooth at $E^{(v)}_\bullet$ if and only if ${\mathcal N}_{v,w}$ is smooth at ${\bf 0}$.
\end{exercise}

We restate the following characterization of V.~Lakshmibai--B.~Sandhya
\cite{Lak} mentioned in the introduction:

\begin{theorem}[\cite{Lak}]
\label{thm:3412-4231}
$X_{w}$ is smooth if and only if $w$ avoids the patterns
$3412$ and $4231$.
\end{theorem}

We refer to \cite{BL00} for a summary of other criteria for 
determining if $X_w$ is smooth.

The following exercise gives another definition of smoothness of $X_w$ at $E^{(v)}_\bullet$:
\begin{exercise}[Jacobian criterion]\label{Jacxyz}
Given generators $g_1,g_2,\ldots,g_t$ of $I_{v,w}$, let ${\bf J}$ be the
$t\times \ell(w_0v)$-size matrix 
${\bf J}=(\frac{\partial g_i}{\partial x_j})$ where $x_1,x_2,\ldots,x_{\ell(w_0v)}$ is some ordering of the 
indeterminates $z_{ab}$ in ${\bf z}^{(v)}$. Then $X_w$ is smooth at $E^{(v)}_\bullet$ if and only if ${\bf J}$ is full rank when
evaluated at ${\bf 0}$.
\end{exercise}

\begin{exercise}\label{exer:smallcounter}
Apply the Jacobian criterion of Exercise~\ref{Jacxyz} to $I_{1234,3412}$ and $I_{1234,4231}$ and confirm that $X_{3412}$ and $X_{4231}$ are singular.
\end{exercise}

\begin{definition}
The \emph{singular locus} of a variety $X$ is 
\[{\rm sing}(X):=\{p\in X:  \dim(\left({\mathfrak m}_p/{\mathfrak m}_p^2\right)^*)>\dim X\}.\]
\end{definition}
If $Y\subset GL_n/B$ is closed and stable under the left-multiplication action by $B$ then
\[Y=\bigcup_{u\in I} X_{u},\]
for some $I$. In particular, $Y={\rm sing}(X_w)$ is closed and $B$-stable.

The first combinatorial criterion describing ${\rm sing}(X_w)$ was given by V.~Lakshmibai--C.~S.~Seshadri~\cite{LS84}:
\begin{theorem}[\cite{LS84}]
\label{thm:firstsingcriterion}
$X_w$ is smooth at $e_v$ if and only if 
\[{\mathcal R}(v,w):=\{(i,j): v<v \, t_{ij} \leq w\}=\ell(w)-\ell(v).\]
\end{theorem}

As mentioned in the Introduction, a pattern avoidance description of 
${\rm sing}(X_w)$ was conjectured by V.~Lakshmibai--B.~Sandhya \cite{Lak}
and proved independently by \cite{billey.warrington,Cortez, KLR,
  manivel1}.  In \cite{WY:governing} we restated the result in terms of interval pattern avoidance. 
  Let ${\mathcal I}_{\mathrm{singular}}$ be the set of pairs $(v,w)$ such that $X_w$ is singular at $e_v$.
  Below, the segment ``$j\cdots i$'' means $j, j-1,
j-2,\ldots, i+1,i$ (if $j<i$ then the segment is
empty).

\begin{theorem}
\label{thm:singularlocus}
The order ideal $\mathcal{I}_{\mathrm{singular}}$ in the poset $({\mathfrak S},\prec_I)$ is minimally
generated by the collection of these families of intervals:
\begin{enumerate}
\item $\big[(a+1)a\cdots 1(a+b+2)\cdots(a+2),\ \ \ (a+b+2)(a+1)a\cdots 2 (a+b+1)\cdots(a+2) 1\big]$ for all integers $a,b>0$.
\item $\big[(a+1)\cdots 1 (a+3) (a+2) (a+b+4)\cdots(a+4),\ \ \ (a+3)(a+1)\cdots 2(a+b+4)1(a+b+3)\cdots(a+4)(a+2)\big]$ for all integers $a,b\geq0$.
\item $\big[1(a+3)\cdots2(a+4), \ \ \  (a+3)(a+4)(a+2)\cdots312 \big]$ for all integers $a>1$.
\end{enumerate}
\end{theorem}

\noindent
\emph{Techniques used in the proofs of Theorem~\ref{thm:singularlocus}:} The proofs by
\cite{billey.warrington, KLR, manivel1} are essentially combinatorial and reduce to Theorem~\ref{thm:firstsingcriterion} or an earlier proof by V.~Gasharov of the sufficiency of the conditions.
The proof in \cite{Cortez} is qualitatively different than the others. It is geometric and proceeds by 
constructing partial resolutions of singularities of the Schubert variety.

\begin{exercise}\label{exer:singlocFeb24abc}
(a) Use Theorem~\ref{thm:singularlocus} to show that
\[{\rm sing}(X_{461253})=X_{142653}\cup X_{241365}\cup
X_{143265}.\] 
Now do the same with Theorem~\ref{thm:firstsingcriterion}; compare and contrast.

(b) Determine ${\rm sing}(X_{523614})$.
\end{exercise}

\begin{exercise}\label{exer:isolatedsing}
Can ${\rm sing}(X_w)=X_{id}$? (That is, can $X_w$ have an isolated singularity?)
\end{exercise}

\begin{hardexercise} \label{exer:proveordisprove}
Prove or disprove: ${\mathcal N}_{v,w}\cong {\mathbb C}^{\ell(w)-\ell(v)}$ if and only if $X_w$ is smooth at
$E^{(v)}_\bullet$. 
\end{hardexercise}

\begin{exercise}
\label{exer:Oct23abc}
Prove that $w$ avoids the families in Theorem~\ref{thm:singularlocus} if and only if $w$ is $3412$ and $4231$
avoiding.\footnote{Theorem~\ref{thm:3412-4231} and Exercise~\ref{exer:Oct23abc} shows that when
${\mathcal P}=$``singular'' something special happens. The
set of permutations appearing as the top element of intervals in
${\mathcal I}_{{\rm singular}}$ is the order ideal generated by $4231$
and $3412$ in the partial order given by \emph{classical} pattern
avoidance, where ``$u$ is smaller than $v$'' if $u$ classically embeds
into $v$. }
\end{exercise}

\subsection{Local complete intersection}

Suppose $S$ is a commutative ring. We refer to \cite{Bruns-Herzog}:
\begin{definition}
A \emph{regular sequence} of $S$ is 
a sequence $s_1,s_2,\ldots,s_d\in S$ such that
$s_i$ is not a zero-divisor on $S/(s_1,\ldots,s_{i-1})$ for $i=1,2,\ldots,d$. 
\end{definition}
\begin{definition}
A local ring $(R,{\mathfrak m}, \Bbbk)$ is a \emph{local complete intersection} (``lci'' for short) if there is a regular local ring $(S,{\mathfrak n})$ (that is,
$\dim_{S/n}{\mathfrak n}/{\mathfrak n}^2$ is the Krull dimension of $S$) and a regular sequence $s_1,\ldots,s_d$ of $S$ such that
$R=S/(s_1,\ldots,s_d)$.
\end{definition}
 
 \begin{definition}
An algebraic variety $X$ is a \emph{local complete intersection} if each local ring ${\mathcal O}_{X,p}$ of $X$ is lci.
\end{definition}
\begin{example}
An affine algebraic variety $X$ of codimension $d$ is a \emph{complete intersection} if it can be cut out by $d$ many equations. Such a variety is also a local complete intersection.
\end{example}

The following exercise provides an alternate definition of lci for our purposes.

\begin{exercise}\label{exer:lcialt}
${\mathcal O}_{X_w,E^{(v)}_\bullet}$ is lci if and only if ${\mathcal N}_{v,w}$ is a complete intersection.
\end{exercise}

H.~Ulfarsson and the first author \cite{lci} have classified which $X_w$ are lci in terms of \emph{classical} pattern avoidance.

\begin{theorem}\label{lcichar}
$X_w$ is lci if and only if $w$ avoids $53241$, $52341$, $52431$, $35142$, $42513$, and $351624$.
\end{theorem}

To prove ``$\Leftarrow$'' of Theorem~\ref{lcichar}, supposing $w$ avoids the stated patterns,
it then suffices by Exercise~\ref{exer:lcialt} to show that ${\mathcal N}_{id,w}$ is a complete intersection
by describing the ${n\choose 2}-\ell(w)$ of generators of $I_{v,w}$. For the converse, the authors of \cite{lci} show that the 
points $E^{(u)}_\bullet$ in Conjecture~\ref{conj:lcilocus} below are not lci, and $w$ interval contains one of those intervals whenever $w$ contains one of the patterns above.

Determination of the non-lci locus of $X_w$ remains open; this was conjectured in \emph{loc.~cit.}:

\begin{conjecture}[{\cite[Section~7]{lci}}]
\label{conj:lcilocus}
The ideal ${\mathcal I}_{\text{non-lci}}$ in the poset
$({\mathfrak S},\prec_I)$.
 is generated by
\begin{enumerate}
\item $[(a+1)a\cdots 1(a+b+2)\cdots (a+2),(a+b+2)(a+1)a\cdots 2(a+b+1)\cdots (a+2)1]$, where $a,b>0$ and $a>1$ or $b>1$; and
\item $[(a+1)\cdots 1(a+3)(a+2)(a+b+4)\cdots (a+4),
(a+3)(a+1)\cdots 2(a+b+4)1(a+b+3)\cdots (a+4)(a+2)]$, where $a,b\geq 0$ and $a+b\geq 1$.
\end{enumerate}
as well as eleven exceptional cases: 
\[[21354, 52341], [132546, 351624], [421653, 642531], [326154, 635241],\] 
\[[215436, 526314], [215436,524613],  [143265, 364152], [143265, 461352], [215436, 526413],\] 
\[[143265, 463152], [2154376, 5274163].\]
\end{conjecture}

\begin{exercise}\label{exer:minimal}
(a) Prove that the Schubert determinantal ideal $I_w$ is generated by all
$r_{st}^w+1$ minors of $Z_{st}$ where $(s,t)\in E(w)$ (Fulton's essential set, as defined
in Exercise~\ref{exer:rotheess}(a).)

(b) Give an example to show that the generators from (a) are indeed fewer than the full set
of generators from Definition~\ref{def:Schubdetidealabc}.

(c) Show by example that the set of generators of $I_w$ from (a) is not minimal in general, 
i.e., $I_w$ is generated by a strictly smaller subset.
\end{exercise}

The next exercise is the content of \cite{GaoYong}:
 
\begin{hardexercise}\label{exer:hardminimal}
(a) Determine a minimal list of generators for $I_w$.

(b) Use a solution to (a) to classify which matrix Schubert varieties are lci.
\end{hardexercise}

A solution to Exercise~\ref{exer:hardminimal}(a) for the more general
case of Kazhdan-Lusztig ideals should resolve Conjecture~\ref{conj:lcilocus}.

\begin{problem}
Give a minimal list of generators for $I_{v,w}$.
\end{problem}

S.~Gao--Y.~Gao (private communication) have reported solutions to this problem (and proved Conjecture~\ref{conj:lcilocus}) in the special cases where $v$ is $123$-avoiding and where $w$ is covexillary. 

\subsection{Gorensteinness} 

For those readers who have the requisite preparation in commutative algebra, recall:
\begin{definition} A local ring $(R,\mathfrak{m}, \Bbbk)$ is 
\emph{Cohen-Macaulay} if ${\rm Ext}_R^i(\Bbbk, R)=0$ for $i\leq\dim R$.
It is \emph{Gorenstein} if, in addition, $\dim_{\Bbbk}
{\rm Ext}_R^{\dim R}(\Bbbk, R)=1$.
\end{definition} 
\begin{definition}
 A variety is Cohen-Macaulay
(respectively Gorenstein) if the local ring at every point is
Cohen-Macaulay (respectively Gorenstein). 
\end{definition}
A reference for both definitions is  \cite{Bruns-Herzog}.

  All Schubert varieties are Cohen-Macaulay. We can determine (or take as a definition of) Gorensteinness of Schubert varieties using the Kazhdan-Lusztig ideals as follows. 
\begin{proposition}\label{prop:GorFeb11bbb}
$X_w$ is Gorenstein at $E^{(v)}_\bullet$ if the last Betti number of the minimal free resolution of
$R/I_{v,w}$ (computed with respect to the positive grading from the natural $T$-action) is $1$.
\end{proposition}

\begin{example}[{\cite[Example~2.8]{WY:governing}}]
$X_{42513}\subseteq {\rm Flags}({\mathbb C}^5)$ is not Gorenstein (the reader can verify this, for example,
using Proposition~\ref{prop:GorFeb11bbb} and a solution to Exercise~\ref{macexernnn}).
Now, $42513$ embeds into ${\underline{526}}4{\underline{13}}$ at the indicated positions. Yet
$X_{526413}\subseteq {\rm Flags}({\mathbb C}^6)$ is Gorenstein. The conclusion is that 
it is impossible to characterize Gorenstein Schubert varieties purely using classical pattern avoidance. 
\end{example}

\begin{theorem}[\cite{WY:Goren, WY:governing}]
\label{thm:Gor_char}
The Schubert variety $X_w$ is Gorenstein if and only if $w$
avoids the following intervals
\begin{enumerate}
\item $\big[(a+1)a\cdots 1(a+b+2)\cdots(a+2), \ \ \ (a+b+2)(a+1)a\cdots 2 (a+b+1)\cdots(a+2) 1\big]$ for all integers $a,b>0$ such that $a\neq b$.
\item $\big[(a+1)\cdots 1 (a+3) (a+2) (a+b+4)\cdots(a+4),\ \ \ (a+3)(a+1)\cdots 2(a+b+4)1(a+b+3)\cdots(a+4)(a+2)\big]$ for all integers $a,b\geq0$, with either $a>0$ or $b>0$.
\end{enumerate}
Equivalently, $X_w$ is Gorenstein if and only if the generic points of
its singular locus are.
\end{theorem}

Observe that Theorem~\ref{thm:Gor_char}(1) is nearly Theorem~\ref{thm:singularlocus}(1) 
 and Theorem~\ref{thm:Gor_char}(2) is nearly Theorem~\ref{thm:singularlocus}(2).
 
\begin{conjecture}
\label{conj:nonGorlocus}
The order ideal $\mathcal{I}_{\mathrm{not \ Gorenstein}}$ in the poset
$({\mathfrak S},\prec_I)$.
 is generated by the families
(1) and (2) from Theorem~\ref{thm:Gor_char}.
\end{conjecture}

\begin{hardexercise} \label{exer:usemanivelcortez}
Prove that Conjecture~\ref{conj:nonGorlocus} is equivalent to the following claim: $E^{(v)}_\bullet\in X_w$ is non-Gorenstein
if and only if $v\leq v'$ where $X_{v'}$ is an irreducible component of ${\rm sing}(X_w)$ and $X_w$ is non-Gorenstein at $E^{(v')}_\bullet$. 
\end{hardexercise}

Conjecture~\ref{conj:nonGorlocus} is true for $n\leq 6$. N. Perrin proved~\cite{Perrin} that
Conjecture~\ref{conj:nonGorlocus} holds on the class of Grassmannian
Schubert varieties.  These are Schubert varieties $X_w$ where $w$ is cograssmannian.  His theorem also includes the case of minuscule
Schubert varieties in other types (see discussion of these varieties in Section~\ref{sec:7}).
More recently, work
of S.~Da Silva \cite{DaSilva} describes a ``Gorensteinization process'' (a partial resolution of
singularities) for a Schubert variety that might prove helpful towards Conjecture~\ref{conj:nonGorlocus}.

\begin{exercise}\label{exer:Feb11ppp}
In general, one has the following containments of properties of local rings:
\begin{center}
regular $\subset$ lci $\subset$ Gorenstein $\subset$ Cohen-Macaulay.
\end{center}

(a) Show that Theorem~\ref{thm:singularlocus}, Conjecture~\ref{conj:nonGorlocus}, and 
Conjecture~\ref{conj:lcilocus} are consistent with these containments.

(b) Give infinitely many examples of Gorenstein rings that are not lci.
\end{exercise}

\begin{definition}
The \emph{Cohen-Macaulay type} of $X_w$ at $E^{(v)}_\bullet$ is the Betti number $\beta_n$ for
$R/I_{v,w}$. The \emph{Cohen-Macaulay type} of $X_w$ is its Cohen-Macaulay type at $E^{(id)}_\bullet$.
\end{definition}

\begin{problem}\label{CMtypeprob}
Characterize which $X_w$ has CM-type $\geq k$.
\end{problem}

If the Cohen-Macaulay type is $1$ the Schubert variety is Gorenstein, hence the case $k=2$ in the 
Problem~\ref{CMtypeprob} is asking for a characterization of non-Gorensteinness, which is answered by
Theorem~\ref{thm:Gor_char}.

\begin{problem}
Determine the locus of points in $X_w$ at which the CM-type is $\geq k$.
\end{problem}

Similarly, the case $k=2$ is conjecturally answered by Conjecture~\ref{conj:nonGorlocus}.
\subsection{Factoriality}

\begin{definition}
A variety is \emph{factorial} if the local ring at every
point is a unique factorization domain.  
\end{definition}

Resolving a conjecture from \cite{WY:governing}, M.~Bousquet-M\'elou-S.~Butler \cite{Butler} 
characterized factorial Schubert varieties by the following theorem:
 
\begin{theorem}[\cite{Butler}]
The Schubert variety $X_w$ is factorial if and only if $w$ classically avoids
$4231$ and interval avoids $[3142,3412]$.
\end{theorem}

The considerations that led to Conjecture~\ref{conj:nonGorlocus} also
lead to the following conjecture.

\begin{conjecture}[{\cite[Conjecture~6.14]{WY:governing}}] \label{conj:nonfactorialconj} The order ideal
$\mathcal{I}_{\mathrm{not \ factorial}}$ in the poset
$({\mathfrak S},\prec_I)$.
 is generated by the following
families: \begin{enumerate} \item $\big[(a+1)a\cdots
1(a+b+2)\cdots(a+2), \ \ \ (a+b+2)(a+1)a\cdots 2 (a+b+1)\cdots(a+2)
1\big]$ for all integers $a,b>0$. \item $\big[(a+1)\cdots 1 (a+3) (a+2)
(a+b+4)\cdots(a+4), \ \ \ (a+3)(a+1)\cdots
2(a+b+4)1(a+b+3)\cdots(a+4)(a+2)\big]$ for all integers $a,b\geq0$.
\end{enumerate} \end{conjecture}

\begin{exercise}
Since regular local rings are unique factorization domains, every smooth variety is factorial.
Furthermore, all unique factorization domains are Gorenstein. Prove Conjecture~\ref{conj:nonfactorialconj}
is consistent with Conjecture~\ref{conj:nonGorlocus}.
\end{exercise}

We do not have an explicit method to check Conjecture~\ref{conj:nonfactorialconj}.

\subsection{Tangent cones and Hilbert-Samuel multiplicity}\label{sec:HSmult}

Informally, the \emph{degree} of a projective variety $X\subseteq {\mathbb P}^n$ is the number of points of
intersection of a ``generic'' plane of dimension equal to the ${\rm codim}(X)$. 

\begin{definition}
The \emph{Hilbert polynomial} $h_S$ of a standard graded ring 
$S={\mathbb C}[x_1,\ldots,x_n]/I(V)$ 
is the unique polynomial such that for $k$ sufficiently large 
\[h_S(k)=\dim_{\mathbb C} S_k\] 
where $S=\bigoplus_k S_k$ is the graded decomposition of $S$.
\end{definition}

The Hilbert polynomial exists, and moreover $\deg h_S=\dim(X)$. Also if one writes
\[h_S=a_d\frac{k^d}{d!}+\text{lower degree terms},\]
the leading coefficient $a_d$ is a positive integer. 

\begin{definition}
The \emph{degree} of $X$ is $a_d$.
\end{definition}

Let $R={\mathbb C}[x_1,\ldots,x_n]$ and let $I\subset R$ be an ideal defining an affine variety $X=V(I)$.

\begin{definition}
The \emph{projectivized tangent cone} $TC_p(X)$ at $p={\bf 0}$ to $X=V(I)$ is the projective variety of ${\mathbb P}^{n-1}$ defined by the (standard) homogeneous ideal $\overline I$
generated by the lowest degree \emph{forms} of every $f\in I$.
\end{definition}

\begin{definition}
The \emph{Hilbert-Samuel multiplicity} of $X=V(I)$ at $p={\bf 0}$ is the degree of $TC_p(X)$ in ${\mathbb P}^{n-1}$. 
\end{definition}

Now suppose $X$ is an arbitrary variety and $p\in X$. We define the projectivized tangent cone and the 
Hilbert-Samuel multiplicity of $X$ at $p$ by first choosing an affine open neighborhood around $p$ with coordinates
such that $p$ becomes ${\bf 0}$ and using the above definitions. If $R$ is the coordinate ring of the ambient affine space that the affine open neighborhood sits in, then we define ${\sf PS}_{p,X}(t)$ to be the graded Hilbert series of
$R/{\overline I}$. It is true that
\begin{equation}
\label{eqn:firstPS}
{\sf PS}_{p,X}(t)=\frac{H_{p,X}(t)}{(1-t)^{\dim(X)}},
\end{equation}
where $H_{p,X}(t)\in {\mathbb Z}[t]$ and moreover, $H_{p,Y}(1)$ is the Hilbert-Samuel multiplicity of $X$ at $p$; see, e.g., \cite[Theorem~5.4.15]{Kreuzer}.

$H_{p,X}(1)=1$ if and only if $p$ is a smooth point of $X$.  Larger values of $H_{p,X}(1)$ measure ``how 
singular'' $X$ is at $p$.

\begin{definition}
The \emph{Kazhdan-Lusztig tangent cone ideal} 
\[{\overline I_{v,w}}\subset R= {\mathbb C}[{\bf z}^{(v)}]\] 
is the ideal generated by all lowest degree terms of any $f\in I_{v,w}$.
\end{definition}

This can be explicitly computed using Gr\"obner bases \cite[Section~6.5]{WY:governing}. In Macaulay2, one may use the
function {\tt TangentCone}. Let $H_{v,w}(t)=H_{E^{(v)}_{\bullet},X_w}(t)$.

\begin{proposition}[{\cite[Section~6.5]{WY:governing}}]
$H_{v,w}(1)=\text{\ the degree of $V({\overline I}_{v,w})$}$.
\end{proposition}

\begin{problem}\label{multprob1ggg}
Give a combinatorial counting rule for $H_{v,w}(1)$.\footnote{A similar sounding, but different, solved problem
is to determine the degree of a Schubert variety in its projective embeddings; see \cite{Postnikov.Stanley} and the references therein. More generally, one can think about the Hilbert polynomial of such an embedding.}
\end{problem}

More generally, L.~Li and the second author conjecture \cite{LiYong1, LiYong2} that:

\begin{conjecture}\label{multprob2hhh}
\label{conj:tangentconeCM}
$R/{\overline I_{v,w}}$ is Cohen-Macaulay. Hence $H_{v,w}(t)\in {\mathbb N}[t]$.
\end{conjecture}

\begin{problem}\label{multprob3iii}
Assuming $H_{v,w}(t)\in {\mathbb N}[t]$ (\emph{e.g.}, Conjecture~\ref{conj:tangentconeCM} holds), give
a combinatorial counting rule for the coefficients of $H_{v,w}(q)$.
\end{problem}

\begin{center}
\fbox{\begin{minipage}{30em}
If $I_{v,w}$ is standard homogeneous then
${\overline I_{v,w}}=I_{v,w}$ (taking the tangent cone does nothing).  Therefore the multiplicity
problems when $I_{v,w}$ is standard homogeneous can be deduced easily from results about the Hilbert series of $R/I_{v,w}$ (Section~\ref{sec:4}).
   \end{minipage}}
\end{center}

This observation about multiplicity was pointed out to us by A.~Knutson who noted this is true whenever $v$ is $123$-avoiding.  By Exercise~\ref{exer:rotheess}, this
this includes the cases that $v$ is  cograssmannian. This corresponds to the cases of multiplicities of
Schubert varieties in $Gr_{k}({\mathbb C}^n)$ of for which there is earlier work by V.~Lakshmibai--J.~Weyman~\cite{LW90}, J.~Rosenthal--A.~Zelevinsky~\cite{Rosen.Zel}, V.~Kreiman-V.~Lakshmibai~\cite{Kreiman.Lakshmibai}, and C.~Krattenthaler~\cite{Krat}. 

In \cite{LiYong1,LiYong2}, Problem~\ref{multprob1ggg}, Conjecture~\ref{multprob2hhh}, and
Problem~\ref{multprob3iii} were resolved in the case that $w$ is covexillary. This case is interesting since $I_{v,w}$ is not homogeneous with respect to the standard grading. The argument proceeds by relating the Gr\"obner degeneration of ${\overline I_{v,w}}$ with respect to an unusual term order to the initial scheme of a matrix Schubert variety, as studied in \cite{KMY}.  Recently, D.~Anderson-T.~Ikeda-M.~Jeon-R.~Kawago
\cite{Jeon.mult} gave a new proof of these results.  Their proof geometrically explains why cograssmannian combinatorics
appears in the covexillary formulae of \cite{LiYong1,LiYong2}. 

Let ${\overline {\mathcal N}_{v,w}}={\rm Spec}(R/{ \overline I_{v,w}})$. D.~Fuchs--A.~Kirillov--S.~Morier-Genoud-V.~Ovsienko \cite{Fuchs.Kirillov} raised ($v=id$ case of) the following question:
\begin{problem} 
Which ${\overline {\mathcal N}_{v,w}}$ are isomorphic?
\end{problem}

\begin{exercise}~\label{exer:easyinverse}
Prove that ${\overline{\mathcal N}_{v,w}}\cong {\overline{\mathcal N}_{v^{-1},w^{-1}}}$. 
\end{exercise}

\begin{problem}
\label{problem:multgeqk}
Determine the generators of the ideal ${\mathcal I}_{\text{multiplicity $\geq k$}}$ in the poset
$({\mathfrak S},\prec_I)$.
\end{problem}

For $k=2$, Problem~\ref{problem:multgeqk} is solved by Theorem~\ref{thm:singularlocus} since
$X_w$ is smooth at $E^{(v)}_\bullet$ if and only if $H_{v,w}(1)=1$. A related result is that of  
K.~Meek \cite{Meek} who determines which Schubert varieties $X_w$ have multiplicity $\geq 3$ 
at $E^{(id)}_\bullet$ (and thus globally on $X_w$).

\subsection{Castelnuovo-Mumford regularity}

\emph{Castelnuovo-Mumford regularity} is a measure, in some sense, of the ``complexity'' of a graded
module.\footnote{This can be made precise in the sense that regularity gives bounds on the size of a Gr\"obner basis and hence on the algorithmic time complexity of various decision problems involving a module.} 
Suppose $R={\mathbb C}[x_1,\ldots,x_N]$ and $S=R/I$ for a standard graded homogeneous ideal $I\subseteq R$. As in Section~\ref{sec:4.3}, $S=R/I$ has a minimal free resolution
\[0\to \bigoplus_j R(-j)^{\beta_{i,j}}\to \bigoplus_j R(-j)^{\beta_{i-1,j}}\to \cdots \to \bigoplus_j R(-j)^{\beta_{0,j}}\to S\to 0.\]
Here $i\leq N$ and $R(-j)$ is the free $R$-module where degrees of $R$ are shifted by $j$ (Definition~\ref{def:degshift}).
 
\begin{definition}
${\rm Reg}(S):=\max\{j-i: \beta_{i,j}(S)\neq 0\}$.
\end{definition}
By Exercise~\ref{exer:HilbexerJun26},
\[{\sf PS}_{S}(t)=\frac{{\mathcal K}_{S}(t)}{(1-t)^{N}},\]
where ${\mathcal K}(S,t)\in {\mathbb Z}[t]$. If $S$ is Cohen-Macaulay, then 
\begin{equation}
\label{eqn:easyreg2}
{\rm Reg}(S)=\deg {\mathcal K}(S,t)-{\rm ht}_R(I),
\end{equation}
where ${\rm ht}_R(I)$ is the \emph{height} of $I$ in $R$. If $V(I)$ is equidimensional (which is true for Cohen-Macaulay modules) then 
${\rm ht}_R(I)$ is the codimension of $V(I)\subseteq {\mathbb C}^N$.

Work of J.~Rajchgot, Y.~Ren, C.~Robichaux, A.~St. Dizier, and A.~Weigandt \cite{RRRSW}
initiated the study of regularity of matrix Schubert varieties and linked it to the degrees of Grothendieck
polynomials. The results of \emph{loc.~cit.} determined the regularity for the case $w$ is cograssmannian. J.~Rajchgot-C.~Robichaux-A.~Weigandt \cite{RRW} extended their formula to
covexillary permutations as well as certain Kazhdan-Lusztig ideals $I_{v,w}$ that are homogeneous
with respect to the standard grading. They also correct (and prove) a regularity conjecture of 
M.~Kummini-V.~Lakshmibai-P. Sastry-C.~S.~Seshadri \cite{Kummini}. 
O.~Pechenik-D.~Speyer-A.~Weigandt \cite{PSW} have proved a formula for the regularity for any matrix Schubert variety.

The following problem was formulated in \cite{Y:CM}:
\begin{problem}
\label{regprob}
Determine a combinatorial rule for ${\rm Reg}(R/{\overline I_{v,w}})$. 
\end{problem}

This would determine the regularity of the tangent cone of $X_w$ at $E^{(v)}_\bullet$. A solution to Problem~\ref{regprob} would generalize \cite{RRRSW, RRW, PSW}.

If Conjecture~\ref{conj:tangentconeCM} holds, then one could apply (\ref{eqn:easyreg2}) and 
conclude

\begin{conjecture}
\label{conj:reg=deg}
${\rm Reg}(R/I_{v,w}')=\deg H_{v,w}$.
\end{conjecture}

Using the covexillary tableau formula of \cite{RRW}, together with
work with L.~Li \cite{LiYong1,LiYong2}, Problem~\ref{regprob} and
Conjecture~\ref{conj:reg=deg} are resolved for $w$ covexillary \cite{Y:CM}.

\subsection{Kazhdan-Lusztig polynomials} \label{sec:KL}

Our final measure is of a different flavor than the others we have considered in that it is a topological rather than algebraic measure and cannot be easily calculated from the Kazhdan--Lusztig ideal, although in principle it is determined by it. Yet, it is of such significance that we would be remiss to not discuss it.

 For each pair of permutations 
$v,w\in S_n$ with $v\leq w$ there is the
\emph{Kazhdan-Lusztig polynomial} $P_{v,w}(q)\in {\mathbb
N}[q]$. These polynomials first appeared in representation theory, rather than Schubert geometry, in terms of a
certain deformation of the group algebra of ${\mathfrak S}_n$ (and more generally that of any Coxeter
group). We will follow \cite{Humphreys:Coxeter}; a similar (but not identical) treatment can also be found in \cite{bjorner.brenti}.

\begin{definition}\label{def:Heckedef}
The \emph{Hecke algebra} ${\mathcal H}_{n-1}$ (of type $A_{n-1}$) is a free module over the ring
${\mathbb Z}[q, q^{-1}]$ with basis given by
$\{T_w: w\in W\}$. It has relations
\[T_{s_i}T_w=
\begin{cases}
T_{s_iw} & \text{if $\ell(s_i w)>\ell(w)$}\\
(q-1)T_w+qT_{s_i w} & \text{otherwise.}
\end{cases}\]
\end{definition}

\begin{exercise}\label{exer:Rpoly}
(a) Prove that if $s_{i_1}\cdots s_{i_{\ell}}$ is a reduced word for $w$ (in the sense of Exercise~\ref{exer:additionalexerSymmetric}) then $T_w=T_{s_{i_1}}\cdots T_{s_{i_{\ell}}}$.

(b) Prove $T_{id}$ is the identity "$1$" of ${\mathcal H}_{n-1}$.

(c) Prove that $(T_{s_i})^{-1}=\frac{1}{q}(T_s-(q-1)T_{id})$.

(d) Conclude there exist polynomials $R_{x,w}\in {\mathbb Z}[q]$ of degree $\ell(w)-\ell(x)$ such that
\[(T_{w^{-1}})^{-1}=(-1)^{\ell(w)}q^{\ell(w)}\sum_{x\leq w} R_{x,w}(q)T_x,\]
where ``$\leq$'' refers to Bruhat order. (These are called the \emph{Kazhdan-Lusztig $R$-polynomials}.)
\end{exercise}

Define the \emph{bar map} $\bar{ \cdot }:{\mathcal H}_{n-1}\to {\mathcal H}_{n-1}$ by sending $q\mapsto q^{-1}$ and  sending $T_{s_i}$ to $T_{s_{i}}^{-1}$. Obtain a semilinear map by extending bar additively.
By Exercise~\ref{exer:Rpoly}, $\overline{\overline{T_{s_i}}}=T_{s_i}$. Hence bar is a ring involution given the
following computational exercise (or see \cite[Section~7.7]{Humphreys:Coxeter}):

\begin{exercise}
Show $\overline{T_s T_w}=\overline{T_s}\  \overline{T_w}$.
\end{exercise}

\begin{exercise}\label{exer:expandedring}
(a) By manipulating the expression in Exercise~\ref{exer:Rpoly}(c), can you construct an element $C_{s_i}$ such that $\overline{C_{s_i}}=C_{s_i}$?

(b) Now repeat (a) after replacing ${\mathbb Z}[q,q^{-1}]$ with ${\mathbb Z}[q^{\frac{1}{2}},q^{-\frac{1}{2}}]$.
\end{exercise}

It is because of Exercise~\ref{exer:expandedring} that the Hecke algebra is defined often using ${\mathbb Z}[q^{\frac{1}{2}},q^{-\frac{1}{2}}]$ rather than ${\mathbb Z}[q,q^{-1}]$. Neither choice is ``standard'' and we will
assume the former convention below. Now we arrive at the
following theorem, the heart of Kazhdan-Lusztig theory:

\begin{theorem}[\cite{Kazhdan.Lusztig}]
\label{thm:KLthm}
For each $w\in S_n$ there is a unique element $C_w$ such that $\overline{C_w}=C_w$ and if one writes
\[C_w=(-1)^{\ell(w)}\left(q^{\frac{-1}{2}}\right)^{\ell(w)}\sum_{v\leq w} (-1)^{\ell(x)}q^{-\ell(x)}P_{v,w}(q) T_v\]
then 
\begin{itemize}
\item[(i)] $P_{v,w}(q)\in {\mathbb Z}[q]$
\item[(ii)] $P_{w,w}(q)=1$
\item[(iii)] $\deg P_{v,w}(q)\leq \frac{\ell(w)-\ell(v)-1}{2}$ if $v<w$; and
\item[(iv)] $P_{v,w}(q)=0$ if $v\not\leq w$.
\end{itemize}
\end{theorem}

\begin{exercise}
(a) Confirm that the $C_{s_i}$ in Theorem~\ref{thm:KLthm} agrees with Exercise~\ref{exer:expandedring}(b).

(b) Prove that $\{C_w:w\in S_n\}$ form a ${\mathbb Z}[q^{\frac{1}{2}},q^{\frac{-1}{2}}]$-basis of ${\mathcal H}_{n-1}$. 

(c) There is a simpler looking formulation of Theorem~\ref{thm:KLthm}. Prove there is a basis $C_w^{\prime}$
of ${\mathcal H}_{n-1}$ (over ${\mathbb Z}[q^{\frac{1}{2}},q^{-\frac{1}{2}}]$) such that
\[C_{w}^\prime=\left(q^{\frac{-1}{2}}\right)^{\ell(w)}\sum_{v\leq w} P_{v,w}(q) T_v.\]
\end{exercise}

\begin{definition}
For $v,w\in S_n$, the \emph{Kazhdan-Lusztig polynomial} is the polynomial $P_{v,w}(q)$ from Theorem~\ref{thm:KLthm}.
\end{definition}
There is an algorithm for computing the $P_{v,w}$ assuming knowledge of the $R$-polynomials. There is also
an explicit recursive definition for $P_{v,w}(q)$ in \cite{Kazhdan.Lusztig}. 

For symmetric groups (or more generally, Weyl groups of finite type), it is 
a surprise from the above presentation that in fact $P_{v,w}\in {\mathbb N}[q]$. Also surprising is that this positivity has a geometric/topological explanation in terms of Schubert varieties. In particular,
 $P_{v,w}(q)$ is the Poincar\'{e} polynomial for the
local intersection cohomology of $X_w$ at $E^{(v)}_\bullet$ \cite{Kazhdan.Lusztig.Schubert}. Thus, $P_{v,w}(q)$ measures the
singularity of $E^{(v)}_\bullet\in X_w$.\footnote{The same proof works for any Weyl group $W$ where one has an underlying
Schubert geometry. However, for general Coxeter groups $W$ this is not true; relatively recently, positivity 
has been established algebraically using \emph{Soergel bimodules} \cite{Elias.Williamson}.} 
 
\begin{definition}\label{rationalsmooth}
$E^{(v)}_\bullet$ is \emph{rationally smooth} in $X_w$ if
$P_{v,w}(q)=1$.
\end{definition} 
Rational smoothness and the ordinary notion of smoothness in algebraic
geometry do not coincide. However they do for the symmetric groups (and for $ADE$ types). 
Thus 
\[P_{v,w}(q)=1 \iff \text{ $X_w$ is smooth at $E^{(v)}_\bullet$.}\]  
Dimensions of local intersection
cohomology groups are not in general (upper or lower) semicontinuous. However,
Irving~\cite{Irving} proved using a representation theoretic interpretation of Kazhdan-Lusztig polynomials that they behave in an upper semicontinuous manner on Schubert
varieties.  (A later more geometric proof appears in~\cite{braden.macpherson}.) 
Therefore, the coefficients of Kazhdan-Lusztig polynomials, can be analyzed
under the rubric of interval pattern avoidance. 

\begin{problem} 
Let $\mathcal{P}_{k,\ell}$ to be the property ``the coefficient of
$q^\ell$ in $P_{u,v}(q)$ is at least $k$'' (or equivalently
``$\dim_{\mathbb C} IH^\ell_{E^{(v)}_\bullet}(X_v)\geq k$'').  Determine
$\mathcal{I}_{\mathcal{P}_{k,\ell}}$ for various values of $k$
and~$\ell$.
\end{problem}

This is an longstanding, well-known open problem:

\begin{problem}
Give a combinatorial counting rule for the coefficients of
$P_{v,w}(q)$.
\end{problem}

To date, such rules are only known in a limited number of cases in type $A$ such as
\cite{BW-321hex, Las-vex}. The latter handles the covexillary case and the subcase where $w$ is Grassmannian.
Outside of type $A$ and similar results
are known for vexillary cases\footnote{In this case there are some analogies between the
Kazhdan-Lusztig polynomials and the $H$-polynomials; see the discussion in \cite{LiYong2}.}
and Schubert varieties of minuscule $G/P$'s. V.~V.~Deodhar's \emph{masks} give
a framework to approach the problem; see \cite{billey.warrington, Jones.Woo}.\footnote{The structure coefficients expanding products of the $C_w'$ basis of ${\mathcal H}_{n-1}$ in the same basis are also positive Laurent polynomials
in $q^{\pm\frac{1}{2}}$ \cite{Springer}.  Another important problem is to find a combinatorial rule that explains this.}

Unlike the $R$-polynomials of Exercise~\ref{exer:Rpoly}(d), the degree of $P_{v,w}(q)$ is not easily determined. It was for some time conjectured that the coefficient of the term of highest \emph{possible}
degree (namely, $\frac{\ell(w)-\ell(v)-1}{2}$) is either $0$ or $1$.  This is the ``0-1 Conjecture''.
\begin{theorem}[\cite{McWa}]\label{the01thing}
The 0-1 Conjecture is true for $n\leq 9$ but false for $n=10$. Specifically the top coefficient when 
$w= 10 \  578293461$ and $v=  54321\ 10\ 98764$ is $4$.
\end{theorem}

\begin{center}
\fbox{\begin{minipage}{30em}
Counterexamples about Kazhdan-Lusztig polynomials may only occur for large $n$.
   \end{minipage}}
\end{center}
Recent work of  \cite{Lanini} (see Exercise~\ref{exer:Lanini}) has
strengthened Theorem~\ref{the01thing}, producing infinitely many counterexamples where $v,w$ are in the same Kazhdan-Lusztig
right cell.  

It is always true that $P_{v,w}(0)=1$. P.~Polo has proved a striking negative result:

\begin{theorem}[\cite{Polo}]
\emph{Any} polynomial
$p(q)$ with nonnegative integer  coefficients and coefficient $1$  is the Kazhdan-Lusztig polynomial for some (explicitly constructed) pair
$v,w\in S_{1+\deg(p)+p(1)}$.
\end{theorem}

The following is a consequence of 
Theorem~\ref{thm:poset} (or rather the isomorphism of Kazhdan-Lusztig varieties that proves it).  It can also be observed from the method of T.~Braden--R.~Macpherson for calculating Kazhdan--Lusztig polynomials from sheaves on moment graphs~\cite{braden.macpherson}.

\begin{proposition}[{\cite{WY:governing}}]
\label{cor:KL_cor}
Suppose $[u,v]$ and $[x,w]$ are isomorphic because of an interval
pattern embedding.  Then $P_{x,w}(q)=P_{u,v}(q)$.
\end{proposition}

Lusztig's \emph{interval conjecture} is a stronger claim:

\begin{conjecture}
\label{conj:Lusztig} 
$P_{a,b}(q)=P_{v,w}(q)$ whenever the Bruhat order intervals $[a,b]$
and $[v,w]$ are isomorphic as posets.
\end{conjecture}
Conjecture~\ref{conj:Lusztig} is discussed
with further references in~\cite{brenti, bjorner.brenti}.  The conjecture would follow
from an affirmative answer to Problem~\ref{problem:KLiso}. In a recent development,
Artificial Intelligence has been employed to attack the conjecture; see \cite{Deepmind1, Deepmind2}.

Since Kazhdan--Lusztig polynomials are the local intersection cohomology Poincar\'e polynomials of Schubert varieties, there is a precise relationship between Kazhdan--Lusztig elements, considered as elements of the Hecke algebra, and the intersection cohomology sheaves of Schubert varieties, considered as elements of the category of perverse sheaves.  This is far beyond the scope of our survey; see, e.g., \cite{Springer, Rietsch} for further reading.

\section{Analogues for other Lie types} \label{sec:7}

\subsection{Background} Rather than working with the flag variety $GL_n/B$, we can replace the group $GL_n$ by an arbitrary complex semisimple Lie group (or indeed an arbitrary semisimple (affine) algebraic group) $G$.  While much of the general background applies, much less is known about the singularities of Schubert varieties in this more general setting.  Of particular interest are the other families of classical groups, which are $G=SO_{2n+1}$ (Type $B_n$), $G=Sp_{2n}$ (Type $C_n$), and $G=SO_{2n}$  (Type $D_n$).\footnote{The ``types'' refer to the Cartan-Killing classification of complex semisimple Lie algebras.}  (The group $SO_N$ behaves quite differently if $N$ is odd or even, so those cases are split up.)  These groups are commonly realized as subgroups of $GL_N$ (where $N=2n+1$ for type $B_n$ and $N=2n$ for types $C_n$ and $D_n$), but there are many possible choices.  We use a choice that has been standard in work on Schubert varieties since at least the work of S.~Billey--M.~Haiman \cite{Billey.Haiman}
and W.~Fulton--P.~Pragacz \cite{Fulton.Pragacz}, which have certain advantages that will be outlined below.

Rather than give general abstract definitions, we will give concrete definitions specifically for each of these families.  However, it helps to have an overall picture of the terminology to start; we refer to 
\cite[Chapter~2]{BL00}  for a summary of the generalities together with references. Given a group $G$, there is a Borel subgroup $B$ and opposite Borel subgroup $B_{-}$.  Our choice of how to realize these subgroups has the advantage that $B$ and $B_{-}$ are respectively the subgroups of upper and lower triangular matrices in $G$.  The \emph{generalized flag variety} is $G/B$.   For each group $G$, there is a finite group $W\subseteq G$ called the \emph{Weyl group} that plays the role of $S_n$ in the case $G=GL_n$.  Our choice of how to present $G$ also has the advantage that $W$ will be a subgroup of $S_N$, realized as permutation matrices. 

\begin{definition}
For each element $w\in W$, the \emph{Schubert cell} is
$X^\circ_w:=BwB/B$.
The \emph{opposite Schubert cell} is
$\Omega^\circ_w:=B_{-}wB/B$.
The \emph{Schubert variety} is
$X_w:=\overline X^\circ_w$,
and the \emph{Kazhdan--Lusztig variety} is
$\mathcal{N}_{v,w}:=X_w\cap \Omega^\circ_v$.
\end{definition}

The opposite big cell $\Omega_{id}^{\circ}$ is also an affine open neighborhood
of $G/B$. Hence Definition~\ref{defn:patch} makes sense.  Lemma~\ref{lemma:KL} also holds in this general setting.\footnote{There is no proof in \cite{Kazhdan.Lusztig}, but it follows easily from \cite[Sec. 28.1]{Humphreys:linearalggroups}.}  Therefore, as in the case $G=GL_n$, singularities of Schubert varieties can be studied by studying Kazhdan--Lusztig varieties and Kazhdan--Lusztig ideals.

\subsection{Kazhdan-Lusztig ideals for the classical groups}

\subsubsection{Type $B_n$} Here the group is $G=SO_{2n+1}$.  In general, one picks a nondegenerate symmetric bilinear form $Q:\mathbb{C}^{2n+1}\times\mathbb{C}^{2n+1}\rightarrow\mathbb{C}$.  
\begin{definition}
 $G=SO_{2n+1}\subseteq SL_{2n+1}$ is the group of linear transformations $M$ such that $Q(Mv,Mw)=Q(v,w)$ for all $v,w\in\mathbb{C}^{2n+1}$.  
 \end{definition}
 Different choices of $Q$ give conjugate subgroups. The standard choice for $Q$ is given as follows, where $e_i$ denotes the $i$-th basis vector:
$$Q(e_i, e_j)=\begin{cases} 1& i+j=2n+2 \\ 0 & i+j\neq 2n+2\end{cases}.$$
Let $J=[j_{ab}]$ be the matrix with $j_{ab}=1$ if $a+b=2n+2$ and $j_{ab}=0$ otherwise.  (Pictorially, this means $J$ has $1$'s on the main {\em antidiagonal} and $0$'s everywhere else.)  Then, for our specific choice of subgroup $SO_{2n+1}$, we have that a matrix $M\in SO_{2n+1}$ if and only if $M^TJM=J$.

With this choice of $SO_{2n+1}$, one identifies the Weyl group $W_{B_n}$ as the permutation group
$$W_{B_n}=\{w\in S_{2n+1}\mid w_0ww_0=w\},$$
where $w_0$ is the permutation $w_0=(2n+1)(2n)\cdots1$ of maximal length in
$S_{2n+1}$.  Equivalently, $w\in W_{B_n}$ if $w(i)+w(2n+2-i)=2n+2$ for all $i$, $1\leq i\leq n+1$.  In particular, $w(n+1)=n+1$.

The following is an easy combinatorial exercise about elements of $W_{B_n}$, considered as permutations:
\begin{exercise}\label{exer:Feb1abc}
(a) Given $1\leq i<j\leq 2n+1$ we have $w(i)>w(j)$ if and only if $w(2n+2-j)>w(2n+2-i)$.

(b) Given $1\leq i<j\leq 2n+1$ with $i+j=2n+2$, we have $w(i)>w(j)$ if and only if $w(i)>w(n+1)>w(j)$.
\end{exercise}
Exercise~\ref{exer:Feb1abc} gives an intuitive justification for the following definition.

\begin{definition}
The \emph{length} of $w\in W_{B_n}$ is
\[\ell_B(w):=\frac{\#\{1\leq i<j\leq 2n+1: i+j\neq 2n+2, w(i)>w(j)\}}{2}.\]  
\end{definition}

The actual justification for this definition is:
\begin{exercise}
Prove that $\ell_B(w)=\dim_{\mathbb C}(X_w)$ for the
Schubert variety $X_w$ in $SO_{2n+1}/B$.
\end{exercise}

We will use the notation $\ell_A(w)$ for the length of $w$ considered as a permutation in $S_{2n+1}$.

Since Lemma 3.2 holds in general, $B$ and $B_-$ are still the subgroups of upper and lower triangular matrices, and $W_{B_n}$ is an explicit set of permutation matrices, we can give coordinates for opposite Schubert cells as before, though we can only give set-theoretic equations for Kazhdan--Lusztig varieties in general.

Given $v\in W_{B_n}$, one identifies the opposite Schubert cell $\Omega^\circ_v$ with a subset of ${\sf Mat}_{2n+1\times 2n+1}$ as follows. 
View $\Omega_v^{A\circ}$ (formerly called $\Omega_v^\circ$ in Section~\ref{sec:opschubJan28}) as the affine subspace consisting of matrices $Z^{(v,A)}$ where $z_{n-v(i)+1,i}=1$, and $z_{n-v(i)+1,s}=0, z_{t,i}=0$ for $s>i$ and $t>n-v(i)+1$. Let ${\bf z}^{(v,A)}\subseteq {\bf z}$ be the unspecialized variables.  Now let $K$ be the ideal generated by the $(2n+1)^2$ entries of $(Z^{(v,A)})^TJ(Z_A^{(v,A)})-J$.

Now let $Z_{st}^{(v,A)}$ (formerly $Z_{st}^{(v)}$) be the southwest $s\times t$ submatrix of $Z^{(v,A)}$.  We can define $I'_{v,w}$ as the ideal of ${\mathbb C}[z^{(v,A)}]$ generated by all $r_{st}^w+1$ minors of $Z_{st}^{(v,A)}$ where $1\leq s,t\leq n$ and
$r_{st}^w$ is defined in Definition~\ref{Woosrank}.  Now we have the following.
\begin{definition}
The \emph{large set-theoretic type $B$ Kazhdan-Lusztig ideal} is
$I_{v,w}=I'_{v,w}+K$.\end{definition}

Actually, one can define an ideal in a smaller set of variables instead:
\begin{exercise}\label{exer:Feb1cde}
If $z_{ij}\in{\bf z}^{(v,A)}$, then $z_{v(j),v^{-1}(i)}\in{\bf z}^{(v,A)}$, and there exists a generator $f$ of $K$ that has $z_{ij}$ and $z_{v(j),v^{-1}(i)}$ as its only linear terms.  Without loss of generality assume that $j\leq v^{-1}(i)$, and let $f_{ij}$ denote this generator.  Then all variables $z_{i'j'}$ showing up in $f_{ij}$ have $j'>j$ or both $j'=j$ and $i'\geq i$.
\end{exercise}

In light of Exercise~\ref{exer:Feb1cde}, we let ${\bf z}^{(v)}\subseteq {\bf z}^{(v,A)}$ to be the set of unspecialized variables $z_{ij}$ with $j>v^{-1}(i)$.  Then we let
$Z^{(v)}$ be the matrix constructed from $Z^{(v,A)}$ by recursively substituting $f_{ij}-z_{ij}$ (or $f_{ij}/2 - z_{ij}$ if $i=v(j)$) for $z_{ij}$ whenever $j\leq v^{-1}(i)$, starting from the southwest corner.  Now we let $Z^{(v)}_{st}$ be the southwest $s\times t$ submatrix of $Z^{(v)}$ (which is $Z^{(v,A)}_{st}$ with the same substitutions).  Then we define the following:
\begin{definition}
The \emph{small set-theoretic type $B$ Kazhdan-Lusztig ideal} is the ideal $\tilde{I}_{v,w}$ of $\mathbb{C}[{\bf z}^{(v)}]$ generated by all $r_{st}^w+1$ minors of $Z_{st}^{(v)}$.
\end{definition}
Now we have the following isomorphism:
$$\mathbb{C}[z^{(v,A)}]/I_{v,w}\cong\mathbb{C}[{\bf z}^{(v)}]/\tilde{I}_{v,w}.$$

Furthermore, the following is not too difficult, given that $B$ and $B_-$ can in fact be identified with the subsets of upper and lower triangular matrices in $SO_{2n+1}$.
\begin{exercise}
Show that $\mathcal{N}_{v,w}$ is set theoretically cut out by $I'_{v,w}$.
\end{exercise}

\begin{exercise}
Give examples to show $I'_{v,w}$ and $\tilde{I}_{v,w}$ are not always radical ideals. 
\end{exercise} 

\begin{problem}
Find a set of generators for $\sqrt{I'_{v,w}}$, or a set of generators for $\sqrt{\tilde{I}_{v,w}}$.
\end{problem}

A.~Knutson's \cite{Knutson:patches} describes Bott-Samelson coordinates for which, under a given term order, explicitly describes the initial ideal, but does not give the Gr\"obner basis itself.  

\subsubsection{Type $C_n$} Now the group is $G=Sp_{2n}$.  We pick a nondegenerate antisymmetric bilinear form $Q$.
\begin{definition}
$G=Sp_{2n}\subset SL_{2n}$ is the group of linear transformations $M$ satisfying $Q(Mv, Mw)=Q(v,w)$ for all $v,w\in\mathbb{C}^{2n}$.
\end{definition}

The standard choice is given by
$$Q(e_i,e_j)=\begin{cases} 1 & i<j, i+j=2n+1 \\ -1 & i>j, i+j=2n+1 \\ 0 & i+j\neq 2n+1\end{cases}.$$  Define a matrix $J=[j_{ab}]$ by $j_{ab}=1$ if $a\leq n$ and $a+b=2n+1$, $j_{ab}=-1$ $a>n$ and $a+b=2n+1$, and $j_{ab}=0$ if $a+b\neq 2n+1$.  (Pictorially, $J$ has $1$'s on the top half of the main {\it antidiagonal}, $-1$'s on the bottom half of the main antidiagonal, and $0$'s elsewhere.  As in type B, $M\in Sp_{2n}$ if and only if $M^TJM-J=0$.

The Weyl group is the permutation group
$$W_{C_n}=\{w\in S_{2n}\mid w_0ww_0=w\},$$
where $w_0$ is the permutation $w_0=(2n)(2n)\cdots1$ of maximal length in
$S_{2n}$.  Equivalently, $w\in W_{C_n}$ if $w(i)+w(2n+1-i)=2n+1$ for all $i$, $1\leq i\leq n$.  
\begin{definition}
The \emph{length} for $w\in W_{C_n}$ is
\begin{multline}\nonumber
\ell_C(w):=\frac{\#\{1\leq i<j\leq 2n: i+j\neq 2n+1, w(i)>w(j)\}}{2}\\  \nonumber
+\#\{1\leq i\leq n: w(i)<w(2n+1-i)\}.
\end{multline}
\end{definition}
We will use the notation $\ell_A(w)$ for the length of $w$ considered as a permutation in $S_{2n}$.

Since Lemma~\ref{lemma:KL}  holds in general, $B$ and $B_-$ are still the subgroups of upper and lower triangular matrices, and $W_{C_n}$ is an explicit set of permutation matrices, we again give coordinates for opposite Schubert cells as before.

Given $v\in W_{C_n}$, we identify the opposite Schubert cell $\Omega^\circ_v$ with a subset of ${\sf Mat}_{2n\times 2n}$.  We define $\Omega_v^{A\circ}$ (formerly called $\Omega_v^\circ$ in Section~\ref{sec:opschubJan28}) and ${\bf z}^{(v,A)}\subseteq {\bf z}$ as before, and again let $K$ be the ideal generated by the $(2n)^2$ entries of $(Z^{(v,A)})^TJ(Z_A^{(v,A)})-J$.

Now let $Z_{st}^{(v,A)}$ (formerly $Z_{st}^{(v)}$) be the southwest $s\times t$ submatrix of $Z^{(v,A)}$.  We can define $I'_{v,w}$ as the ideal of ${\mathbb C}[z^{(v,A)}]$ generated by all $r_{st}^w+1$ minors of $Z_{st}^{(v,A)}$ where $1\leq s,t\leq n$ and
$r_{st}^w$ is defined in Definition~\ref{Woosrank}.  Now we have the following.
\begin{definition}
The \emph{large set-theoretic type $C$ Kazhdan-Lusztig ideal} is
$I_{v,w}=I'_{v,w}+K$.
\end{definition}

Exercise~\ref{exer:Feb1cde} also holds in this situation (though the polynomials $f_{ij}$ are different).  Hence, we let ${\bf z}^{(v)}\subseteq {\bf z}^{(v,A)}$ to be the set of unspecialized variables $z_{ij}$ with $j>v^{-1}(i)$.  Then we let
$Z^{(v)}$ be the matrix constructed from $Z^{(v,A)}$ by recursively substituting $f_{ij}-z_{ij}$ (or $f_{ij}/2 - z_{ij}$ if $i=v(j)$) for $z_{ij}$ whenever $j\leq v^{-1}(i)$, starting from the southwest corner.  Finally we let $Z^{(v)}_{st}$ be the southwest $s\times t$ submatrix of $Z^{(v)}$ (which is $Z^{(v,A)}_{st}$ with the same substitutions).  Then we define the following:
\begin{definition}\label{def:smallC}
The \emph{small set-theoretic type $C$ Kazhdan-Lusztig ideal} is the ideal $\tilde{I}_{v,w}$ of $\mathbb{C}[{\bf z}^{(v)}]$ generated by all $r_{st}^w+1$ minors of $Z_{st}^{(v)}$.
\end{definition}
Now we have the following isomorphism:
$$\mathbb{C}[z^{(v,A)}]/I_{v,w}\cong\mathbb{C}[{\bf z}^{(v)}]/\tilde{I}_{v,w}.$$

For type $C_n$ it follows from~\cite[Prop. 6.1.1.2]{SMT} that the Kazhdan--Lusztig ideal is indeed radical.  (See also~\cite[Prop. 4.12]{EFRW}.)  In the case where $v$ is $123$-avoiding, it is shown in~\cite{EFRW} that $Z^{(v)}$ is a symmetric matrix (in fact a generic symmetric matrix with certain entries set to 0) after certain rows and columns consisting of only $1$'s and $0$'s are deleted.  Furthermore, they show the following.

\begin{theorem}[\cite{EFRW}]
In the case $v$ is $123$-avoiding, the defining minors in Definition~\ref{def:smallC} form a Gr\"obner basis for $\tilde{I}_{v,w}$ under an appropriate (specified) term order.
\end{theorem}

One obtains as a consequence a combinatorial commutative algebra proof of the analogue of \cite[Theorem~4.5]{WY:Grobner} in this case (see the comments after Theorem~\ref{thm:specializationa}).  Just as Theorem~\ref{thm:specializationa} is related to formulas for (double) Schubert and (double) Grothendieck polynomials for $GL_n$, it is related to analogous polynomials for all the classical groups \cite{Ikeda.Mihalcea.Naruse, Kirillov.Naruse}.

\subsubsection{Type $D_n$}  The group here is $G=SO_{2n}$.  The standard choice of nondegenerate symmetric form is given by
$$Q(e_i,e_j)=\begin{cases} 1& i+j=2n+1 \\ 0 & i+j\neq 2n+1\end{cases}.$$  The Weyl group is
$$W_{D_n}=\{w\in S_{2n}\mid w_0ww_0=w, \#(\{w(1),\ldots,w(n)\}\cap\{1,\ldots,n\})\equiv 0\pmod{2}\}.$$

One can work as in types $B_n$ and $C_n$ to obtain a Kazhdan--Lusztig ideal, but in this case the naive choice is neither radical nor even set-theoretically correct!  In part, this is because Bruhat order on $W_{D_n}$ is not the restriction of Bruhat order on permutations. It seems that the kind of results and problems we have discussed for other types  are farther off into the horizon in
type $D_n$.

\subsection{Billey--Postnikov pattern avoidance}
S.~Billey--A.~Postnikov \cite{Billey.Postnikov} define a notion of pattern avoidance based on root subsystems which is now commonly called \emph{Billey--Postnikov avoidance}. Their definition is in terms of \emph{crystallographic root systems}. Rather than defining root systems here, we instead summarize what their notion says for the classical 
groups only.  Most of the details translating the general definition in terms of root systems to the concrete definitions of pattern embeddings below can be found in an unpublished research report by K.~Haenni~\cite{Haenni}. The details for type $A$ are also in the original paper of  Billey and Postnikov~\cite{Billey.Postnikov}, and details in type $B$ can be found in \cite{Woo:hultman}.

Below, it is important to consider the group the element sits in, not just the element as a permutation.  For example, a permutation $w\in W_{C_n}$  is distinct from the same permutation $w$ considered as an element of $S_n=W_{A_{n-1}}$.  Thus, in the definition below, we write a Weyl group element as $(w, W)$, where $W$ indicates the Weyl group we are considering $w$ to belong to.  For simplicity, we call our Weyl groups $A_{n-1}$, $B_n$, $C_n$, and $D_n$ in the definition.

\begin{definition}
Given a classical groups $V,W$ and elements $v\in V$ and $w\in W$, we say {\em $(v,V)$ (Billey--Postnikov) embeds in $(w,W)$} if $m\leq n$ and there exist indices $1\leq \phi_1<\phi_2<\ldots < \phi_m$ such that any of the following hold:
\begin{enumerate}
\item $V=A_{m-1}$, $W=A_{n-1}$, $\phi_m\leq n$ and $w(\phi_1),\ldots,w(\phi_m)$ are in the same relative order as $v(1),\ldots,v(m)$.\footnote{This is the same as Definition~\ref{def:basicpattern}.} 
\item $V=A_{m-1}$, $W=A_{n-1}$, $\phi_m\leq n$ and $w(\phi_1),\ldots,w(\phi_m)$ such that $w(\phi_1),\ldots,w(\phi_m)$ are in the {\em reverse} relative order as  $v(1),\ldots,v(m)$.\footnote{This is equivalent to $w_0vw_0$ embedding in $w$ according to Definition~\ref{def:basicpattern}.}
\item $V=A_{m-1}$, $W=B_n$, $\phi_m\leq 2n+1$ and $\phi_i+\phi_j\neq 2n+2$ for any $i, j$, where $w(\phi_1),\ldots,w(\phi_m)$ are in the same relative order as $v(1),\ldots,v(m)$.  (In particular, since $\phi_i+\phi_j\neq 2n+2$, we cannot have $\phi_i=n+1$ for any $i$.)
\item $V=A_{m-1}$, $W=C_n$, $\phi_m\leq 2n$ and $\phi_i+\phi_j\neq 2n+1$ for any $i, j$, where $w(\phi_1),\ldots,w(\phi_m)$ are in the same relative order as $v(1),\ldots,v(m)$.
\item \label{AtoD} $V=A_{m-1}$, $W=D_n$, $\phi_m\leq 2n$ and with $\phi_i+\phi_j\neq 2n+1$ for any $i, j$, where
 $w(\phi_1),\ldots,w(\phi_m)$ are in the same relative order as $v(1),\ldots,v(m)$.
\item $V=B_m$, $W=B_n$, $\phi_m\leq n$ and $w(\phi_1),\ldots,w(\phi_m),w(n+1),w(2n+2-\phi_m),\ldots,w(2n+2-\phi_1)$ are in the same relative order as $v(1),\ldots,v(2m+1)$.
\item $V=C_m$, $W=C_n$, $\phi_m\leq n$ and $w(\phi_1),\ldots,w(\phi_m),w(2n+1-\phi_m),\ldots,w(2n+1-\phi_1)$ are in the same relative order as $v(1),\ldots,v(2m)$.
\item \label{DtoD} $V=D_m$, $W=D_n$, $\phi_m\leq n$ and $w(\phi_1),\ldots,w(\phi_m),w(2n+1-\phi_m),\ldots,w(2n+1-\phi_1)$ are in the same relative order as $v(1),\ldots,v(2m)$, except that we allow either or both
\begin{enumerate}
\item $w(\phi_m)$ and $w(2n+1-\phi_m)$ to be in a different order than $v(m)$ and $v(m+1)$, or
\item $w(\phi_a)$ and $w(2n+1-\phi_a)$ to be in a different order than $v(a)$ and $v(2m+1-a)$, where $a$ is whichever of $v^{-1}(m)$ and $v^{-1}(m+1)$ that is less than or equal to $m$.\footnote{For the reader trying to derive this definition from the original definition of Billey and Postnikov, note that allowing these two cases actually conflates two issues, the Dynkin diagram automorphism of $D_n$, and the fact that there is no inversion in $D_n$ between the middle entries.  Billey--Postnikov avoidance comes in two versions, left and right, and how one accounts for the ``missing'' inversion in $D_n$ differs between these two versions, but the Dynkin diagram automorphism saves us from having to figure out which is which.}
\end{enumerate}
\item There is a Weyl group isomorphism between $\phi: W_{A_3}=S_4\rightarrow W_{D_3}$, so when considering embeddings of $(v,A_3)$ to $(w, D_n)$, one must consider both embeddings of $(v,A_3)$ according to (\ref{AtoD}) and embeddings of $(\phi(v),D_3)$ according to (\ref{DtoD}).
\item There are additional Weyl group automorphisms $\phi_1,\phi_2: W_{D_4}\rightarrow W_{D_4}$, so when considering embeddings of $(v,D_4)$ to $(w, D_n)$, one must consider embeddings of $v$, $\phi_1(v)$, and $\phi_2(v)$ according to (\ref{DtoD}).
\end{enumerate}
\end{definition}

Billey--Postnikov avoidance has appeared recently in some purely combinatorial contexts, for example in work of C.~Gaetz--Y.~Gao~\cite{Gaetz.Gao.separable} and of the first author~\cite{Woo:hultman}.

\subsection{Interval pattern avoidance}
Using Billey--Postnikov pattern avoidance, the first author \cite{Woo:general} has extended the results on interval pattern avoidance to arbitrary Lie type.  The proof uses the pattern map of Billey--Braden~\cite{Billey.Braden}
Let $[u,v]$ be a Bruhat interval in some Weyl group $V$ and $[x,w]$ a Bruhat interval in $W$.

\begin{definition}
$[u,v]$ \emph{interval pattern embeds in} $[x,w]$ if there is a
common embedding $\Phi=(\phi_1,\ldots,\phi_m)$ of $u$ into $x$ and $v$
into $w$, where the entries of $x$ and $w$ outside of $\Phi$ agree,
and, furthermore, $\ell(v)-\ell(u)=\ell(w)-\ell(x)$.  (In addition, if $V=D_m$ and $W=D_n$, the ways in which the relative orders of $v$ and its embedding in $w$ fail to agree must match the ways in which the relative orders of $u$ and its embedding in $x$ fail to agree.)
\end{definition}

The generalization of Exercise~\ref{exer:determined} holds and hence one can say the following:

\begin{definition}
$[u,v]$ \emph{interval pattern embeds} in $w$ if $[u,v]$ interval pattern embeds in $[\Phi(u),w]$.
\end{definition}

Definitions 6.13 and 6.14 can be made verbatim (substituting the set of all $(w,W)$ where $W$ is a (classical) Weyl group and $w\in W$ for $\mathfrak{S}$), and the analogue of Theorem 6.16, and hence Corollary 6.18, is proved in 
\cite{Woo:general}.

\subsection{Singularity classification problems}
The question of classifying smooth Schubert varieties using pattern avoidance (as is done in 
Theorem~\ref{thm:3412-4231} for type $A$) is solved in S.~Billey--A.~Postnikov's \cite{Billey.Postnikov} (most of the work amounts to restating earlier work of Billey \cite{B98} in terms of Billey--Postnikov avoidance). S.~Kumar \cite{Kumar:nil} has given a general type algebraic characterization of which points are
singular in terms of the \emph{nil-Hecke ring}.
 
\begin{problem}
\label{prob:BCDsing}
Determine a combinatorial description of the singular locus of each Schubert variety in $SO_{2n+1}/B, Sp_{2n}/B$, and 
$SO_{2n}/B$.
\end{problem}

One defines the \emph{Hecke algebra} for other types by replacing the role of the symmetric group $S_n$
in Definition~\ref{def:Heckedef} with the Weyl group $W$ associated to $G$. Similarly one defines Kazhdan-Lusztig
polynomials $P_{v,w}(q)$ for Weyl group elements $v,w\in W$ satisfying $v\leq w$ in Bruhat order for $W$.  By definition, $X_w$ is rationally smooth at $e_v$ if $P_{v,w}(1)=1$. A theorem of D.~Peterson is that smoothness
and rational smoothness agree in types $ADE$. It is known that in types $B_n$ and $C_n$, the rational singular locus and the singular locus
differ.  While the singular loci differ between types $B_n$ and $C_n$, their rational singular loci agree, since the Kazhdan--Lusztig polynomial only depends on the Coxeter group and not the lengths of the roots. See \cite{BL00}.

\begin{problem}
\label{prob:BCrsing}
Determine a combinatorial description of the  rational singular locus of each Schubert variety in $SO_{2n+1}/B$ (or equivalently) $Sp_{2n}/B$.
\end{problem}

\begin{hardexercise}  (cf.~Exercise~\ref{exer:isolatedsing})
Can the (rational) singular locus of $X_w$ be $X_{id}$? That is, can $X_w$ have an isolated singularity?
\end{hardexercise}

Even the following special cases are open in general:

\begin{problem}\label{prob:BCDabc}
Solve Problems~\ref{prob:BCDsing} and~\ref{prob:BCrsing} for the case where $w$ 
has only one ascent\footnote{These are the
maximal length coset representatives for a maximal standard parabolic subgroup of $W$.} (where, in $W_{B_n}$, ascents at $i$ and $2n+2-i$ (for $W_{B_n}$) count only once, and where in $W_{C_n}$ and $W_{D_n}$, ascents at $i$ and $2n+1-i$ also count only once).
\end{problem}

For Grassmannians, Problem~\ref{prob:BCDabc} is implicitly solved by A.~Zelevinsky \cite{Z83}. For 
minuscule parabolic see V.~Lakshmibai-J.~Weyman \cite{LW90} and M.~Brion-P.~Polo \cite{Brion.Polo}.

Insofar as the more general problems (P1) and (P2) are concerned, substantially less is known about the
measures discussed in this paper outside of type $A$. In most cases, there are not even conjectures. For example:
 
\begin{problem}\label{problem:Gorfactbbb}
Determine which Schubert varieties in $SO_{2n+1}/B, Sp_{2n}/B$, and 
$SO_{2n}/B$ are Gorenstein and/or factorial.
\end{problem}

The argument used in \cite{WY:Goren, Butler} to characterize Gorenstein or factorial Schubert varieties in type $A$ begins by reducing the problem to finding
solutions for a system of linear equations.  This part of the argument extends to all types, but a specific
combinatorial conjecture has eluded us.  While in type $A$, all solutions to these linear systems are integral, there
are non-integral solutions in type $C$, so one might separately characterize the Schubert varieties that are $\mathbb{Q}$-Gorenstein or $\mathbb{Q}$-factorial.
Now, one would also like to describe the non-Gorenstein locus and give a uniform answer for all Lie types. Since interval pattern avoidance was useful to give answers to 
Problem~\ref{problem:Gorfactbbb} in type $A$, one pursues similar answers in the other types using the generalized 
notions of this section.

Since the Kazhdan-Lusztig ideals are hard to get a handle on in types $B,C,D$, the tangent cones
are even more difficult to handle. In Section~\ref{sec:HSmult} we stated a principle that, in the good cases where the Kazhdan-Lusztig
ideal is standard homogeneous it already defines the tangent cone. As noted to us by A.~Knutson, 
this is true of any (co)minuscule $G/P$. This includes ordinary Grassmannians, as well as all maximal orthogonal
and Lagrangian Grassmannians. Hence in all such cases, the Hilbert-Samuel multiplicity at a $T$-fixed
point can be determined by the $K$-polynomial associated to that point; see work of W.~Graham-V.~Kreiman
\cite{Graham.Kreiman:TAMS, Graham.Kreiman:Fourier} and the references therein. That said, by analogy
with \cite{RRRSW} it would be interesting to solve:

\begin{problem}
Determine an explicit, root-system uniform combinatorial rule for the regularity of a $T$-fixed point in a 
(co)minuscule Schubert variety  (generalizing the rule of \cite{RRRSW}).
\end{problem}

In the classical types, the Hilbert-Samuel multiplicities~\cite{Jeon.mult} and Kazhdan--Lusztig polynomials~\cite{Jeon.KL}
for covexillary Schubert varieties have been determined, giving analogues of the results of
\cite{LiYong1,LiYong2} and \cite{Las-vex} respectively.  Here, covexillary has a definition in terms of an analogue of the essential set~\cite{AndFul.pfaff}, but it is equivalent to $w$ avoiding $3412$
as a permutation~\cite{AndFul.vex}.  (This can be rephrased in terms of Billey--Postnikov avoidance, but at the cost of requiring more elements.)

One can go beyond the classical types and study similar questions for exceptional types and even for infinite dimensional Kac-Moody groups.  There has been some significant work particularly in affine type $A$.  E.~Richmond--W.~Slofstra~\cite{Richmond.Slofstra.affine} have characterized the smooth Schubert varieties in affine type $A$, following a characterization of the rationally smooth Schubert varieties by S.~Billey--A.~Crites~\cite{Billey.Crites}.  B.~Elek--D.~Huang
\cite{Elek.Huang} have generalized Theorem~\ref{thm:Grobnerthm} to the affine type A flag variety.

\section{Remarks about other varieties}\label{sec:8}

\begin{center}
\fbox{\begin{minipage}{30em}
One can also use analogues of patches to study other subvarieties of the flag manifold.
   \end{minipage}}
\end{center}

\subsection{Richardson varieties}
In what follows, one may assume $G=GL_n$, however the results hold for (partial) flag varieties associated to any complex semisimple Lie group $G$ (see Section~\ref{sec:7}).

\begin{definition}
The \emph{opposite Schubert variety} is $X^w:=\overline{\Omega^\circ_w}$.
\end{definition}
\begin{definition}
The \emph{Richardson variety}
$X_v^w$ is $X_v\cap X^w$.
\end{definition}
 It is nonempty provided $w\leq v$, and in that case it is an irreducible variety of dimension $\ell(v)-\ell(w)$. It is known to be normal and Cohen-Macaulay. When
$w=w_0$, $X_v^w=X_v$. Thus Richardson varieties are generalizations of Schubert varieties. Therefore
most questions about singularities of Schubert varieties can be asked of the Richardson varieties.\footnote{The expansion of the cohomology class of the Richardson into the Schubert basis, i.e.,
$[X_{v}^w]=\sum_{u\in W} C_{v,u}^w[X_u]$ is precisely the topic of \emph{Schubert calculus}. The coefficients
$C_{v,u}^w$ are nonnegative integers and it is an open problem for most $G/P$ to give a combinatorial counting rule for them.}

The following result of A.~Knutson and the authors \cite{KWY} shows that many of the problems in fact reduce to
the Schubert case. That is:

\begin{center}
\fbox{\begin{minipage}{30em}
The patch of  $X_v^w$ at a point $p$ is the Cartesian product of Kazhdan-Lusztig
varieties for $X_v$ and $X_w$ at $p$.
   \end{minipage}}
\end{center}

The uniform proof of this result is a generalization of Lemma~\ref{lemma:KL}.
The result has a number of immediate consequences. For example, it proves that Richardson
varieties are normal, Cohen-Macaulay, and have rational singularities since these properties are known of
the Schubert varieties. It also implies 

\begin{corollary}
\label{cor:Pissmooth}
${\rm Singlocus}(X_{w}^v)=({\rm Singlocus}(X_w)\cap X^v)\cup (X_w\cap {\rm Singlocus}(X^v))$.
\end{corollary}

One also sees that the Hilbert-Samuel multiplicities for Richardson varieties factor:

    \begin{corollary}
     \label{Cor:Hfactors}
     Let $xB\in X_{w}^v$.\footnote{$xB$ need not be a $T$-fixed point.} Then 
     ${\rm mult}(xB, X_w^v)={\rm mult}(xB,X_w)\cdot {\rm mult}(xB,X^v)$.
     \end{corollary}

However, problems about Richardson varieties remain. For example:

\begin{problem}
\label{prob:richgor}
Determine which Richardson varieties in $G/B$ are Gorenstein?
\end{problem}    

Problem~\ref{prob:richgor} is open for $GL_n/B$, though a solution would follow from a solution to Conjecture~\ref{conj:nonGorlocus}. For Grassmannians (and minuscule $G/P$'s), it follows from \cite{Perrin} combined with
\cite{KWY}.

Rather than taking intersections of two Schubert varieties with respect to opposite flags, one can do the same for a collection of Schubert varieties with respect to a ``cyclic permutation'' of a reference flag. This is the \emph{positroid variety}. Recent work of S.~Billey-J.~Weaver \cite{Billey.Weaver} gives
a pattern avoidance criterion for smoothness of these varieties in the Grassmannian. More finely,
one can take the common refinement of $n!$ Bruhat decompositions with respect to all permutations of a reference flag. This is the \emph{matroid stratification} \cite{GGMS}. However, Mn\"ev's Universality theorem implies that for the 
Grassmannian, these strata can contain essentially \emph{any} singularity \cite{Mnev}.

\subsection{Peterson and Hessenberg varieties}

\begin{definition} The \emph{Peterson variety} is 
\[{\sf Pet}_n:=\{F_{\bullet}\in {\rm Flags}({\mathbb C}^n): N\cdot F_n\subset F_{i+1}\},\]
where $N$ is the regular, nilpotent $n\times n$ matrix consisting of a single Jordan block.
\end{definition}

D.~Peterson introduced ${\sf Pet}_n$ in connection to his study of quantum cohomology of ${\rm Flags}({\mathbb C}^n)$. The content of \cite{IY} is the following:

\begin{center}
\fbox{\begin{minipage}{30em}
Using patches one proves a combinatorial description of the singular locus of ${\sf Pet}_n$, 
and that ${\sf Pet}_n$ is a local complete intersection.
   \end{minipage}}
\end{center}

Even for Peterson varieties, singularity problems remain. For example:
\begin{problem}[{\cite[Section~6]{IY}}] Determine a combinatorial formula for the Hilbert-Samuel multliplicities of ${\sf Pet}_n$.
\end{problem}

We refer to \emph{ibid} for further details and references.

Peterson varieties are special cases of \emph{Hessenberg varieties}. These varieties come in various
generalities, we follow the definition of
F.~de Mari--C.~Procesi--M.~A.~Shayman \cite{Hessdef} for ${\rm Flags}({\mathbb C}^n)$. Let $M$ be any linear
operator on ${\mathbb C}^n$. Fix a non-decreasing function $h:[n]\to [n]$ such that $h(i)\geq i$ for each $i\in [n]$; this is the \emph{Hessenberg function}.

\begin{definition}
The \emph{Hessenberg variety} is
\[{\rm Hess}(M,h)=\{F_{\bullet}\in GL_n/B: M\cdot F_i\subseteq F_{h(i)}, \ \forall i\in [n]\}.\]
\end{definition} 

\begin{example}
If $M$ is the identity matrix and $h(i)=i$ then ${\rm Hess}(M,h)={\rm Flags}({\mathbb C}^n)$. \qed
\end{example}

\begin{example}
If $M=N$ is regular nilpotent and $h(i)=i+1$, ${\rm Hess}(M,h)={\rm Pet}_n$.\qed
\end{example}

\begin{example}
If $N$ is a nilpotent matrix and $h(i)=i$ then ${\rm Hess}(M,h)$ is a \emph{Springer fiber}, an object of
significance in geometric representation theory of the symmetric group.
\end{example}

J.~Tymoczko \cite{Tymoczko} proves that the Hessenberg variety is ``paved by affines'',
a consequence of which is a combinatorial formula for the topological Betti numbers of ${\rm Hess}(M,h)$. 
The question of the singularity structure of ${\rm Hess}(M,h)$ was raised in \cite[Section~7]{IY}. 
Using patch ideals, some initial exploration was done in \emph{ibid.}; see, e.g., later work of 
H.~Abe--L.~Dedieu--F.~Galetto--M.~Harada \cite{Harada} and 
E.~Insko--M.~Precup
\cite{IP}.  L.~Escobar-M.~Precup-J.~Shareshian \cite{EPS} classify Hessenberg varieties that
are Schubert varieties. See that paper for more discussion/references on Hessenberg varieties.

\subsection{Spherical symmetric orbit closures}\label{sec:spherical}
\begin{definition}
A subgroup $K$ of $G$ is \emph{symmetric} if $K=G^\theta$ is the fixed point subgroup of an involutive
automorphism $\theta$ of $G$. In addition, such a $K$ is \emph{spherical} if the action of $K$ on 
$G/B$ by left-multiplication has finitely many orbits. Such $(G,K)$ are called \emph{spherical symmetric
pairs}.
\end{definition}

\begin{example}
For $G=GL_n$ there are three such spherical symmetric
subgroups $K$, namely, $K=O_n$ (the orthogonal group), $K=Sp_{n}$ (the symplectic group, assuming
$n$ is even), and $K=GL_p\times GL_q$ (invertible $p+q=n$ block matrices). 
\end{example}

We are interested in the
singularities of the (finitely many) $K$-orbit closures. 
Once again, many of the problems that we considered for Schubert varieties are valid for $K$-orbit closures.
For instance, the following problem is open:
\begin{problem}
Determine the singular locus of the orbit closures for the three spherical symmetric pairs
$(GL_n,O_n), (GL_{2n}, Sp_{2n}), (GL_{n}, GL_p\times GL_q)$. 
\end{problem}

The content of \cite{WWY} is: 
\begin{center}
\fbox{\begin{minipage}{30em}
It is equivalent to study $B$-orbits on $G/K$. On the latter, 
there is an analogue of the Kazhdan-Lusztig varieties for the closures in the
case of the pair $(GL_{n}, GL_p\times GL_q)$. The analogue consists of the \emph{Mars-Springer varieties}. In addition,
one finds an analogue of interval pattern avoidance in this context. 
   \end{minipage}}
\end{center}
Very few of the results available for
Kazhdan-Lusztig varieties are known. We do not know a Gr\"obner basis for the
Mars-Springer ideals except in some cases for $K=Sp_{n}$~\cite{Marberg.Pawlowski}.

\subsection{Quiver loci}
Our last example does not live in a flag variety, but is nonetheless closely related to the study of Schubert varieties.

\begin{definition}
A \emph{quiver} $Q$ is a directed graph. We say that $Q$ is \emph{equioriented of type $A_n$} if $Q$
is a directed path $\bullet \rightarrow \bullet \rightarrow \bullet \rightarrow \cdots
\rightarrow \bullet$ with $n$ vertices.
\end{definition}

\begin{definition}
A \emph{representation} of a quiver $Q$ with dimension vector $d=(d_1,d_2,\ldots,d_n)$ is
of the form $V_1\stackrel{M_1}{\longrightarrow} V_2 \stackrel{M_2}{\longrightarrow} \cdots 
\stackrel{M_{n-1}}{\longrightarrow} V_n$ where $V_i$ is a vector space over ${\mathbb C}$ of dimension $d_i$
and $M_i:V_{i}\to V_{i+1}$ is a linear transformation.
\end{definition}

\begin{definition}
The \emph{representation space} ${\sf Rep}_Q(d)$ of a dimension vector $d=(d_1,d_2,\ldots,d_n)$ is
\[{\sf Rep}_Q(d):=\prod_{a\in {\sf arc}(Q)} {\sf Mat}_{d(ha),d(ta)}({\mathbb C})\]
where ${\sf arc}(Q)$ is the set of arcs $a=ha\to ta$ of $Q$.
\end{definition} 

\begin{definition}
The \emph{base change group} $GL(d)$ of a dimension vector $d=(d_1,d_2,\ldots,d_n)$ is
\[GL(d):=\prod_{v\in {\sf vertices}(Q)} GL_{d(v)}({\mathbb C}).\]
\end{definition}

$GL(d)$ acts on ${\sf Rep}_Q(d)$ as follows: Suppose $(M_1,\ldots,M_n)\in {\sf Rep}_Q(d)$ and $g=(g_1,\ldots,g_n)\in GL(d)$ then
$g\cdot (M_1,\ldots,M_n)=(g_{ha}V_a g_{ta}^{-1})_{a\in {\sf arc}(Q)}$.

\begin{definition}
A \emph{quiver loci} is one of the (finitely many)
 $GL(d)$-orbit closures in ${\sf Rep}_Q(d)$.
 \end{definition} 
There is a connection to Schubert varieties.
A.~Zelevinsky \cite{Zelev} showed that the quiver loci for equioriented type $A_n$ quivers are set-theoretically in bijection with certain open subsets of a Schubert variety. Lakshmibai-Magyar \cite{Lak.Mag} proved this map is a scheme-theoretic isomorphism. 

R.~Kinser-J.~Rajchgot \cite{Kinser.Rajchgot1} prove that:

\begin{center}
\fbox{\begin{minipage}{30em}
 For any orientation of a type $A_n$ quiver,   the quiver loci
are isomorphic to a Kazhdan-Lusztig variety in a partial flag manifold, up to an explicit smooth factor.
   \end{minipage}}
\end{center}

In R.~Kinser-A.~Knutson-J.~Rajchgot \cite{KKR} by using the above relationship with
Kazhdan-Lusztig varieties together with the Hilbert series theorem of \cite{WY:Grobner} to give a formula for the Hilbert series of quiver loci. It would therefore be interesting to develop in detail the singularities of quiver loci by reduction to the Schubert variety case. 

One can ask similar questions for quivers of other Dynkin types; we point to, e.g., \cite{KR:D}
and the references therein. There, quiver loci are not related to Kazhdan-Lusztig varieties but an 
analogue for symmetric varieties $GL_{p+q}/GL_p\times GL_q$ (see Section~\ref{sec:spherical}).

\section{Hints, notes, and references for selected exercises}\label{sec:hints}

\noindent
\emph{Exercise~\ref{exer:24decomp}}:
In the $k=2,n=4$ case, the rows of the $1$'s give the ``$I$''.

\medskip
\noindent
\emph{Exercise~\ref{exer:Borels}} For (b), the answer is $gBg^{-1}$. There is a bijection between 
points in ${\rm Flags}({\mathbb C}^n)$ and their stabilizer ${\mathcal B}=\{gBg^{-1}:g\in GL_n\}$. All Borel subgroups (in the general sense) are $G$-conjugate. Hence ${\rm Flags}({\mathbb C}^n)$ may be identified with the set of all Borel subgroups of $GL_n$, not privileging one Borel over another in the description.

\medskip
\noindent
\emph{Exercise~\ref{exer:Bruhatexer}}: For a solution see
\cite[p.63--64]{Manivel}.

\medskip
\noindent
\emph{Exercise~\ref{exer:additionalexerSymmetric}}: If $S_n$ is described as the set of permutations of $[n]$ then
it is generated by the simple transpositions $s_i:=(i \ i+1)$. Thus the map that sends $\sigma_i\to s_i$ is surjective. 
One can write down a ``lexicographically smallest'' factorization $F$ of $w$ which has length $\ell(w)$. Parts (b) and (c) ask to show that
any reduced word can be ``moved'' to $F$ by the relations.

\medskip
\noindent
\emph{Exercise~\ref{exer:upandright}}: Multiplying by $B$ on the left is an upward row operation and doing so on the right
is a rightward column operation.

\medskip
\noindent
\emph{Exercise~\ref{exer:KLexer}}: If $x=3421$, (a) is saying
\[\left[\begin{matrix} 0 & 0 & 0 & 1 \\ 0 & 0 & 1 & 0\\ 1 & 0 & 0 & 0 \\ * & 1 & 0 & 0\end{matrix}\right]
\left[\begin{matrix} 1 & 0 & 0 & 0 \\ 0 & 1 & 0 & 0 \\ * & * & 1 & 0 \\ * & * & * & 1\end{matrix}\right]
=\left[\begin{matrix} * & * & * & 1 \\ * & * & 1 & 0 \\ 1 & 0 & 0 & 0 \\ * & 1 & 0 & 0\end{matrix}\right].\]

\medskip
\noindent
\emph{Exercise~\ref{exer:theeqnsJan30}}: The condition $\dim(F_j\cap E_i)\geq k$ is equivalent to 
$\dim(\pi_i(F_j))\leq n-k$ where $\pi_i$ is the projection onto all but the first $i$ coordinates.

\medskip
\noindent
\emph{Exercise~\ref{exer:inhomgex}:} $I_{1324,3412}$.

\medskip
\noindent
\emph{Exercise~\ref{exer:inverseiso}}: For the first part, let $\mathcal{O}_w$ be the orbit of $(E^{(w)}_\bullet, E_\bullet)\in G/B \times G/B$ under the diagonal action of $G$, and let $\overline{\mathcal{O}_w}$ be its closure.  Note that projection on to the first and second factors gives fiber bundles with fibers $X_{w^{-1}}$ and $X_w$ respectively.  Now consider an affine neighborhood of the point $(E^{(v)}_\bullet,E_\bullet)\in\overline{\mathcal{O}_w}$.  For the second part,
$X_w\cong X_{w^{-1}}$ need not be true. E.~Richmond--W.~Slofstra \cite{Richmond.Slofstra.iso} have 
classified Schubert varieties up to isomorphism.

\medskip
\noindent
\emph{Exercise~\ref{exer:mono}}:
By Exercises~\ref{exer:Jan28.23qs} and~\ref{exer:specialcase333} 
combined, $I_w$ is prime. Now use \cite[Proposition~1.2]{Mono}.

\medskip
\noindent
\emph{Exercise~\ref{exer:rotheess}:} (a) In the cograssmannian case, the essential set boxes all lie in the same
column. (b) follows from (a).

\medskip
\noindent
\emph{Exercise~\ref{lemma:embed_iso}}: Use Exercise~\ref{exer:Bruhatexer}(b).

\medskip
\noindent
\emph{Exercise~\ref{exer:Feb11uuu}:} One example is $w=413625$.

\medskip
\noindent
\emph{Exercise~\ref{exa:proof_example}:} See \cite[Theorem~4.2]{WY:governing}.

\medskip
\noindent
\emph{Exercise~\ref{exer:Lanini}:} This is \cite[Theorem~2.4]{Lanini}.

\medskip
\noindent
\emph{Exercise~\ref{runningisspecialcase}:} Think about $w=52143$.

\medskip
\noindent
\emph{Exercise~\ref{exer:Buchberger}}: See how Buchberger's
algorithm works. One step of which is to confirm that the ``$S$-pair'' $S\left(\left|\begin{matrix} x_{11} & x_{12} \\ x_{21} & x_{22}\end{matrix}\right|, \left|\begin{matrix} x_{11} & x_{13} \\ x_{21} & x_{23}\end{matrix}\right|\right)$
is in the ideal generated by the nine $2\times 2$ minors (which is true as it is equal to $-x_{21}\left|\begin{matrix} x_{12} & x_{13} \\ x_{22} & x_{23}\end{matrix}\right|$). The reader can check the same for
 $S\left(\left|\begin{matrix} x_{11} & x_{12} \\ x_{21} & x_{22}\end{matrix}\right|, \left|\begin{matrix} x_{12} & x_{13} \\ x_{22} & x_{23}\end{matrix}\right|\right)$ and perhaps guess what the $S$-pair test in Buchberger's algorithm is from these examples, if they do not already know it.

\medskip
\noindent
\emph{Exercise~\ref{exer:KMgrob}}: Use Exercise~\ref{exer:specialcase333}.

\medskip	
\noindent
\emph{Exercise~\ref{exer:Jan28.23qs}}: By Theorem~\ref{thm:Grobnerthm}, $I_{v,w}$ is radical. Now use
the fact that $X_w$ is irreducible (why is \emph{that} true?) and ${\mathcal N}_{v,w}$ is (essentially) an affine open neighborhood
of $X_w$.

\medskip
\noindent
\emph{Exercise~\ref{exer:coincidence}}: For example, 
\[\left[\begin{matrix} \cdot & \cdot & \cdot \\ 
+ & + & \cdot \\
+ & + & \cdot 
\end{matrix}\right]\leftrightarrow \tableau{1&1\\2 &2}
\text{\ \ and \ } \left[\begin{matrix} \cdot & \cdot & + \\ 
+ & \cdot & +\\
+ & \cdot & \cdot
\end{matrix}\right]\leftrightarrow \tableau{1&2\\2 &3}.\] 

\medskip
\noindent
\emph{Exercise~\ref{runningexfinale}:} For (c) see \cite[Section~5]{KMY} and specifically Theorem~5.8 of that paper.

\medskip
\noindent
\emph{Exercise~\ref{exer:computetwogroth}}:
${\mathfrak G}_{312}=\frac{(x_1-y_1)(x_1-y_2)}{y_1 y_2}$ and ${\mathfrak G}_{132}=-\frac{x_1 x_2 - y_1 y_2}{y_1 y_2}$.

\medskip
\noindent
\emph{Exercise~\ref{sec5final}:} ${\mathcal K}(R/I_{132,132};t_1,t_2,t_3)=\frac{(t_1-t_3)(t_1-t_2)}{t_2 t_3}$.

\medskip
\noindent
\emph{Exercise~\ref{macexernnn}:} See the {\tt Schubsingular} package available at the authors' websites.

\medskip
\noindent
\emph{Exercise~\ref{exer:singlocFeb24abc}:} (a) See \cite[Example~6.2]{WY:governing}. (b)
${\rm sing}(X_{523614})=X_{215634}\cup X_{321546}$. 

\medskip
\noindent
\emph{Exercise~\ref{exer:isolatedsing}}: No. One can argue this using Theorem~\ref{thm:singularlocus}.
Another argument uses ``parabolic moving''; see \cite[Section~5.1]{WY:Grobner}. 

\medskip
\noindent
\emph{Exercise~\ref{exer:minimal}:} (a) See \cite[Lemma~3.10]{Fulton:duke}. (b) and (c): $w=3142$.

\medskip
\noindent
\emph{Exercise~\ref{exer:usemanivelcortez}}: This requires knowing the Gorensteinness of the points in the maximal singular locus; see results of L.~Manivel \cite{manivel2} or A.~Cortez \cite{Cortez}.

\medskip
\noindent
\emph{Exercise~\ref{exer:Feb11ppp}:} Using Exercise~\ref{exer:getsworse}, $X_w$ is Gorenstein 
(respectively, lci) if and only if ${\mathcal N}_{id,w}$ is Gorenstein (respectively, lci). Now compare the patterns of Theorems~\ref{lcichar} and~\ref{thm:Gor_char}.

\medskip
\noindent
\emph{Exercise~\ref{exer:easyinverse}}:
Use Exercise~\ref{exer:inverseiso}. The assertion, for $v=id$, was conjectured by  D. Eliseev-A. Panov \cite{Eliseev.Panov} and given a proof
in \cite[Section~5.7]{Fuchs.Kirillov}.

\section*{Acknowledgements}
We thank the many mathematicians, including Allen Knutson, Ezra Miller, and William Fulton, who have shaped our view of this subject over the past two decades, and
whose thoughts are directly or indirectly represented here. With regards to this specific document, we
thank Casey Appleton, Justin Chen, Shiliang Gao, Abigail Price, Ryan Roach, Ada Stelzer, John Stembridge,  Avery St.~Dizier, and Zhuo Zhang for helpful communications.  We also thank the participants of the 2022 ICLUE summer program on representation theory and combinatorics at UIUC for providing a stimulating environment that also influenced the exposition. We are grateful to Erik Insko, Martha Precup, and Edward
Richmond for instigating the authors to write this chapter. Our exposition of some basic material
was influenced by Richard Borcherds' online lecture series on algebraic geometry. AW was partially
supported by a Simons Collaboration Grant. AY was partially supported by a Simons Collaboration Grant, an NSF RTG grant, and an appointment at the UIUC Center for Advanced Study.

\end{document}